\documentclass[reqno,11pt,a4paper]{amsart}
\usepackage{graphicx} % Required for inserting images

\usepackage[toc]{appendix}
\usepackage[T1]{fontenc}
\usepackage{enumerate} 
\usepackage{amsmath,amsfonts,amssymb,mathrsfs, bbm }
\usepackage[latin1]{inputenc}
\usepackage[english]{babel}
\usepackage{mathtools}
\usepackage{braket}
\usepackage{color}
\usepackage{hyperref}
\usepackage{xfrac, nicefrac}
\usepackage{csquotes}
\usepackage{todonotes}
\usepackage{esint}

\usepackage{orcidlink}

\usepackage[backend=biber]{biblatex}
\newtheorem{theorem}{Theorem}[section]
\newtheorem{lemma}[theorem]{Lemma}
\newtheorem{proposition}[theorem]{Proposition}

\newtheorem{definition}[theorem]{Definition}
\newtheorem{remark}[theorem]{Remark}
\newtheorem*{remark*}{Remark}
\newtheorem*{definition*}{Definition}

\numberwithin{equation}{section}

\usepackage{xcolor}

\usepackage{hyperref}

\newcommand{\R}{\mathbb{R}}
\newcommand{\N}{\mathbb{N}}

\newcommand{\B}{\mathcal{B}}
\newcommand{\T}{\mathcal{T}}

\newcommand{\data}{\textnormal{\texttt{data}}}
\newcommand{\datah}{\textnormal{\texttt{data}}_{\rm{h}}}

\newcommand{\mint}{\mathop{\int\hskip -1.05em -\, \!\!\!}\nolimits}

\def\Yint#1{\mathchoice
    {\YYint\displaystyle\textstyle{#1}}%
    {\YYint\textstyle\scriptstyle{#1}}%
    {\YYint\scriptstyle\scriptscriptstyle{#1}}%
    {\YYint\scriptscriptstyle\scriptscriptstyle{#1}}%
      \!\iint}
\def\YYint#1#2#3{{\setbox0=\hbox{$#1{#2#3}{\iint}$}
    \vcenter{\hbox{$#2#3$}}\kern-0.51\wd0}}
\def\longdash{{-}\mkern-5mu{-}} 
\def\miint{\Yint\longdash}

\def\A{\mathcal A}
\def\l{\lambda}
\def\L{\Lambda}

\def\wB{\widetilde{B}}
\def\e{\varepsilon}

\def\sm{\setminus}

\title[Regularity of quasilinear parabolic equations]{Schauder estimates for quasilinear parabolic equations with general Orlicz growth}

\begin{document}

\begin{abstract}
We develop a unified approach to Schauder estimates for scalar
parabolic equations
\[
    \partial_t u-\operatorname{div}\big(\A(x,t,Du)\big)=f(x,t)
\]
with general Orlicz growth  and coefficients that are H\"older continuous in the spatial variables and merely measurable in time. Combining intrinsic gradient estimates with a perturbative
argument, we establish quantitative local bounds and H\"older
estimates for the spatial gradient in the supercritical range
and, for locally bounded weak solutions, throughout the full
admissible growth range.

Our results extend the Schauder regularity theory for parabolic $p$-Laplace type equations to a broad class of possibly nonradial operators with Orlicz growth, while simplifying the regularity arguments and treating singular, degenerate, and mixed growth laws within a unified framework.
\end{abstract}


\author[Carlo Alberto Antonini]{Carlo Alberto Antonini \orcidlink{0000-0002-7663-1090}}  \address{Carlo Alberto Antonini \\ 
Dipartimento di Matematica ``Federico Enriques'', 
Universit\`a degli studi di Milano, 
Viale Cesare Saldini 50, 20133,
Milan,
Italy\\ ORCID ID: 0000-0002-7663-1090}
\email{\url{carlo.antonini@unimi.it}\\ \url{antonini@altamatematica.it}}

\maketitle

\section{Introduction}\label{sec:intro}

In this paper we study the local regularity of solutions to scalar
parabolic equations of the form
\begin{equation}\label{eq:1}
    \partial_t u-\operatorname{div}\big(\A(x,t,Du)\big)
    =f(x,t)\qquad\text{in }\Omega_T,
\end{equation}
where $\Omega_T\coloneqq\Omega\times(0,T)$,
$\Omega\subset\R^N$ is open, $N\geq2$, and $T>0$.

The growth and ellipticity assumptions on $\A$ are modeled on
the parabolic Orlicz-Laplace equation
\begin{equation*}
    \partial_t u-\operatorname{div}\left(
    \frac{B'(|Du|)}{|Du|}\,Du\right)=f,
\end{equation*}
where $B:[0,\infty)\to[0,\infty)$ is a Young function of the form
\begin{equation}\label{Bb}
    B(t)\coloneqq\int_0^t b(s)\,ds,
    \qquad b(t)\coloneqq t\,a(t).
\end{equation}
The particular choice $B(t)=t^p/p$
gives the evolutionary $p$-Laplace equation
\[
\partial_t u-\Delta_p u=f,\qquad \Delta_p u\coloneqq  \mathrm{div}\big(|Du|^{p-2} Du\big).
\]

We assume that
\begin{equation}\label{iasa}
\begin{split}
    &a\in C^1((0,\infty);\R_+),\qquad a(1)=1,\\
    &-1<i_a\leq\inf_{t>0}\frac{t\,a'(t)}{a(t)}
    \leq\sup_{t>0}\frac{t\,a'(t)}{a(t)}\leq s_a<\infty.
\end{split}
\end{equation}
Without loss of generality, we may suppose that $i_a<0<s_a$, and we also write
\begin{equation*}
    i_B\coloneqq i_a+2\quad\text{and}\quad s_B\coloneqq s_a+2.
\end{equation*}
In particular, $1<i_B<2<s_B<\infty$.

We assume that the vector field
$\A=\{\A^i\}_{i=1}^N:\Omega_T\times\R^N\to\R^N$
is a Carath\'eodory map such that
$\xi\mapsto\A(x,t,\xi)$ is of class
$C^1(\R^N\setminus\{0\})$ for almost every $(x,t)\in\Omega_T$.
We assume that there exist $\lambda_0,\Lambda_0>0$ such that
\begin{equation}\label{ass:A}
\begin{split}
    &\sum_{i,j=1}^N\left|
    \frac{\partial\A^i}{\partial\xi_j}(x,t,\xi)
    \right|\leq\Lambda_0\,a(|\xi|),\\
    &\sum_{i,j=1}^N
    \frac{\partial\A^i}{\partial\xi_j}(x,t,\xi)
    \eta_i\eta_j
    \geq\lambda_0\,a(|\xi|)|\eta|^2
\end{split}
\end{equation}
for almost every $(x,t)\in\Omega_T$, all
$\xi\in\R^N\setminus\{0\}$ and $\eta\in\R^N$.
Its dependence on the spatial variable satisfies, for some
$\alpha\in(0,1)$ and $\Lambda_{\rm h}>0$,
\begin{equation}\label{A:hold}
    |\A(x,t,\xi)-\A(y,t,\xi)|
    \leq\Lambda_{\rm h}\big(1+b(|\xi|)\big)|x-y|^\alpha
\end{equation}
for all $x,y\in\Omega$, $\xi\in\R^N$, and almost every
$t\in(0,T)$. Finally, we assume that
\begin{equation}\label{ass:f}
    f\in L^d(\Omega_T),\qquad d>N+2.
\end{equation}
Notation and definitions, including the parabolic Orlicz spaces and the
notion of local weak solution, are given in Section \ref{sec:prel}.
The structural constants $\data$ and $\datah$ are specified in \eqref{data}.

The main result of this paper is the following theorem, which
establishes Schauder estimates for \eqref{eq:1} under general
Orlicz growth.
\begin{theorem}\label{main:thm}
    Let
    $u$
    be a local weak solution to \eqref{eq:1} under
    \eqref{iasa}, \eqref{ass:A}, \eqref{A:hold} and \eqref{ass:f}.
    Assume that either
    \[
        i_B>\frac{2N}{N+2}
        \qquad\text{or}\qquad
        u\in L^\infty_{\rm loc}(\Omega_T).
    \]
    Then
    \[
        u\in C^{1,1/2}_{\rm loc}(\Omega_T),
        \qquad
        Du\in C^{\beta_*,\beta_*/2}_{\rm loc}(\Omega_T)
    \]
    for some $\beta_*=\beta_*(\data,\alpha,d)\in(0,1)$.
    More precisely, the following estimates hold.

    \begin{itemize}
    \item[(i)] If $i_B>2N/(N+2)$, then, for every
    $Q_{R,S}=Q_{R,S}(x_0,t_0)\Subset\Omega_T$ and $c\in\R$,
    \begin{equation}\label{main:boundedness}
    \begin{split}
        \|u-c\|_{L^\infty(Q_{R/2,S/4})}
        \leq C_{R,S}\left[
        1+\miint_{Q_{R,S}}
        \big\{|u-c|^2+B(|u-c|)\big\}\,dxdt
        \right]^{\frac{i_B}{(N+2)i_B-2N}},
    \end{split}
    \end{equation}
    where $C_{R,S}$ depends only on $\datah$, $R,S$,
    and $\|f\|_{L^d(Q_{R,S})}$.

    \item[(ii)] For every $Q_{2\rho}(x_0,t_0)\Subset\Omega_T$,
    with $0<2\rho\leq1$, there exist
    \[
        R_*=R_*(\datah,\|f\|_{L^d(Q_{2\rho}(x_0,t_0))})
        \in(0,1],
        \qquad C=C(\data,\alpha,d)\geq1,
    \]
    such that, setting
    \begin{equation}\label{pert:lambda-final}
        \lambda\coloneqq C\left[
        1+\frac{\omega}{\min\{\rho,R_*\}}
        +B^{-1}\left(
        \frac{\omega^2}{\min\{\rho,R_*\}^2}\right)
        \right],\qquad  \omega\coloneqq\operatorname*{osc}_{Q_{2\rho}(x_0,t_0)}u,
    \end{equation}
    we have
    \begin{equation}\label{pert:gradient}
        \|Du\|_{L^\infty(Q_\rho(x_0,t_0))}\leq\lambda.
    \end{equation}
    Moreover,
    \begin{equation}\label{pert:holder}
    \begin{split}
        |Du(x,t)-Du(y,s)|
        \leq C\lambda\left[
        \frac{|x-y|+\sqrt{a(\lambda)|t-s|}}
        {\min\{\rho,R_*\}\min\{1,\sqrt{a(\lambda)}\}}
        \right]^{\beta_*}
    \end{split}
    \end{equation}
    for all $(x,t),(y,s)\in Q_{\rho/2}(x_0,t_0)$, and
    \begin{equation}\label{pert:oscillation}
        \operatorname*{osc}_{Q_r(x,t)}u
        \leq C(\datah)\left[
        \lambda+b(\lambda)
        +\|f\|_{L^d(Q_{2\rho}(x_0,t_0))}\right]r
    \end{equation}
    whenever $(x,t)\in Q_{\rho/2}(x_0,t_0)$ and
    $0<r\leq\rho/2$.
    \end{itemize}
\end{theorem}

\subsection{Examples}\label{intro:examples}

We give a few examples of equations covered by
Theorem \ref{main:thm}.
For examples of admissible Orlicz functions $B$, we refer to
\cite[pp.~5--6]{Antonini2027}, \cite{Lieberman1991},
and \cite[Section 2.3]{BaroniLindfors2017}.
Such growth conditions arise, for instance, in energy models
from mechanics, fluid dynamics, and magnetism.
A concrete example is provided by scalar potential formulations
of stationary, irrotational subsonic flows of compressible fluids.
The corresponding energy densities depend on the gradient
through nonpower laws and, after suitable extension, satisfy
these growth conditions.
We refer to \cite[Introduction]{BaroniLindfors2017}
for a discussion of these models.

As already mentioned, the prototypical example is the evolutionary
$p$-Laplace equation
\[
    \partial_t u-\Delta_p u=f,
    \qquad p>1,
\]
corresponding to $a(t)=t^{p-2}$ and $B(t)=t^p/p$.

Our assumptions also allow for spatially varying coefficients
and a divergence datum:
\[
    \partial_t u-\operatorname{div}\big(
    c(x,t)a(|Du|)Du\big)
    =f+\operatorname{div}F.
\]
Here $c:\Omega_T\to\R$ and $F:\Omega_T\to\R^N$
are bounded, measurable in time, and H\"older
continuous in space, uniformly in time, with $0<c_0\leq c(x,t)\leq c_1<\infty$.

Anisotropic examples are given by
\[
    \partial_t u-\operatorname{div}\left(
    c(x,t)\,a\big(H(Du)\big)
    \nabla_\xi H^2(Du)\right)
    =f+\operatorname{div}F,
\]
where $H$ is a norm on $\R^N$, of class
$C^2(\R^N\setminus\{0\})$, such that
\[
    \nu\,\mathrm{Id}
    \leq \frac12\nabla_\xi^2H^2(\xi)
    \leq L\,\mathrm{Id},
    \qquad \xi\neq0,
\]
for some $0<\nu\leq L<\infty$. A further example is provided by operators associated
with a variable metric:
\[
    \partial_t u-\operatorname{div}\left(
    a\left(\sqrt{M(x,t)Du\cdot Du}\right)
    M(x,t)Du\right)
    =f+\operatorname{div}F,
\]
where $M$ is a symmetric matrix field satisfying
\[
    \nu\,\mathrm{Id}\leq M(x,t)\leq L\,\mathrm{Id},
\]
and is measurable in time and H\"older continuous
in space, uniformly in time.
This corresponds to the spatially and temporally dependent
norm $H(x,t,\xi)\coloneqq\sqrt{M(x,t)\xi\cdot\xi}$.

\subsection{Related literature}\label{intro:literature}

In the groundbreaking works
\cite{DiBenedettoFriedman1984,DiBenedettoFriedman1985,DiBenedettoFriedman1985Addendum} DiBenedetto and Friedman
established local H\"older continuity of the gradient for solutions
to the evolutionary $p$-Laplace system
\[
    \partial_t \mathbf{u}-\operatorname{div}(|D\mathbf{u}|^{p-2}D\mathbf{u})=0,
\]
in the supercritical range $p>2N/(N+2)$. The main idea is to work with intrinsic cylinders, whose time
scale is adjusted to the size of the gradient to compensate for
the different scaling of the evolutionary and diffusive terms. Further developments include the work
of Chen \cite{Chen1986}, the boundary estimates of Chen and DiBenedetto
\cite{ChenDiBenedetto1989}, and Choe's treatment of singular equations
\cite{Choe1991}. We refer to
\cite{DiBenedetto1993,Lieberman1996,DiBenedettoGianazzaVespri2012}
for accounts of the theory.

For systems with variable coefficients and a nontrivial
right-hand side, Misawa \cite{Misawa2002} considered systems of the form
\begin{equation}\label{eq:misawa}
    \partial_t\mathbf{u}-\operatorname{div}\left[
    \big(G(x,t)D\mathbf{u}:D\mathbf{u}\big)^{\frac{p-2}{2}}
    G(x,t)D\mathbf{u}
    \right]
    =
    \operatorname{div}\big(|F|^{p-2}F\big),
\end{equation}
where $G$ is a symmetric, uniformly elliptic matrix.
In the supercritical range $p>2N/(N+2)$, he established
local H\"older continuity of the gradient, assuming
that $G$ and $F$ are H\"older continuous with respect
to the usual parabolic metric. Further criteria for the right-hand side in the vectorial setting
were obtained by Karim and Misawa \cite{KarimMisawa2016}.

A related line of research concerns the transfer of integrability
from the data to the gradient. Starting with the higher integrability
theorem of Kinnunen and Lewis \cite{KinnunenLewis2000}, this theory
was developed further by Misawa \cite{Misawa2005} and Acerbi and
Mingione \cite{AcerbiMingione2007} in the framework of nonlinear
parabolic Calder\'on--Zygmund estimates.

Beyond power growth, Lieberman \cite{Lieberman2006} studied autonomous
parabolic systems of the form
\begin{equation}\label{eq:lieberm}
    \partial_t \mathbf{u}-\operatorname{div}(a(|D\mathbf{u}|)D\mathbf{u})=0.
\end{equation}
Under the conditions \eqref{iasa}, and assuming
a priori boundedness of the gradient, he proved its H\"older
continuity. His assumptions
also include a scale invariant continuity condition on the derivative
of the stress field $a(|\mathbf{p}|)\mathbf{p}$.

A substantial development of the inhomogeneous theory came from
pointwise gradient estimates in terms of potentials of the datum.
Duzaar and Mingione
\cite{DuzaarMingione2009,DuzaarMingione2011} obtained such estimates
through caloric Riesz potentials for parabolic equations with
quadratic growth.

Kuusi and Mingione subsequently developed an intrinsic potential
theory for nonlinear parabolic equations, announced in
\cite{KuusiMingione2011Parabolic}. For autonomous scalar equations
of $p$-Laplace type with $p\geq2$, they established pointwise
gradient bounds through intrinsic Wolff potentials
\cite{KuusiMingione2014Wolff}. In the singular range $2-1/(N+1)<p\leq2$, intrinsic Riesz
potential estimates were obtained in
\cite{KuusiMingione2013Gradient}. Dong and Zhu
\cite{DongZhu2022} subsequently extended these estimates
to the larger range $2N/(N+1)<p\leq2$.

For $p\geq2$, the passage from Wolff to Riesz potentials was
achieved in \cite{KuusiMingione2014Riesz}, where the coefficients
are  also allowed to be Dini continuous in space and merely measurable
in time. The resulting estimates give gradient continuity
under local uniform decay of the caloric Riesz potential.
In particular, the Riesz potential criteria in
\cite{KuusiMingione2013Gradient,KuusiMingione2014Riesz}
yield gradient continuity under the borderline Lorentz
condition $f\in L(N+2,1)$ in their respective ranges.

Under the same Lorentz condition, Kuusi and Mingione
obtained local gradient boundedness throughout the
supercritical range $p>2N/(N+2)$ for autonomous scalar
equations of $p$-Laplace type and systems with Uhlenbeck
structure \cite{KuusiMingione2012Potential}.
They subsequently established gradient continuity throughout
this range for evolutionary $p$-Laplace systems, allowing
for a scalar coefficient that is Dini continuous in space
and measurable in time
\cite{KuusiMingione2014Borderline}.
For problems with a divergence datum, their perturbative
approach \cite{KuusiMingione2012Perturbation} yields gradient
continuity under suitable Dini assumptions on the spatial
coefficients and the datum, and nonlinear Schauder estimates
when both are H\"older continuous in space.
These latter results apply to scalar equations with general
operators of $p$-Laplace type and to systems with
radial structure.

Other contributions to parabolic potential theory
include the works of Liskevich, Skrypnik and Sobol
\cite{LiskevichSkrypnikSobol2011Potential,LiskevichSkrypnikSobol2011Gradient,LiskevichSkrypnikSobol2013Parabolic}
and Liskevich and Skrypnik \cite{LiskevichSkrypnik2013Singular}.
For related estimates for solutions of porous medium and doubly
nonlinear equations, we refer to B\"ogelein, Duzaar and Gianazza
\cite{BogeleinDuzaarGianazza2013Potential} and Sturm \cite{Sturm2018}.

A gradient H\"older theory for parabolic systems with a variable
exponent $p(x,t)$ was studied in 
\cite{BogeleinDuzaar2012}.
For scalar equations, Adimurthi and Banerjee
\cite{AdimurthiBanerjee2022} gave a unified approach to gradient
H\"older regularity in the singular and degenerate regimes.
Adimurthi, Ghosh and Tewary \cite{AdimurthiGhoshTewary2022}
studied time-dependent nonstandard growth, including double-phase
and variable-exponent models, under an a priori gradient bound.
Gradient bounds for systems with stronger degeneracies at small
gradient sizes were obtained by P. Ambrosio and B\"auerlein
\cite{AmbrosioBauerlein2024}.
For systems with power growth, intrinsic gradient oscillation estimates
were developed further by B\"ogelein, Duzaar, Liao and Scheven
\cite{BogeleinDuzaarLiaoScheven2022}. More recently, B\"ogelein, Duzaar, Gianazza, Liao and Scheven
\cite{Bogelein2026} proved Schauder estimates for locally
bounded weak solutions to the parabolic $p$-Laplace system
\[
\partial_t \mathbf{u}
-\operatorname{div}\big(a(x,t)(\mu^2+|D\mathbf{u}|^2)^{\frac{p-2}{2}}D\mathbf{u}\big)=0,
\qquad \mu\in[0,1],
\]
for every $p>1$, assuming that the coefficient $a(x,t)$
is H\"older continuous in space and measurable in time. 
A related direction concerns parabolic problems with
nonstandard $(p,q)$-growth. B\"ogelein, Duzaar and Marcellini
\cite{BogeleinDuzaarMarcellini2013Equations} established
quantitative local gradient bounds for scalar equations
with variable coefficients under suitable restrictions
on the gap between $p$ and $q$.
Variational approaches to existence for parabolic systems
with nonstandard and general growth were developed in
\cite{BogeleinDuzaarMarcellini2013Variational,
Marcellini2020Parabolic}.

For parabolic equations with Orlicz growth, Hwang and Lieberman
\cite{HwangLieberman2015Degenerate,HwangLieberman2015Singular}
established H\"older continuity of bounded solutions in the
degenerate and singular regimes, under the assumptions
$i_a\geq0$ and $s_a\leq0$, respectively.
For quantitative H\"older estimates in the degenerate regime,
see also Amaral and Ara\'ujo \cite{AmaralAraujo2026}.
Local boundedness for systems with general growth was proved by
Isernia \cite{Isernia2018}.

Baroni and Lindfors
\cite{BaroniLindfors2017} obtained interior spatial Lipschitz estimates
and boundary continuity for a broad class of autonomous scalar equations
\[
\partial_t u-\mathrm{div}\big(\A(Du) \big)=0,
\]
in the supercritical regime $i_B>2N/(N+2)$.

Diening, Scharle and Schwarzacher
\cite{DieningScharleSchwarzacher2019} established quantitative
local gradient bounds for parabolic systems
\eqref{eq:lieberm} with Uhlenbeck structure and Orlicz growth,
assuming local square integrability of the gradient and
an additional growth condition, satisfied in particular when
$i_B>2-4/N$.
Their argument treats singular and degenerate diffusion
simultaneously. Combined with Lieberman's theorem
\cite{Lieberman2006}, these bounds also yield gradient
H\"older continuity under a further continuity assumption
on the second derivative of the energy.

Ok, Scilla and Stroffolini \cite{OkScillaStroffolini2024Uhlenbeck}
subsequently obtained boundedness of solutions and gradients, and
H\"older continuity of the gradient, directly from the natural energy
class in the supercritical range. Their gradient H\"older estimates
require an additional H\"older condition on the Hessian of the radial energy.

Other aspects of the Orlicz theory include higher integrability,
established by H\"ast\"o and Ok \cite{HastoOk2021} and extended to
generalized Orlicz growth in \cite{HastoOk2025Generalized}, and
Calder\'on-Zygmund estimates with discontinuous coefficients, proved
by Oh and Ok \cite{OhOk2022}.
The parabolic Lipschitz truncation and caloric approximation of Diening,
Schwarzacher, Stroffolini and Verde
\cite{DieningSchwarzacherStroffoliniVerde2017}
provide tools for systems with general growth. For systems without
Uhlenbeck structure, these tools lead to partial gradient regularity;
see Foss, Isernia, Leone and Verde
\cite{FossIserniaLeoneVerde2023} and Ok, Scilla and Stroffolini
\cite{OkScillaStroffolini2024Partial}.

% Boundary gradient bounds were obtained by B\"ogelein
% \cite{Bogelein2015Global} for systems with nonhomogeneous
% Dirichlet data, and by B\"ogelein, Duzaar, Liao and Scheven
% \cite{BogeleinDuzaarLiaoScheven2022Boundary} for systems with
% Uhlenbeck structure under homogeneous Dirichlet conditions
% in convex domains.
% We also mention the work of Lee, Lian, Yun and Zhang \cite{LeeLianYunZhang2025Boundary}, who
% developed a viscosity approach to boundary gradient
% H\"older regularity for the parabolic $p$-Laplace equation.
% For boundary gradient bounds for nondegenerate systems
% with $(p,q)$-growth in smooth convex domains, see Strunk
% \cite{Strunk2026Boundary}.

\subsection{Main novelties and strategy of the proof}\label{intro:proof}

Theorem \ref{main:thm} extends the scalar parabolic Schauder
theory from power growth to general Orlicz growth, allowing
for coefficients that are H\"older continuous in space
and merely measurable in time.

 The results  cover singular and degenerate
diffusion, including growth laws exhibiting both behaviors,
within a single argument. The scalar structure of the problem allows us
to consider general, possibly nonradial operators without
imposing any additional continuity assumption on
$\nabla_\xi\A$, in contrast to the results for parabolic
systems with Orlicz growth 
\cite{Lieberman2006,OkScillaStroffolini2024Uhlenbeck}.

A further contribution concerns the homogeneous equation
\[
    \partial_t v-\operatorname{div}\big(\A_0(t,Dv)\big)=0,
\]
for which we establish gradient bounds on arbitrary cylinders
$Q_{R,S}$ in Theorem \ref{thm:grbddhom}. More precisely, setting
\[
    E\coloneqq\miint_{Q_{R,S}}B(|Dv|)\,dxdt,
    \qquad F(s)\coloneqq s^2a(s)^{(N+2)/2},
\]
in the supercritical range $i_B>2N/(N+2)$ we prove
\begin{equation}\label{se:0}
    \|Dv\|_{L^\infty(Q_{R/2,S/4})}
    \leq C\left[
    \left(\frac{S}{R^2}E\right)^{1/2}
    +F^{-1}\left(
    \left(\frac{R^2}{S}\right)^{N/2}E
    \right)\right],
\end{equation}
while for bounded solutions we obtain
\begin{equation}\label{se:1} \|Dv\|_{L^\infty(Q_{R/2,S/4})}
    \leq C\left[
    \frac{\operatorname*{osc}_{Q_{R,S}}v}{R}
    +B^{-1}\left(
    \frac{(\operatorname*{osc}_{Q_{R,S}}v)^2}{S}
    \right)\right],
\end{equation}
which holds throughout the admissible growth range $i_B>1$. 

The first estimate controls the gradient solely through
the natural energy. This is the Orlicz analogue of the classical bounds for the gradient of solutions to the evolutionary
$p$-Laplace system-- see
\cite[Chapter VIII, Theorems~5.1 and~5.2$'$, Remark~5.4]{DiBenedetto1993} and \cite[Propositions 3.5, 3.11]{Bogelein2026}.
Its formulation for general growth may also be useful in the
development of parabolic gradient potential estimates.

We remark that the supercritical restriction
cannot in general be dropped without additional information
on the solution: already for the evolutionary $p$-Laplace
equation, local boundedness may fail in the subcritical range;
see \cite[Chapter V, Section 5]{DiBenedetto1993} and the
counterexample in
\cite[Chapter XII, Section 13]{DiBenedetto1993}.

The second estimate replaces the energy by the oscillation
of the solution and extends gradient boundedness to the full
growth range $i_B>1$.  Together with the intrinsic gradient oscillation
estimates described below, it provides the homogeneous comparison estimates
needed in the perturbative argument.

Despite the more general growth, our proof of
\eqref{se:0}--\eqref{se:1} considerably simplifies the classical
argument for the $p$-Laplace system in
\cite[Chapter VIII, Section 5]{DiBenedetto1993}.
The idea is the following.
Following the approach of
\cite[Theorem 5.2]{KuusiMingione2012Potential} and
\cite[Proposition 3.5]{BaroniLindfors2017},
we first establish a conditional pointwise gradient estimate
on intrinsic cylinders:
\[
    \sup_{Q_R^{(\lambda)}(x_0,t_0)}|Dv|\leq4\lambda
    \quad\Longrightarrow\quad
    |Dv|(x_0,t_0)
    \leq\lambda+C\left(
    \frac{1}{a(\lambda)}
    \miint_{Q_R^{(\lambda)}(x_0,t_0)}
    B(|Dv|)\,dxdt\right)^{1/2}.
\]
This is the content of Proposition \ref{prop:temp:grbd}.
Its proof exploits the fact that
$H\coloneqq(|Dv|^2-\lambda^2)_+$ is a nonnegative subsolution
to a linear equation whose coefficients are uniformly
parabolic after intrinsic scaling.

Once this is established, we obtain \eqref{se:0} through
an interpolation argument on nested cylinders, distinguishing
the cases
\[
    \theta_\lambda\coloneqq\frac{S}{R^2}a(\lambda)\geq1
    \qquad\text{and}\qquad
    \theta_\lambda<1.
\]
These correspond to whether the prescribed time depth is
larger or smaller than the intrinsic one--
compare with the arguments in
\cite[Chapter VIII, p.~239, below (5.4), and p.~241]{DiBenedetto1993}.

To prove \eqref{se:1}, we instead compare the gradient scale
$\lambda$ with the oscillation scale $\omega/r$, distinguishing
the cases
\[
    \lambda r\geq K\omega
    \qquad\text{and}\qquad
    \lambda r<K\omega,
\]
where $\omega=\operatorname*{osc}_{Q_{R,S}}v$ and $K$
is a sufficiently large structural constant.

Next we establish the H\"older continuity of $Dv$ through
the intrinsic oscillation estimate
\begin{equation}\label{se:2}
  \sup_{Q^{(\l)}_R}|Dv|\leq \lambda\implies  \operatorname*{osc}_{Q_r^{(\lambda)}}Dv
    \leq C\lambda\left(\frac rR\right)^\beta,
    \qquad 0<r\leq R.
\end{equation}

Its proof rests on the parabolic De Giorgi type inequalities in
Propositions \ref{pr:1DG} and \ref{prop:levels}, which extend the
elliptic inequalities of \cite[Lemmas 4.6, 4.7]{Antonini2027}, together with the
logarithmic estimate in Proposition \ref{prop:log}. 
 They lead to the fundamental alternative of Proposition \ref{prop:fund:alt}:
either the gradient stays away from zero in a smaller cylinder,
where the differentiated equation becomes uniformly parabolic after
intrinsic scaling, or its size decreases by a fixed factor.

The De Giorgi inequalities simplify this argument and allow
the full growth class to be treated without separate proofs
for singular and degenerate diffusion. As already mentioned,
no additional continuity assumption on $\nabla_\xi\A$ is
needed in our argument, whereas in the vectorial setting such assumptions
are used in the first alternative to control the comparison
with solutions of the linearized system--see \cite[Lemma 6.6, pp.~146--147]{OkScillaStroffolini2024Uhlenbeck}
and \cite[Lemma 5.3, pp.~362--363]{Lieberman2006}.

Finally, we transfer these estimates to \eqref{eq:1}
through a perturbative argument on intrinsic cylinders,
similar to, yet simpler than, those in
\cite{KuusiMingione2012Perturbation,KuusiMingione2014Riesz,Bogelein2026}. More in detail, we fix $\lambda\geq1$
as an upper bound for the gradient $|Du|$ on an outer cylinder.
On each smaller intrinsic cylinder
$Q_R^{(\lambda)}=Q_R^{(\lambda)}(x_0,t_0)$, we consider
the solution $v$ to the frozen problem
\[
    \begin{cases}
        \partial_t v-\operatorname{div}\big(\A(x_0,t,Dv)\big)=0
        &\text{in }Q_R^{(\lambda)},\\
        v=u&\text{on }\partial_{\rm par}Q_R^{(\lambda)}.
    \end{cases}
\]
The comparison estimate of Lemma \ref{lem:pert:comparison} gives
\[
    \miint_{Q_{R/2}^{(\lambda)}}|Du-Dv|^2\,dxdt
    \leq C\lambda^2\mathcal P_\lambda(R)^{2/s_B},
\]
where $\mathcal P_\lambda(R)$, defined in \eqref{pert:size},
measures the spatial oscillation of the coefficients
and the contribution of the source term. For $\lambda\geq1$,  owing to the assumptions  \eqref{A:hold} and \eqref{ass:f},
these errors are controlled by $(R/R_*)^\sigma$ for some $\sigma\in(0,1)$ depending on the data. 

The subsequent argument can then be summarized as follows:
\[
\begin{aligned}
    \sup_{Q_R^{(\lambda)}}|Du|\leq\lambda
    &\implies
    \operatorname*{osc}_{Q_R^{(\lambda)}}u
    \leq C\lambda R
    &&\text{(Oscillation Lemma \ref{lem:pert:gluing})}
    \\
    &\implies
    \operatorname*{osc}_{Q_R^{(\lambda)}}v
    \leq C\lambda R
    &&\text{(Maximum principle)}
    \\
    &\implies
    \|Dv\|_{L^\infty(Q_{R/2}^{(\lambda)})}
    \leq C\lambda
    &&\text{(by \eqref{se:1})}
    \\
    &\implies
    \operatorname*{osc}_{Q_r^{(\lambda)}}Dv
    \leq C\lambda\left(\frac rR\right)^\beta
    &&\text{(by \eqref{se:2})}
        \\
    &\implies
    \left(
    \miint_{Q_r^{(\lambda)}}
    |Du-(Du)_{Q_r^{(\lambda)}}|^2\,dxdt
    \right)^{1/2}
    \leq C\lambda\left(\frac r{R_0}\right)^{\beta_*}
    &&\text{(Comparison)}
\end{aligned}
\]
for all $0<r\leq R/2$, and a suitable $R_0$.
We then exploit a simple trick: spatial integration by parts gives $|(Du)_{Q_r^{(\lambda)}}|\leq C\,\omega/r$,
where $\omega$ denotes the oscillation of $u$ on the outer cylinder. Combining this bound with a telescoping argument
based on the Campanato estimate above yields
\[
    |Du(x,t)|
    \leq C\lambda\left(\frac r{R_0}\right)^{\beta_*}
    +C\frac{\omega}{r}.
\]
Choosing $r$ as a sufficiently small fixed fraction of $R_0$
and taking the supremum over $(x,t)$, we may absorb the term
involving $\lambda$  and (modulo an iteration on nested cylinders) we  obtain \eqref{pert:lambda-final}-\eqref{pert:gradient}. The Campanato estimate above then gives
the H\"older continuity of $Du$ in \eqref{pert:holder}.

The arguments above are first carried out under
additional smoothness assumptions on the solution and the
coefficients. These are removed in Section \ref{sec:approximation}
through an approximation procedure based on
Lemma \ref{lem:approx:A}, which regularizes the stress field
$\A$ while preserving its structural properties.

\subsection{Outline of the paper} The paper is organized as follows.
Section \ref{sec:prel} collects notation and preliminary results.
Section \ref{sec:bddness} establishes local boundedness of solutions.
The homogeneous gradient estimates are proved in
Section \ref{sec:homogeneous}, while
Section \ref{sec:perturb} develops the perturbative argument
and proves the main theorem under auxiliary smoothness assumptions.
These assumptions are removed through the approximation procedure
in Section \ref{sec:approximation}.

\section{Preliminary results}\label{sec:prel}

\subsection{Notation}

We denote points in $\R^{N+1}$ by $(x,t)$, where
$x=(x_1,\dots,x_N)\in\R^N$ and $t\in\R$.
We denote by $x\cdot y$ the standard scalar product in $\R^N$,
and by $|x|=\sqrt{x\cdot x}$ the Euclidean norm of $x$.
We write $e_i$ for the $i$-th canonical unit vector of $\R^N$.
Also, for a matrix $M=\{M_{ij}\}$, we denote by
$|M|=(\sum_{i,j}|M_{ij}|^2)^{1/2}$ its Frobenius norm.

We write $B_R(x_0)$ for the $N$-dimensional ball of radius
$R>0$ centered at $x_0\in\R^N$, and when $x_0=0$ we simply
write $B_R\coloneqq B_R(0)$.
Given $R,S>0$ and $(x_0,t_0)\in\R^{N+1}$, we define the
backward parabolic cylinder
\[
    Q_{R,S}(x_0,t_0)
    \coloneqq B_R(x_0)\times(t_0-S,t_0),
\]
and write
\[
    Q_R(x_0,t_0)\coloneqq Q_{R,R^2}(x_0,t_0).
\]
For $\lambda>0$, we also introduce the intrinsic cylinders
\begin{equation}\label{def:cilintr}
    Q_R^{(\lambda)}(x_0,t_0)
    \coloneqq
    B_R(x_0)\times
    \left(t_0-\frac{R^2}{a(\lambda)},t_0\right).
\end{equation}
When the center is the origin, we simply write
$Q_{R,S}$, $Q_R$ and $Q_R^{(\lambda)}$, respectively.
The reference to the center will also be omitted whenever
it is clear from the context.

More generally, we write $Q=\mathcal B\times\mathcal T$
for a cylinder with spatial base $\mathcal B\subset\R^N$
and time interval $\mathcal T=(t_1,t_2)$.
Its parabolic boundary is denoted by
\[
    \partial_{\rm par}Q
    \coloneqq
    (\partial\mathcal B\times\mathcal T)
    \cup
    (\overline{\mathcal B}\times\{t_1\}).
\]

Given a measurable set $E$ in $\R^N$ or $\R^{N+1}$,
we denote by $\partial E$ its boundary, $\overline E$ its
closure, and $|E|$ its Lebesgue measure in the corresponding
dimension. We write $\mathcal H^{N-1}$ for the
$(N-1)$-dimensional Hausdorff measure.
The notation $E\Subset U$ means that $\overline E$ is a
compact subset of the open set $U$.
Also, we write $\mathbbm{1}_E$ for the characteristic
function of $E$.

For a real-valued function $v=v(x,t)$, we write
\[
    Dv=(D_1v,\dots,D_Nv),
    \qquad D_kv=\partial_{x_k}v,
\]
for its spatial gradient, $D^2v$ for its spatial Hessian
matrix, and $\partial_tv$ for its time derivative.
These derivatives are understood in the distributional
sense whenever classical derivatives are not available.
If $F=(F^1,\dots,F^m)$ is vector-valued, we write
$(DF)_{ij}=D_jF^i$.
For the stress field $\A=\A(x,t,\xi)$, the symbols
$\nabla_x\A$ and $\nabla_\xi\A$ denote differentiation
with respect to $x$ and $\xi$, respectively, with the
other variables kept fixed.

We use the standard notation $L^p(E)$ and
$\|\cdot\|_{L^p(E)}$ for Lebesgue spaces and their norms.
If $E\subset\R^N$ and $Q\subset\R^{N+1}$ have finite
positive measure, we define the spatial and space-time
averages of an integrable function $F$ by
\[
    (F)_E
    \coloneqq \mint_E F\,dx
    =\frac{1}{|E|}\int_E F\,dx,
    \qquad
    (F)_Q
    \coloneqq \miint_Q F\,dxdt
    =\frac{1}{|Q|}\iint_Q F\,dxdt.
\]

For a measurable function $F:E\to\R^m$ and $a\in\R^m$,
we write
\[
    \{F=a\}\coloneqq\{z\in E:F(z)=a\}.
\]
Similarly, if $F$ is real-valued, its sub- and superlevel
sets are denoted by $\{F<a\}$, $\{F>a\}$, etc.
We also write $\operatorname{spt}F$ for the support of $F$.
For a real-valued function $v$, we denote its positive
and negative parts by
\[
    v_+\coloneqq\max\{v,0\},
    \qquad
    v_-\coloneqq-\min\{v,0\}.
\]
If $F=(F^1,\dots,F^m)$, we define its oscillation over $E$ by
\[
    \operatorname*{osc}_E F
    \coloneqq
    \max_{i=1,\dots,m}
    \left(
        \operatorname*{ess\,sup}_E F^i
        -
        \operatorname*{ess\,inf}_E F^i
    \right).
\]
For continuous functions, essential suprema and infima
can be replaced by ordinary suprema and infima.

For $\beta\in(0,1)$, the parabolic H\"older seminorm is
defined by
\[
    [F]_{C^{\beta,\beta/2}(Q)}
    \coloneqq
    \sup_{\substack{(x,t),(y,s)\in Q\\(x,t)\neq(y,s)}}
    \frac{|F(x,t)-F(y,s)|}
    {\big(|x-y|+|t-s|^{1/2}\big)^\beta}.
\]
We endow $C^{\beta,\beta/2}(Q)$ with the norm
\[
    \|F\|_{C^{\beta,\beta/2}(Q)}
    \coloneqq
    \|F\|_{L^\infty(Q)}
    +[F]_{C^{\beta,\beta/2}(Q)}.
\]
The notation $C^{1,1/2}_{\rm loc}$ refers to functions
with continuous spatial first derivatives and local
$1/2$-H\"older continuity with respect to time.

Finally, we introduce the structural quantities
\begin{equation}\label{data}
    \begin{split}
        \data&\coloneqq(N,\lambda_0,\Lambda_0,i_a,s_a),
        \\
        \datah&\coloneqq(\data,\alpha,d,\Lambda_{\rm h}).
    \end{split}
\end{equation}
Throughout the paper, $c$ and $C$ denote positive constants
whose values may change from line to line.
Their dependence on the relevant parameters will be
specified in the statements and proofs.
\subsection{Young functions} \label{sec:young}

In this section we recollect few useful properties of the Young function $B$ and of $b,a$. Their proofs can be found in \cite[Section 2.1]{Antonini2027}.

We have
\begin{equation}\label{ciu:a}
    \begin{split}
    C^{i_a}a(t)\leq a(C\,t)\leq C^{s_a}a(t)\,,\quad &\text{for all $C>1$}
    \\
    c^{s_a}a(t)\leq a(c\,t)\leq c^{i_a}a(t)\quad &\text{for all $c<1$,}
    \end{split}
\end{equation}

Given  $m_0,M_0>0$  two fixed constants, there holds
\begin{equation}\label{ultra:utile}
    c_0\leq\frac{a(t_2)}{a(t_1)}\leq C_0\,,\quad \text{for all $t_1,t_2>0$ such that $m_0\leq\frac{t_2}{t_1}\leq M_0$}
\end{equation}
where $c_0,C_0>0$ depend on $i_a,s_a,m_0,M_0$.
\vspace{0.1cm}

Setting
\begin{equation}\label{def:isb}
    i_b=i_a+1,\quad\text{and}\quad s_b=s_a+1\,,
\end{equation}
then
\begin{equation}\label{ibsb}
    \begin{split}
            0<i_b\leq\inf_{t>0}\frac{b'(t)\,t}{b(t)}\leq \sup_{t>0} \frac{b'(t)\,t}{b(t)}\leq s_b<\infty\,,
    \end{split}
\end{equation}

In particular, the function $b$ is strictly monotone, $b(0)=0$, $\lim_{t\to +\infty} b(t)=+\infty$, and
\begin{equation}\label{b2t}
\begin{split}
    C^{i_b}b(t)\leq b(C\,t)\leq C^{s_b}b(t)\,,\quad &\text{for all $C>1$}
    \\
    c^{s_b}b(t)\leq b(c\,t)\leq c^{i_b}b(t)\quad &\text{for all $c<1$,}
    \end{split}
\end{equation}
and the (quasi-)triangle inequality for $b(t)$:
\begin{equation}\label{triangleb}
    b(t+s)\leq 2^{s_b}\left\{b(t)+b(s)\right\},\quad t,s\geq 0.
\end{equation}

 The functions $B(t)$, $b(t)\,t=a(t)\,t^2$ are equivalent, that is
\begin{equation}\label{Bcomeb}
    \left(\frac{1}{1+s_b}\right)a(t)\,t^2\leq B(t)\leq \left(\frac{1}{1+i_b}\right)a(t)\,t^2,\quad \text{for $t>0$.}
\end{equation}

\begin{equation}\label{B=1}
   \frac{1}{2+s_a}\leq  B(1)\leq \frac{1}{2+i_a}.
\end{equation}
We also set
\begin{equation}
    i_B=i_b+1,\quad\text{and}\quad s_B=s_b+1\,,
\end{equation}
so that we have
\begin{equation}\label{iBsB}
    1<i_B\leq \inf_{t>0}\frac{B'(t)t}{B(t)}\leq \sup_{t>0}\frac{B'(t)t}{B(t)}\leq s_B<\infty\,,
\end{equation}
we have  the so-called $\Delta_2$ and $\nabla_2$-properties: 
\begin{equation}\label{B2t}
\begin{split}
    &C^{i_B}B(t)\leq B(C\,t)\leq C^{s_B}B(t)\,,\quad \text{for all $C>1$}
    \\
    &c^{s_B}B(t)\leq B(c\,t)\leq c^{i_B}B(t)\quad \text{for all $c<1$,}
    \end{split}
\end{equation}
and the (quasi-)triangle inequality for $B(t)$:
\begin{equation}\label{triangle}
    B(t+s)\leq 2^{s_B}\left\{B(t)+B(s)\right\},\quad t,s>0.
\end{equation}

We also have
\begin{equation}\label{usso}
    t^{i_B}\leq C\,(B(t)+1)\quad\text{for all $t>0$.}
\end{equation}
for some constant $C=C(i_B,s_B)$.
The following monotonicity properties hold:
\begin{equation}\label{mon:iBsB}
    \begin{split}
        &t\mapsto \frac{B(t)}{t^{\mathfrak i_B}} \quad\text{is a nondecreasing  function for all $\mathfrak i_B\leq i_B$}
        \\
        &t\mapsto \frac{B(t)}{t^{\mathfrak s_B}} \quad\text{is a nonincreasing function for all $\mathfrak{s}_B\geq s_B$}.
    \end{split}
\end{equation}

Another elementary, yet very useful inequality is the following:
\begin{equation}\label{el:comparison}
    B(|\xi-\eta|)\leq C\,\delta\big[B(|\xi|)+B(|\eta|) \big]+C\,\delta^{-1}\,a\big(|\xi|+|\eta| \big)\,|\xi-\eta|^2,
\end{equation}
valid for all $\delta \in (0,1]$, for all $\xi,\eta\in \R^n$, with $C=C(i_a,s_a)$. 
\vspace{0.1cm}

Let us now introduce the \emph{complementary Young function} (also called conjugate Young function)
\begin{equation}\label{def:wB}
    \widetilde{B}(t)=\sup\left\{st-B(s):\,s>0 \right\}\,.
\end{equation}
This may be written as
\begin{equation}\label{wB}
    \wB(t)=\int_0^t b^{-1}(s)\,ds
\end{equation}
where $b^{-1}(s)$ is the inverse function of $b$.
We have that
\begin{equation*}
    0<\frac{1}{s_b}\leq \inf_{s>0}\frac{(b^{-1})'(s)\,s}{b^{-1}(s)}\leq \sup_{s>0}\frac{(b^{-1})'(s)\,s}{b^{-1}(s)}\leq \frac{1}{i_b}<\infty,
\end{equation*}
and
\begin{equation}\label{wBcome}
    \left(\frac{1}{i_B'}\right)b^{-1}(s)\,s\leq \wB(s)\leq \left(\frac{1}{s_B'}\right)b^{-1}(s)\,s,\quad s>0,
\end{equation}
and
\begin{equation}
    1<s'_B\leq\inf_{t>0}\frac{\wB'(s)s}{\wB(s)}\leq \sup_{t>0}\frac{\wB'(s)s}{\wB(s)} \leq i'_B<\infty\,,
\end{equation}
where 
\[
i_B'=\frac{i_B}{i_B-1}\quad \text{and}\quad s_B'=\frac{s_B}{s_B-1}
\]
are the H\"older's conjugates of $i_B,s_B>1$, respectively.
We have
\begin{equation}\label{tB:usef}
\begin{split}
    &c^{i_B'}\wB(t)\leq \wB(c t)\leq c^{s_B'}\wB(t),\quad\text{for all } c\in (0,1)\text{ and $t>0$}
    \\
    &C^{s_B'}\wB(t)\leq \wB(C t)\leq C^{i_B'}\wB(t),\quad\text{for all } C\geq 1 \text{ and $t>0$},
    \end{split}
\end{equation}
and
\begin{equation}\label{tB=1}
    c(i_a,s_a)\leq \widetilde{B}(1)\leq C(i_a,s_a).
\end{equation}
Moreover, the function $B\circ b^{-1}$ is equivalent to $\widetilde{B}$, that is
\begin{equation}\label{come:BwB}
  c(i_a,s_a)\,\wB(t)  \leq B\big(b^{-1}(t) \big)\leq C(i_a,s_a)\,\wB(t),\quad t>0.
\end{equation}

We have Young's inequality for Orlicz-functions
\begin{equation}\label{Young}
\begin{split}
st& \leq\delta^{i_B} B(s)+\delta^{-i_B'}\wB(t)\quad\text{and}
\\
    st&\leq \delta^{s_B'}\wB(s)+\delta^{-s_B}\,B\left( t\right),\quad \text{for any $\delta\in (0,1]$,}
    \end{split}
\end{equation}
which readily follow by the definition of $\wB$ in \eqref{def:wB}, together with \eqref{B2t} and \eqref{tB:usef}.

Moreover, we have
\begin{equation}\label{wBbt}
    c(i_b,s_b)\,B(t)\leq \wB\big(b(t)\big)\leq C(i_b,s_b)\,B(t)\,,\quad t>0,
\end{equation}
and the following version of Young's inequality holds:
\begin{equation}\label{young1}
    b(t)\,s\leq \delta\,B(t)+C(i_a,s_a)\,\delta^{-(s_B-1)}\,B(s).
\end{equation}

The following approximation lemma holds 

\begin{lemma}\label{lemma:ae}
    Let $a\in C^1((0,\infty))$ be a function satisfying \eqref{iasa}, with $i_a\leq 0\leq s_a$. Then there exists a sequence of functions $\{a_\e\}_{\e\in(0,1)}\subset C^\infty\big([0,+\infty)\big)$, such that, for all $\e>0$, $a_\e(1)=1$,
    \begin{equation}\label{ae:nonzero}         
  \frac{1}{2}\, \e\leq  \,a_\e(t)\leq 2\,\e^{-1}\quad\text{for all $t\geq 0$,}
    \end{equation}   
and
    \begin{equation}\label{iaesae}
  a_\e(1)=1,\qquad i_a\leq \inf_{t>0}\frac{a_\e'(t)\,t}{a_\e(t)}\leq \sup_{t>0}\frac{a_\e'(t)\,t}{a_\e(t)}\leq s_a,
    \end{equation}
Moreover, by setting
\begin{equation}\label{def:bBe}
    b_\e(t)=a_\e(t)\,t,\quad B_\e(t)=\int_0^t b_\e(s)\,ds
\end{equation}
  we have
    \begin{align}
        &\lim_{\e\to 0} a_\e=a\quad\text{uniformly in  $[L,M]$ for every $0<L\leq M$}\label{ae:unif}
        \\
       & \lim_{\e\to 0} b_\e=b\quad\text{uniformly in  $ [0,M]$ for every $M>0$}\label{be:unif}
    \end{align}
    and hence
    \begin{align}
    &\lim_{\e\to 0} a_\e(|\xi|)\,\xi=a(|\xi|)\,\xi\quad\text{uniformly in  $|\xi|\leq M$ for every $M>0$.}\label{aexi:unif}
    \\
        &\lim_{\e\to 0} B_\e=B\quad\text{uniformly in $ [0,M]$ for every $M>0$.}\label{BE:unif}
        \end{align}
Furthermore
\begin{equation}\label{aexi:Cinf}
 \text{the map}\quad \xi\mapsto a_\e(|\xi|)\quad\text{is of class $C^\infty(\R^N)$.}  
\end{equation}
\end{lemma}
For the proof, see \cite[Lemma 2.1]{Antonini2027}, \cite[Lemma 3.3]{CianchiMazya2011}, \cite[Lemma 4.5]{CianchiMazya2014}. The normalization $a_\e(1)=1$ can obtained by dividing
the approximating functions there by their value at $1$.

\subsection{The stress field \texorpdfstring{$\mathcal{A}$}{A}}\label{sec:stress}
In this subsection, we collect some properties of the vector field $\A$, 
often referred to as the \emph{stress field}.  

We start with the following lemma establishing growth and coercivity properties of $\A$. The proof follows from
\cite[Lemma 2.4]{Antonini2027}, with the same argument
applying to $\A(x,t,\xi)$ in place of $\A(x,\xi)$.

\begin{lemma}\label{lemma:Agrco}
  Assume that $a,\A$ satisfy \eqref{iasa}, \eqref{ass:A}. Then we have
\begin{equation}\label{co:gr}
    \begin{split}
        &\sum_{i=1}^N |\A^i(x,t,\xi)-\A^i(x,t,0)|\leq C_1\,b(|\xi|)
        \\
        &\big(\A(x,t,\xi)-\A(x,t,0)\big)\cdot \xi\geq c_1\,B(|\xi|)\,,\quad \text{for a.e. $(x,t)\in \Omega_T$ and  all $\xi\in \R^N$,}
    \end{split}
\end{equation}
where $c_1=c_1(\data), C_1(\data)$ are positive constants. Moreover,
\begin{equation}\label{strong:coer}
    \big(\A(x,t,\xi)-\A(x,t,\eta)\big)\cdot (\xi-\eta)\geq c_0\,a\big(|\xi|+|\eta|\big)\,|\xi-\eta|^2.
\end{equation}
with $c_0=c_0(\data)>0$.
\end{lemma}

Now let  $Q=\B\times \T\subset \Omega_T$; since the equation is left unchanged by subtracting $\A(x_0,t,0)$ for any $x_0\in \Omega$, we may assume that
\[
\A(x_0,t,0)=0\quad\text{for all $t\in (0,T)$, }
\]
for any fixed $x_0\in \B$. By \eqref{A:hold} it then follows that 
\begin{equation}\label{ovu:Linf}
\sup_{(x,t)\in Q}|\A(x,t,0)|\leq \Lambda_{\mathrm{h}}\,(\mathrm{diam} \,Q)^\alpha\quad\text{for all $Q\Subset \Omega_T$,}
\end{equation}
 From this inequality, \eqref{co:gr} and \eqref{Young}, we deduce that
\begin{equation}\label{deb:ass:A}
    \begin{cases}
        |\A(x,t,\xi)|\leq C_1\,b(|\xi|)+\Lambda_{\mathrm{h}}\,(\mathrm{diam} Q)^\alpha
        \\
        \A(x,t,\xi)\cdot \xi\geq c_1\,B(|\xi|)-c\widetilde{B}\left(\Lambda_{\mathrm{h}}\,(\mathrm{diam} Q)^\alpha\right),
    \end{cases}
\end{equation}
for a.e. $(x,t)\in Q$ and $\xi \in \R^N$, with $c_1,C_1,c$ depending on $\data$.

We shall also need to regularize Equation \eqref{eq:1}.
To this end, we establish the following approximation lemma
for the stress field $\A$.

\begin{lemma}[Approximation of the stress field]\label{lem:approx:A}
    Let $a$ and $\A$ satisfy \eqref{iasa}, \eqref{ass:A},  \eqref{A:hold} and $\A(x_0,t,0)=0$ for some $x_0\in \Omega$ and all $t\in (0,T)$. Let   $Q\Subset\Omega_T$ be a cylinder, and let $a_\e,\,b_\e,B_\e$ be given by Lemma \ref{lemma:ae}.
    
    Then there exists a sequence of vector fields
    $\{\A_\epsilon\}_{\epsilon\in (0,1)}\subset  C^\infty(Q\times\R^N;\R^N)$ such that the following holds:
\begin{equation}\label{Ae:c}
    \A_\epsilon(x,t,\cdot)\to\A(x,t,\cdot)\quad\text{locally uniformly
    in $\R^N$ for a.e. $(x,t)\in Q$;}
\end{equation}
there exist $c,C>0$ depending only on $\data$ such that, for all $\epsilon >0$,

    \begin{equation}\label{approx:A:ell}
    \begin{split}
        \nabla_\xi\A_\epsilon(x,t,\xi)\zeta\cdot\zeta
        \geq c\,a_\epsilon(|\xi|)|\zeta|^2,\qquad|\nabla_\xi\A_\epsilon(x,t,\xi)|
        \leq C\,a_\epsilon(|\xi|),
    \end{split}
    \end{equation}
    for every $(x,t)\in Q$, $\xi,\zeta\in\R^N$,
\begin{equation}\label{sup:Dx}
    \sup_{Q'\times\R^N}|\nabla_x\A_\epsilon|\leq C_\epsilon<+\infty\quad\text{for all $Q'\Subset Q$,}
\end{equation}
  for some constant $C_\epsilon$,  and 
    \begin{equation}\label{approx:A:hold}
        |\A_\epsilon(x,t,\xi)-\A_\epsilon(y,t,\xi)|
        \leq C\Lambda_{\mathrm h}|x-y|^\alpha
        \big(1+b_\epsilon(|\xi|)\big)
    \end{equation}
    whenever $(x,t),(y,t)\in Q$. 
    
    In particular, by Lemma \ref{lemma:Agrco} and \eqref{ae:nonzero}, \begin{equation}\label{approx:A:co}
    \begin{split}
        |\A_\epsilon(x,t,\xi)-\A_\epsilon(x,t,0)|
        &\leq C b_\epsilon(|\xi|),\\
        \big(\A_\epsilon(x,t,\xi)-\A_\epsilon(x,t,0)\big)\cdot\xi
        &\geq c B_\epsilon(|\xi|),
    \end{split}
    \end{equation}
    as well as
    \begin{equation}\label{approx:A:unif}
        c\,\epsilon|\zeta|^2\leq
        \nabla_\xi\A_\epsilon(x,t,\xi)\zeta\cdot\zeta,
        \qquad |\nabla_\xi\A_\epsilon(x,t,\xi)|\leq C\epsilon^{-1}.
    \end{equation}

\end{lemma}

\begin{proof}
    We follow the proof of \cite[Lemma 2.5]{Antonini2027},
    with an additional convolution argument in $(x,t)$. Throughout
    the proof, $c,c',c'',C,C',C''>0$ denote constants depending only
    on $\data$, unless otherwise specified.

    Let $\rho\in C_c^\infty(B_1)$ and
    $\vartheta\in C_c^\infty((-1,1))$ be standard, nonnegative radially symmetric
    mollifiers with unit integral. For $\gamma>0$
    sufficiently small, we define
    \begin{equation}\label{approx:A:gamma}
        \A_\gamma(x,t,\xi)
        \coloneqq\iint_{B_1\times (-1,1)}\rho(y)\vartheta(s)
        \A(x-\gamma y,t-\gamma^2s,\xi)\,dyds,
        \qquad (x,t)\in Q.
    \end{equation}
    
By \eqref{ass:A}, and since the convolution
acts only on the $(x,t)$-variables, we have

\begin{equation}\label{approx:A:gammaell}
    \begin{split}
   \nabla_\xi\A_\gamma(x,t,\xi)\zeta\cdot\zeta
        \geq\l_0\, a(|\xi|)|\zeta|^2,\qquad|\nabla_\xi\A_\gamma(x,t,\xi)|
        \leq C\L_0 \,a(|\xi|),
    \end{split}
    \end{equation}
    for $\xi\neq0$ and $\zeta\in \R^N$. Set
    \[c_\gamma(x,t)\coloneqq \A_\gamma(x,t,0).\]
    In view of Lemma \ref{lemma:Agrco} and the
    properties of convolution, we have
    \begin{equation}\label{approx:A:gr}
        |\A_\gamma(x,t,\xi)-c_\gamma(x,t)|
        \leq C_1 b(|\xi|),
    \end{equation}
    where $C_1=C_1(\data)>0$.
    
    Moreover, by \eqref{approx:A:gamma} and \eqref{A:hold}, it is easy to see that
    \begin{equation}\label{approx:A:gammahold}
        |\A_\gamma(x,t,\xi)-\A_\gamma(y,t,\xi)|
        \leq\Lambda_{\mathrm h}|x-y|^\alpha
        \big(1+b(|\xi|)\big).
    \end{equation}
 Next, we further regularize in the $\xi$-variable via convolution
 \[
  \widehat\A_{\gamma,\e}(x,t,\xi)
        \coloneqq\A_\gamma*_{\xi}\rho_\e(x,t,\xi)=\int_{\R^N} \rho(z)\,\A_\gamma(x,t,\xi-\epsilon\,z)\,dz.
 \]
    Let $K=K(\data)\geq1$ be sufficiently large,
    to be fixed below, and choose
    $0<\delta<e^{-K}$. Then  we take
    $\eta_\delta\in C^\infty([0,\infty))$ such that
    \begin{equation}\label{approx:A:cut}
    \begin{split}
        &0\leq\eta_\delta\leq1,
        \qquad |r\eta_\delta'(r)|\leq K^{-1},\\
        &\eta_\delta\equiv1
        \quad\text{in }[0,\delta]\cup
        [e^{2K}\delta^{-1},+\infty),\\
        &\eta_\delta\equiv0
        \quad\text{in }[e^{2K}\delta,\delta^{-1}].
    \end{split}
    \end{equation}
    Such functions can be constructed by smoothing suitable
    piecewise affine functions of $\log r$-- see the Remark following the proof of \cite[Lemma 2.5]{Antonini2027}.
    By \eqref{ae:unif}, \eqref{be:unif} and the
    positivity of $a$, we may choose $0<\e_\delta<\delta/2$
    so small that, whenever $0<\e<\e_\delta$,
    \begin{equation}\label{approx:A:comparison}
        \frac12 a_\e(r)\leq a(r)\leq2a_\e(r)
        \quad\text{for all  }\delta\leq r
        \leq e^{2K}\delta^{-1}.
    \end{equation}
Next, by \eqref{approx:A:gammaell} and \eqref{approx:A:comparison}, for all $\xi\in \R^N$ such that $\delta\leq|\xi|
    \leq e^{2K}\delta^{-1}$, we have
\begin{equation}\label{un:coer}
    \begin{split}
        \nabla_\xi \widehat\A_{\gamma,\e}(x,t,\xi)\zeta\cdot \zeta
       & \geq \lambda_0\,\left(\int_{B_1}\rho(z)\,a(|\xi-\e z|)\,dz\right)\,|\zeta|^2
        \\
        &\geq c\,\l_0\,a(|\xi|)\,|\zeta|^2\geq c'\,a_\e(|\xi|)\,|\zeta|^2,\qquad\text{for all $\delta\leq|\xi|
    \leq e^{2K}\delta^{-1}$},
    \end{split}
\end{equation}
where in the second inequality we used
$c\,a(|\xi|)\leq a(|\xi-\e z|)\leq C\,a(|\xi|)$, which follows
from \eqref{ultra:utile}, since
$|\xi|/2\leq|\xi-\e z|\leq 2|\xi|$
for $|\xi|\geq\delta$, $\e<\delta/2$,
and $|z|\leq1$.
With a similar reasoning, by \eqref{approx:A:gammaell} and \eqref{approx:A:comparison}, we deduce
\begin{equation}\label{du:growth}
    |\nabla_\xi \widehat{\A}_{\gamma,\e}(x,t,\xi)|\leq C\,a_\e(|\xi|)\quad\text{for all $\delta\leq|\xi|\leq e^{2K}\delta^{-1}$,}
\end{equation}
while from \eqref{approx:A:gr}, \begin{equation}\label{approx:A:hat}
\begin{split}
        |\widehat\A_{\gamma,\e}(x,t,\xi)
        -c_\gamma(x,t)|
        &\leq  \int_{B_1}\rho(z)\left| \A_\gamma(x,t,\xi-\e z)-\A_\gamma(x,t,0)\right|\,dz
        \\
        &\leq  C \int_{B_1}\rho(z)\,b(|\xi-\e z|)\,dz
        \\
        &\leq C'b(|\xi|)\stackrel{\eqref{approx:A:comparison}}{\leq} C'' b_\e(|\xi|)\qquad \text{for all $\delta\leq |\xi| \leq e^{2K}\delta^{-1}$,}
        \end{split}
    \end{equation}
where in the second inequality we used $b(|\xi-\e z|)\leq b(2|\xi|)\leq Cb(|\xi|) $ by the monotonicity of $b$ and \eqref{b2t}.

We now define    \begin{equation}\label{approx:A:def}
    \begin{split}
        \A_{\gamma,\e,\delta}(x,t,\xi)
        &\coloneqq
        \big(1-\eta_\delta(|\xi|)\big)
        \widehat\A_{\gamma,\e}(x,t,\xi)+\eta_\delta(|\xi|)
        \big(c_\gamma(x,t)+a_\e(|\xi|)\xi\big).
    \end{split}
    \end{equation}
    Since $\eta_\delta\equiv1$ near $\xi=0$
    and \eqref{aexi:Cinf} holds, we have
    $\A_{\gamma,\e,\delta}\in
    C^\infty(Q\times\R^N;\R^N)$ and
    \begin{equation}\label{approx:A:zero}
        \A_{\gamma,\e,\delta}(x,t,0)
        =c_\gamma(x,t).
    \end{equation}
Moreover, since $\eta_\delta\equiv1$ for $|\xi|\geq e^{2K}\delta^{-1}$,
 we have
$\nabla_x\A_{\gamma,\e,\delta}(x,t,\xi)=\nabla_xc_\gamma(x,t)$
on this set, and therefore
\begin{equation}\label{yuia}
\sup_{Q'\times \R^N} |\nabla_x\A_{\gamma,\e,\delta}|\leq C_{\gamma,\e,\delta}\quad\text{for all $Q'\Subset Q$.}
\end{equation}  
    We now claim that, for all $\xi,\zeta\in\R^N$,
    \begin{equation}\label{approx:A:struct}
    \begin{split}
        \nabla_\xi\A_{\gamma,\e,\delta}(x,t,\xi)
        \zeta\cdot\zeta
        \geq c\,a_\e(|\xi|)|\zeta|^2,\qquad |\nabla_\xi\A_{\gamma,\e,\delta}(x,t,\xi)|
        \leq C\,a_\e(|\xi|),
    \end{split}
    \end{equation}
    with $c,C>0$ independent of $\gamma,\e,\delta$.
To see this, we differentiate 
    \eqref{approx:A:def}:
    \begin{equation}\label{approx:A:der}
    \begin{split}
        \nabla_\xi\A_{\gamma,\e,\delta}
        &=\big(1-\eta_\delta(|\xi|)\big)\nabla_\xi
        \widehat\A_{\gamma,\e}
        +\eta_\delta(|\xi|)\, \nabla_\xi(a_\e(|\xi|)\xi)\\
        &\quad+
        \big(c_\gamma
        -\widehat\A_{\gamma,\e}+a_\e(|\xi|)\xi\big)
        \otimes \eta'_\delta(|\xi|)\,\frac{\xi}{|\xi|}.
    \end{split}
    \end{equation}
    By \eqref{iaesae} (see also the computations of \cite[pp. 17]{Antonini2027}),
    \begin{equation}\label{cirua}
        \nabla_\xi(a_\e(|\xi|)\xi)\zeta\cdot\zeta
        \geq(1+i_a)
        a_\e(|\xi|)|\zeta|^2,
        \qquad
        |\nabla_\xi(a_\e(|\xi|)\xi)|\leq C a_\e(|\xi|).
    \end{equation}
 Hence, in the set $\{\xi \in \R^N: \eta_\delta'(|\xi|)=0\}$, Equation \eqref{approx:A:struct} follows from \eqref{approx:A:cut}, 
    \eqref{un:coer},  \eqref{du:growth}. On the remaining
    annuli, which are contained in $\delta\leq|\xi|\leq e^{2K}\delta^{-1}$ by   \eqref{approx:A:cut}, Equation \eqref{approx:A:hat} gives
    \[
    \begin{split}
        \big|c_\gamma
        -\widehat\A_{\gamma,\e}+a_\e(|\xi|)\xi\big|
        |\eta'_\delta(|\xi|)|\leq
        \frac{C}{K|\xi|}b_\e(|\xi|)=\frac{C}{K}a_\e(|\xi|).
    \end{split}
    \]
    By choosing $K=K(\data)$ sufficiently
    large, and using   \eqref{un:coer}, \eqref{cirua} into \eqref{approx:A:der} we deduce
    \begin{equation*}
        \nabla_\xi \A_{\gamma,\e,\delta}(x,t,\xi)\zeta\cdot\zeta\geq  c\,a_\e(|\xi|)\,|\zeta|^2-\frac{C}{K}a_\e(|\xi|)|\zeta|^2\geq c'\,a_\e(|\xi|)\,|\zeta|^2.
    \end{equation*}
The bound for
    $|\nabla_\xi\A_{\gamma,\e,\delta}|$ follows
    in the same way, by using \eqref{du:growth}. This proves \eqref{approx:A:struct}.

    Moreover, by \eqref{approx:A:gammahold},
    \[
    \begin{split}
        |\widehat\A_{\gamma,\e}(x,t,\xi)
        -\widehat\A_{\gamma,\e}(y,t,\xi)|\leq\Lambda_{\mathrm h}|x-y|^\alpha
        \int_{B_1}\rho(z)
        \big(1+b(|\xi-\e z|)\big)\,dz.
    \end{split}
    \]
  
If $\eta_\delta(|\xi|)\neq1$, then, using
$b(|\xi-\e z|)\leq C\,b(|\xi|)$ on this set
and \eqref{approx:A:comparison}, we deduce

    \[
        \int_{B_1}\rho(z)
        \big(1+b(|\xi-\e\,z|)\big)\,dz
        \leq C\big(1+b_\e(|\xi|)\big).
    \]
    Since
    $|c_\gamma(x,t)-c_\gamma(y,t)|
    \leq\Lambda_{\mathrm h}|x-y|^\alpha$ by \eqref{approx:A:gammahold},
    and $a_\e$ and $\eta_\delta$ are independent
    of $(x,t)$, \eqref{approx:A:def} and the two estimates above yield
    \begin{equation}\label{approx:A:holderaux}
        |\A_{\gamma,\e,\delta}(x,t,\xi)
        -\A_{\gamma,\e,\delta}(y,t,\xi)|
        \leq C\Lambda_{\mathrm h}|x-y|^\alpha
        \big(1+b_\e(|\xi|)\big).
    \end{equation}

We now prove \eqref{Ae:c}. Fix $M>0$ and
let $\delta>0$ be sufficiently small that
$\delta^{-1}>M$. By \eqref{approx:A:cut},
$\eta_\delta(|\xi|)=0$ whenever
$e^{2K}\delta\leq|\xi|\leq M$.
Thus, by \eqref{approx:A:def},
\eqref{approx:A:gr} and the properties
of convolution, we have
\[
\begin{split}
\sup_{|\xi|\leq M}&
|\A_{\gamma,\e,\delta}(x,t,\xi)-\A(x,t,\xi)|\leq
\sup_{|\xi|\leq M}
|\A_\gamma(x,t,\xi)-\A(x,t,\xi)|\\
&\qquad+
\sup_{|\xi|\leq M}
|\widehat\A_{\gamma,\e}(x,t,\xi)
-\A_\gamma(x,t,\xi)|+
C b(e^{2K}\delta)
+b_\e(e^{2K}\delta).
\end{split}
\]
By the properties of convolution and \eqref{be:unif}, passing to the limit we infer
\[
\lim_{\gamma\to 0}\lim_{\delta\to 0}\lim_{\e\to 0}\A_{\gamma,\e,\delta}(x,t,\xi)=\A(x,t,\xi)\quad\text{ uniformly in $|\xi|\leq M$, for a.e. $(x,t)\in Q$.}
\]
Thereby choosing  $\delta_j\downarrow0$,
    $\e_j\in(0,\e_{\delta_j})$ and
    $\gamma_j\downarrow0$ by a diagonal argument,
    and setting
    $\A_j\coloneqq\A_{\gamma_j,\e_j,\delta_j}$, then $\A_j\in C^\infty(Q\times \R^N;\R^N)$, \eqref{Ae:c} holds, and  \eqref{approx:A:ell}-\eqref{approx:A:hold}
    follow from the estimates \eqref{yuia}, \eqref{approx:A:struct} and \eqref{approx:A:holderaux}, up to relabeling the sequence.
\end{proof}

\subsection{Orlicz Lebesgue and Sobolev spaces}\label{sec:spaces}

Let $\Omega$ be an open subset of $\R^N$.
The Lebesgue-Orlicz space associated with $B$ is defined as
\[
    L^B(\Omega)
    \coloneqq
    \left\{
        u:\Omega\to\R\text{ measurable}:
        \int_\Omega B(|u|)\,dx<\infty
    \right\}.
\]
We endow this space with the \emph{Luxemburg norm}
\begin{equation}\label{lux}
    \|u\|_{L^B(\Omega)}
    \coloneqq
    \inf\left\{
        k>0:
        \int_\Omega B\left(\frac{|u(x)|}{k}\right)\,dx
        \leq1
    \right\}.
\end{equation}
The same definitions apply to vector-valued functions,
with $|u|$ denoting the Euclidean norm.

Since $B$ and its complementary Young function
$\widetilde B$ satisfy the $\Delta_2$-condition,
$L^B(\Omega)$ is a separable and reflexive Banach space;
see \cite[Theorems 3.5.2 and 3.6.6]{HarjulehtoHasto2019}.
Moreover, convergence in $L^B(\Omega)$ is equivalent
to convergence of the modular, that is,
\[
    u_j\to u\quad\text{in }L^B(\Omega)
    \quad\Longleftrightarrow\quad
    \lim_{j\to\infty}
    \int_\Omega B(|u_j-u|)\,dx=0.
\]
The Orlicz-Sobolev space is defined as
\[
    W^{1,B}(\Omega)
    \coloneqq
    \left\{
        u\in L^B(\Omega):
        Du\in L^B(\Omega;\R^N)
    \right\},
\]
where $Du$ denotes the distributional gradient of $u$.
It is a separable and reflexive Banach space when
endowed with the norm
\[
    \|u\|_{W^{1,B}(\Omega)}
    \coloneqq
    \|u\|_{L^B(\Omega)}
    +\|Du\|_{L^B(\Omega)}.
\]
Endowed with this norm, $W^{1,B}(\Omega)$ is a separable
and reflexive Banach space; see
\cite[Theorem 6.1.4(b)--(d)]{HarjulehtoHasto2019}. Accordingly, we define
\[
    W^{1,B}_c(\Omega)
    \coloneqq
    \left\{
        u\in W^{1,B}(\Omega):
        \operatorname{spt}u\Subset\Omega
    \right\},
    \qquad
    W^{1,B}_0(\Omega)
    \coloneqq
    \overline{C_c^\infty(\Omega)}^{\,W^{1,B}(\Omega)}.
\]
If $\Omega$ is a bounded Lipschitz domain,
$W^{1,B}_0(\Omega)$ coincides with the space of functions
in $W^{1,B}(\Omega)$ having zero boundary trace.
The local space $W^{1,B}_{\rm loc}(\Omega)$ consists of
functions belonging to $W^{1,B}(\Omega')$ for every
open set $\Omega'\Subset\Omega$.

Let now $I=(t_1,t_2)$ be a bounded open interval.
For measurable functions on $\Omega\times I$,
the space $L^B(\Omega\times I)$ and its Luxemburg norm
are defined as above, with integration taken with
respect to $dxdt$.
We use the notation
\[
    L^B(I;W^{1,B}(\Omega))
    \coloneqq
    \left\{
        u\in L^B(\Omega\times I):
        D_xu\in L^B(\Omega\times I;\R^N)
    \right\},
\]
endowed with the norm
\[
    \|u\|_{L^B(I;W^{1,B}(\Omega))}
    \coloneqq
    \|u\|_{L^B(\Omega\times I)}
    +\|Du\|_{L^B(\Omega\times I)}.
\]
Here and below, this notation is understood in the
space-time modular sense. In particular,
\[
    u\in L^B(I;W^{1,B}(\Omega))
    \quad\Longleftrightarrow\quad
    \int_I\int_\Omega
    \big[B(|u|)+B(|Du|)\big]\,dxdt<\infty,
\]
where $Du$ is understood in the distributional sense.
The same argument as in the proof of
\cite[Theorem 6.1.4]{HarjulehtoHasto2019}
shows that this space is a separable and reflexive
Banach space. Indeed, the map
\[
    u\mapsto(u,Du)
\]
identifies it with a closed subspace of
$L^B(\Omega\times I)\times L^B(\Omega\times I;\R^N)$.

Similarly, we set
\[
    \begin{split}
    L^B(I;W^{1,B}_0(\Omega))
    \coloneqq
    \big\{
        u\in L^B(I;W^{1,B}(\Omega)):
        u(\cdot,t)\in W^{1,B}_0(\Omega)
        \text{ for a.e. }t\in I
    \big\}.
    \end{split}
\]
Thus, a lateral boundary condition $u=g$ on
$\partial\Omega\times I$ is understood as
\[
    u-g\in L^B(I;W^{1,B}_0(\Omega)).
\]
Local versions of the preceding spaces are defined
by restriction to cylinders compactly contained
in the domain.

The natural energy class considered in this paper is
\[
    L^B(I;W^{1,B}(\Omega))
    \cap L^\infty(I;L^2(\Omega)),
\]
where the second condition means that
\[
    \operatorname*{ess\,sup}_{t\in I}
    \int_\Omega |u(x,t)|^2\,dx<\infty.
\]
We also use $C(\overline I;L^2(\Omega))$ for the space
of functions continuous in time with respect to
the strong $L^2(\Omega)$ topology.
An initial condition $u(\cdot,t_1)=u_0$, with
$u_0\in L^2(\Omega)$, is understood as
\[
    \lim_{t\downarrow t_1}
    \|u(\cdot,t)-u_0\|_{L^2(\Omega)}=0.
\]

When $B(t)=t^p/p$, $1<p<\infty$, the preceding spaces
reduce to the usual Lebesgue and Sobolev spaces,
with equivalent norms.

\begin{definition}[Local weak solution]\label{def:weak:solution}
    We say that $u$ is a \emph{local weak solution}
    to \eqref{eq:1} if
    \[
        u\in
        C_{\rm loc}\big((0,T);L^2_{\rm loc}(\Omega)\big)
        \cap
        L^B_{\rm loc}\big((0,T);W^{1,B}_{\rm loc}(\Omega)\big)
    \]
    and for every $\varphi\in C_c^\infty(\Omega_T)$,
    \[
        -\iint_{\Omega_T}u\,\partial_t\varphi\,dxdt
        +
        \iint_{\Omega_T}
        \A(x,t,Du)\cdot D\varphi\,dxdt
        =
        \iint_{\Omega_T}f\,\varphi\,dxdt.
    \]
\end{definition}

\begin{definition}[Weak solution to the Cauchy--Dirichlet problem]
\label{def:weak:dirichlet}
Let
\[
    u_0\in C([0,T];L^2(\Omega))
    \cap L^B((0,T);W^{1,B}(\Omega)).
\]
We say that $u$ is a \emph{weak solution} to 
\begin{equation*}
    \begin{cases}
        \partial_t u-\mathrm{div}\big( \A(x,t,Du)\big)=f\quad &\text{in $\Omega_T$}
        \\
        u=u_0\quad &\text{on $\partial_{\rm par}\Omega_T$}
    \end{cases}
\end{equation*}
 if it satisfies
the preceding weak formulation,
\[
    u\in C([0,T];L^2(\Omega))
    \cap L^B((0,T);W^{1,B}(\Omega)),
\]
and
\[
    u-u_0\in L^B((0,T);W^{1,B}_0(\Omega)),
    \qquad
    u(\cdot,0)=u_0(\cdot,0)
    \quad\text{in }L^2(\Omega).
\]
The initial condition is understood in the strong
$L^2(\Omega)$ sense.
\end{definition}

\begin{remark}\label{rem:energy}
\rm{
We observe that the growth bound
\eqref{deb:ass:A},  \eqref{wBbt} and \eqref{tB:usef} yield
\[
    \widetilde B\big(|\A(x,t,Du)|\big)
    \leq C_Q\big(1+B(|Du|)\big)
    \qquad\text{a.e. in }Q,
\]
for every cylinder $Q\Subset\Omega_T$.
Consequently,
\[
    \A(x,t,Du)\in
    L^{\widetilde B}_{\rm loc}(\Omega_T;\R^N),
\]
and by H\"older's inequality in Orlicz spaces,
\[
    \iint_Q|\A(x,t,Du)\cdot Du|\,dxdt
    \leq
    2\,\|\A(x,t,Du)\|_{L^{\widetilde B}(Q)}
    \|Du\|_{L^B(Q)}
    <\infty.
\]
In particular, $\A(x,t,Du)\cdot Du\in
L^1_{\rm loc}(\Omega_T)$.}
\end{remark}

\subsection{Technical lemmata}
In this section we recollect a few technical lemmas.  We start with De Giorgi's hypergeometric lemma (see \cite[Lemma 7.1, pp. 220]{G03}).
\begin{lemma}\label{lem:hyp}
    Let $\{Z_j\}$, $j=0,1,2,\dots$, be a sequence of positive numbers satisfying the recursive inequality
    \begin{equation*}
        Z_{j+1}\leq C_0\,b^{j}\,Z_j^{1+\alpha}\,,
    \end{equation*}
    where $C_0,b>1$ and $\alpha>0$ are given numbers. If $Z_0\leq C_0^{-1/\alpha}b^{-1/\alpha^2}$, then
        \begin{equation*}
        \lim_{j\to \infty} Z_j=0\,.
        \end{equation*}
\end{lemma}

We recall the standard iterative lemma \cite[Chapter 6, Lemma 6.1, p. 191]{G03}
\begin{lemma}\label{lem:iteration}
    Let $\phi:[\rho,R]\to[0,\infty)$ be a bounded function. Assume that there exist
    constants $0\leq\theta<1$, $A,B\geq0$, and $\alpha,\beta>0$ such that
    \[
        \phi(t)
        \leq
        \theta\,\phi(s)
        +\frac{A}{(s-t)^\alpha}
        +\frac{B}{(s-t)^\beta}
    \]
    for every $\rho\leq t<s\leq R$. Then there exists a constant
    $C=C(\theta,\alpha,\beta)>0$ such that
    \[
        \phi(\rho)
        \leq
        C\left\{
            \frac{A}{(R-\rho)^\alpha}
            +
            \frac{B}{(R-\rho)^\beta}
        \right\}.
    \]
\end{lemma}

We recall the following isoperimetric type inequality.

\begin{lemma}[\cite{DiBenedetto1993},
    Chapter I, Lemma 2.2]\label{lem:isoper}
Let $v\in W^{1,1}(B_\rho(x_0))\cap C(B_\rho(x_0))$
for some $\rho>0$ and $x_0\in\R^N$, and let
$k$ and $l$ be any pair of real numbers such that
$k<l$. Then there exists a constant
$C=C(N)>0$, such that
\begin{equation}\label{eq:degiorgi:isoperimetric}
    (l-k)|\{v>l\}|
    \leq
    C\,
    \frac{\rho^{N+1}}{|\{v<k\}|}
    \int_{\{k<v<l\}}|Dv|\,dx.
\end{equation}
Here all the level sets are understood to be
subsets of $B_\rho(x_0)$.
\end{lemma}

\section{Boundedness of solutions}\label{sec:bddness}
In this section we study the local boundedness of solutions of \eqref{eq:1}.  We start with the following 

\begin{lemma}[Parabolic Orlicz-Sobolev inequality]
    Let $Q=\mathcal{B}\times I$, where $\mathcal{B}\subset \R^N $ is a ball, and $I\subset \R$ is an open interval; let
    $v\in L^\infty(I;L^2(\mathcal{B}))\cap W^{1,B}(Q)$ be a function compactly supported
    in space for a.e. $t\in I$. Then
    \begin{equation}\label{par:sob}
        \iint_Q |v|^{\frac{2i_B}{N}}\,B(|v|)\,dxdt
        \leq
        C\,
        \left(
        \sup_{t\in I}\int_{\mathcal{B}\times\{t\}}|v|^2\,dx
        \right)^{\frac{i_B}{N}}
        \iint_Q\big(B(|Dv|)+B(|v|)\big)\,dxdt,
    \end{equation}
    for some constant $C=C(N,i_a,s_a)>0$.
\end{lemma}

\begin{proof}
Set $  V\coloneqq B(|v|)^{1/i_B}$.
Then
\[
\begin{split}
    |DV|^{i_B}
    &=
    \left(\frac{1}{i_B}\right)^{i_B}
    B(|v|)^{1-i_B}\,b(|v|)^{i_B}\,|Dv|^{i_B}\stackrel{\eqref{Bcomeb}}{\leq}
    C\,B(|v|)\,
    \left(\frac{|Dv|}{|v|}\right)^{i_B}.
\end{split}
\]
Thanks to \eqref{mon:iBsB}, the map
$t\mapsto B(t)/t^{i_B}$ is monotone increasing. In particular, if
$|Dv|\leq |v|$, then $|DV|^{i_B}\leq C\,B(|v|)$,
whereas if $|Dv|>|v|$, then $\frac{B(|v|)}{|v|^{i_B}}\,|Dv|^{i_B}
    \leq B(|Dv|)$ by monotonicity. In both cases, we have found
\[
    |DV|^{i_B}
    \leq C\,\big(B(|Dv|)+B(|v|)\big).
\]
Then, if $i_B^*=\frac{N i_B}{N-i_B}$, by Sobolev inequality
\[
\begin{split}
    \left(
    \int_\mathcal{B} B(|v|)^{\frac{N}{N-i_B}}\,dx
    \right)^{\frac{N-i_B}{N}}
    &=
    \left(
    \int_\mathcal{B} |V|^{i_B^*}\,dx
    \right)^{\frac{i_B}{i_B^*}}\leq
    C(N,i_B)\int_\mathcal{B} |DV|^{i_B}\,dx
    \\
    &\leq
    C\int_\mathcal{B}\big(B(|Dv|)+B(|v|)\big)\,dx.
\end{split}
\]
Next, by H\"older's inequality with exponents
$\big(\frac{N}{i_B},\frac{N}{N-i_B}\big)$, and the above estimate,
\begin{equation*}
    \begin{split}
        \iint_Q |v|^{\frac{2i_B}{N}}\,B(|v|)\,dxdt
        &\leq
        \int_I
        \left(
        \int_\mathcal{B} |v|^2\,dx
        \right)^{\frac{i_B}{N}}
        \left(
        \int_\mathcal{B}
        B(|v|)^{\frac{N}{N-i_B}}\,dx
        \right)^{\frac{N-i_B}{N}}
        dt
        \\
        &\leq
        C\,
        \left(
        \sup_{t\in I}
        \int_{B\times\{t\}}|v|^2\,dx
        \right)^{\frac{i_B}{N}}
        \iint_Q
        \big(B(|Dv|)+B(|v|)\big)\,dxdt,
    \end{split}
\end{equation*}
that is our thesis.
\end{proof}

\begin{theorem}[Local boundedness]\label{thm:bddness}
    Let $u$ be a local weak solution to \eqref{eq:1} under \eqref{iasa}, \eqref{ass:A}, \eqref{A:hold} and 
    \begin{equation}\label{fdsmall}
        f\in L^d(\Omega_T),\quad d>\frac{N+i_B}{i_B}.
    \end{equation}
    Suppose that \begin{equation}\label{ia:basso}
        i_B>\frac{2N}{N+2}.
    \end{equation}
    Then $u\in L^\infty_{\rm loc}(\Omega_T)$, and for every $Q_{R,S}=Q_{R,S}(x_0,t_0)\Subset\Omega_T$ and every $c\in\R$,
\begin{equation}\label{u:boundedness}
 \|u-c\|_{L^\infty(Q_{R/2,S/4})}
 \le C_{R,S}\,\left[1+\miint_{Q_{R,S}}
 \left\{ |u-c|^2+B(|u-c|) \right\}dxdt
 \right]^{ \frac{i_B}{(N+2)\,i_B-2N}}.
\end{equation}
Here $C_{R,S}$ is a constant depending on $\datah, R,S, \|f\|_{L^d(Q_{R,S})} $.
    \end{theorem}

\begin{proof}
By scaling and translation, we may assume that $c=0$, $(x_0,t_0)=(0,0)$ and $R=S=1$. Indeed, for $(y,s)\in Q_1=B_1\times(-1,0)$, we may define
\[
    v(y,s)\coloneqq u(x_0+Ry,t_0+Ss)-c.
\]
Then $v$ is a local weak solution to
\[
    \partial_s v-\operatorname{div}_y
    \widehat{\A}(y,s,Dv)=\widehat f(y,s)
    \qquad\text{in }Q_1,
\]
where
\[
    \widehat{\A}(y,s,\xi)
    \coloneqq \frac{S}{R}\,
    \A\left(x_0+Ry,t_0+Ss,\frac{\xi}{R}\right),
    \qquad
    \widehat f(y,s)\coloneqq
    S f(x_0+Ry,t_0+Ss).
\]
As
\[
    \nabla_\xi\widehat{\A}(y,s,\xi)
    =
    \frac{S}{R^2}\,
    \nabla_\xi\A
    \left(x_0+Ry,t_0+Ss,\frac{\xi}{R}\right),
\]
and, by \eqref{ciu:a},
\[
    \min\{R^{-i_a},R^{-s_a}\}\,a(t)
    \leq a(t/R)
    \leq
    \max\{R^{-i_a},R^{-s_a}\}\,a(t)
    \qquad\text{for all }t>0.
\]
it follows that $\widehat{\A}$ satisfies
\eqref{ass:A} and \eqref{A:hold} with the same
function $a$ and possibly different constants,
depending also on $R,S$. Moreover,
\[
    \|\widehat f\|_{L^d(Q_1)}
    =
    S^{1-\frac1d}R^{-\frac Nd}
    \|f\|_{L^d(Q_{R,S})}.
\]
Thus, absorbing the factors depending on $R,S$
into $C_{R,S}$, it suffices to prove
\eqref{u:boundedness} when $c=0$,
$(x_0,t_0)=(0,0)$ and $R=S=1$.
In the following, to ease the notation, we write $u,\A,f$ in place
of $v,\widehat{\A},\widehat f$, respectively.
\vspace{0.1cm}

    \noindent\textit{Step 1: Caccioppoli inequality on level sets.} In this step, since $u$ is not weakly differentiable with respect to the time variable, the computations are formal. They can be rigorously justified by means of Steklov averages; see \cite[I.3-(i) and the proof of II.Proposition 3.1]{DiBenedetto1993}.

Let $\zeta\in C^\infty(Q_1;[0,1])$
be such that
\begin{equation*}
    \operatorname{spt}\zeta(\cdot,t)\subset K\Subset B_1\quad\text{for all $t\in (-1,0)$,}\quad \text{and}\quad\zeta(\cdot,-1)=0.
\end{equation*}
 Let $t_0\in (-1,0)$ and $Q_{t_0}=B_1\times (-1,t_0)$; we multiply equation \eqref{eq:1} with $(u-\kappa)_+\,\zeta^{s_B}$ and integrate over $Q_t$, thus getting
    \[
    \iint_{Q_t}  \partial_t u\,(u-\kappa)_+\,\zeta^{s_B}\,dxdt+\iint_{Q_t}  \A(x,t,Du)\cdot D((u-\kappa)_+\zeta^{s_B})=\iint_{Q_t} f\,(u-\kappa)_+\,\zeta^{s_B}\,dxdt.
    \]
   As $\zeta(\cdot,-1)=0$, we have 
   \[
   \begin{split}
   \iint_{Q_{t_0}}\partial_t u\,(u-\kappa)_+\,\zeta^{s_B}dxdt&=\iint_{Q_{t_0}}\frac{1}{2}\partial_t[(u-\kappa)_+^2\,\zeta^{s_B}]\,dxdt-\iint_{Q_{t_0}}\frac{s_B}{2}\,(u-\kappa)^2_+\,\zeta^{s_B-1}\,\partial_t \zeta dxdt
   \\
   &=\frac{1}{2}\int_{B_1\times \{t_0\}}(u-\kappa)_+^2\,\zeta^{s_B}dx-\frac{s_B}{2}\iint_{Q_{t_0}}\,(u-\kappa)^2_+\,\zeta^{s_B-1}\,\partial_t \zeta dxdt.
   \end{split}
   \] 
Next, by \eqref{deb:ass:A}
\[
\begin{split}
    &\A(x,t,Du)\cdot D(u-\kappa)_+\zeta^{s_B}\geq c_1\,B\big(|D(u-\kappa)_+|\big)\,\zeta^{s_B}-C_0\,\zeta^{s_B}\,\mathbbm{1}_{\{u>\kappa\}}
\end{split}
\] 
and using also \eqref{young1}, for all $\delta\in (0,1)$ we have
\[
\begin{split}
    |\A(x,t,Du)&\,(u-\kappa)_+\,\zeta^{s_B-1}\,D\zeta|
    \\
    &\leq C_1\,b\big(|D(u-\kappa)_+|\big)\,(u-\kappa)_+\,\zeta^{s_B-1}\,|D\zeta|+C_0\,(u-\kappa)_+\,\zeta^{s_B-1}\,|D\zeta|
    \\
    &\leq \delta\,\widetilde{B}\big(b(|D(u-\kappa)_+|)\,\zeta^{s_B-1}\big)+C_\delta\,B\left((u-\kappa)_+\,|D\zeta|\right)+C\widetilde{B}(\zeta^{s_B-1})\,\mathbbm{1}_{\{u>\kappa\}}.
\end{split}
\]
% Then, by \eqref{ass:f} and \eqref{young1},
% \[
% \begin{split}
% \iint_{Q_1} |f|\,(u-\kappa)_+\,\zeta^{s_B}dxdt\leq  \delta &\,\iint_{Q_1}  \widetilde{B}\big(b(|D(u-\kappa)_+|)\zeta^{s_B-1} \big)
% \\
% &+C_\delta\iint_{Q_1}  B((u-\kappa)_+\,\zeta)\,dxdt+\iint_{Q_1}  |f_0|\,(u-\kappa)_+\,\zeta^{s_B}\,dxdt.
% \end{split}
% \]
We also observe that, by \eqref{tB:usef} and \eqref{come:BwB}, we have
\[
 \widetilde{B}\big(b(|D(u-\kappa)_+|)\zeta^{s_B-1} \big)\leq C\,B\big(|D(u-\kappa)_+|\big)\,\zeta^{s_B}.
\]
Therefore, by combining the above expressions, by choosing $\delta\in(0,1)$ small enough depending on the data to reabsorb terms, and  taking the supremum over $t_0\in (-1,0)$, we obtain the Caccioppoli inequality
\begin{equation}\label{cac:lev}
\begin{split}
    \sup_{t_0\in(-1,0)}\int_{B_1\times \{t_0\}} &(u-\kappa)_+^2\,\zeta^{s_B}\,dx+\iint_{Q_1} B(|D(u-\kappa)_+|\,\zeta^{s_B/i_B})\,dxdt
    \\
    \leq C &\,\iint_{Q_1}  B\big((u-\kappa)_+(|D\zeta|+\zeta)\big)\,dxdt+C\,|\{u>\kappa\}\cap \mathrm{supp}\,\zeta|
    \\
    &+C\,\iint_{Q_1}  (u-\kappa)_+^2\,|\partial_t \zeta|\,dxdt+C\,\iint_{Q_1}  |f|\,(u-\kappa)_+\,\zeta^{s_B}\,dxdt,
    \end{split}
\end{equation}
where we also used that $0\leq \zeta\leq 1$, \eqref{tB=1}, and that $B(|D(u-\kappa)_+|\,\zeta^{s_B/i_B})\leq B(|D(u-\kappa)_+|)\,\zeta^{s_B}$ by \eqref{B2t}.
\vspace{0.1cm}

\noindent\textit{Step 2: Iteration and conclusion.}
Let $1/2<r_1<r_2<1$, and choose $\zeta$ such that
\begin{equation*}
    \zeta\in C^\infty_c(Q_{r_2};[0,1]),\qquad
    \zeta\equiv 1\quad\text{on $Q_{r_1}$},\qquad
    |D\zeta|\leq \frac{C(N)}{r_2-r_1},\qquad
    |\partial_t\zeta|\leq \frac{C(N)}{(r_2-r_1)^2}.
\end{equation*}
We combine \eqref{cac:lev} with the Sobolev inequality \eqref{par:sob}.
Recalling that $s_B/i_B\geq 1$, and using \eqref{B2t},
\eqref{triangle}, and
$\zeta^{2s_B/i_B}\leq \zeta^{s_B}$ as $i_B\leq 2$, we get
\begin{equation}\label{lev:sob}
    \begin{split}
        \iint_{Q_{r_1}} &(u-\kappa)_+^{\frac{2i_B}{N}}\,
        B\big((u-\kappa)_+\big)\,dxdt
        \\
        &\leq
        \iint_{Q_1}
        (u-\kappa)_+^{\frac{2i_B}{N}}\,
        \zeta^{\frac{2i_B}{N}\frac{s_B}{i_B}}\,
        B\left((u-\kappa)_+\zeta^{\frac{s_B}{i_B}}\right)\,dxdt
        \\
        &\leq C
        \bigg\{
        \iint_{Q_{r_2}}
        B\left(\frac{(u-\kappa)_+}{r_2-r_1}\right)\,dxdt
        +|\{u>\kappa\}\cap Q_{r_2}|
        \\
        &\hspace{2.5cm}
        +\iint_{Q_{r_2}}
        \frac{(u-\kappa)_+^2}{(r_2-r_1)^2}\,dxdt
        +\iint_{Q_{r_2}}
        |f|(u-\kappa)_+\,dxdt
        \bigg\}^{1+\frac{i_B}{N}}.
    \end{split}
\end{equation}

Now, for $j=0,1,\dots$, let
\[
r_j\coloneqq \frac12+2^{-j-1},
\qquad
\kappa_j\coloneqq K(1-2^{-j-1}),
\]
where $K\geq 1$ will be chosen later, and set
\[
Q_j\coloneqq Q_{r_j},
\qquad
w_j\coloneqq (u-\kappa_j)_+,
\qquad
A_j\coloneqq Q_j\cap\{u>\kappa_j\}.
\]
Notice that $r_j-r_{j+1}=2^{-j-2}$ and $\kappa_{j+1}-\kappa_j=K\,2^{-j-2}$. We also set
\begin{equation}\label{def:Yj}
    Y_j\coloneqq
    \iint_{Q_j}\big(w_j^2+B(w_j)\big)\,dxdt.
\end{equation}
For every $j\geq 0$, we further introduce the intermediate level
\[
\widehat\kappa_j\coloneqq\frac{\kappa_j+\kappa_{j+1}}{2},
\qquad
\widehat w_j\coloneqq(u-\widehat\kappa_j)_+,
\qquad
\widehat A_j\coloneqq Q_j\cap\{u>\widehat\kappa_j\}.
\]
Then
\begin{equation}\label{iter:midgap}
    \widehat\kappa_j-\kappa_j
    =
    \kappa_{j+1}-\widehat\kappa_j
    =
    K\,2^{-j-3}.
\end{equation}
Applying \eqref{lev:sob} with
$r_1=r_{j+1}$, $r_2=r_j$, and $\kappa=\widehat\kappa_j$, and using
\eqref{B2t}, we find
\begin{equation}\label{iter:pre}
    \begin{split}
        \iint_{Q_{j+1}}
        \widehat w_j^{\frac{2i_B}{N}}B(\widehat w_j)\,dxdt
        \leq C\,2^{Cj}
        \bigg\{
        Y_j+|\widehat A_j|
        +\iint_{Q_j}|f|\,\widehat w_j\,dxdt
        \bigg\}^{1+\frac{i_B}{N}}.
    \end{split}
\end{equation}
Indeed, since $\widehat w_j\leq w_j$, we have $B\left(\frac{\widehat w_j}{r_j-r_{j+1}}\right)
=
B(2^{j+2}\widehat w_j)
\leq C2^{js_B}B(w_j)$, while $\frac{\widehat w_j^2}{(r_j-r_{j+1})^2}\leq C4^jw_j^2$. We next estimate the measure of $\widehat A_j$. On $\widehat A_j$ we have, by \eqref{iter:midgap}, $w_j\geq
\widehat\kappa_j-\kappa_j
=
K2^{-j-3}$; consequently
\begin{equation}\label{iter:measure}
    |\widehat A_j|
    \leq
    C4^jK^{-2}Y_j.
\end{equation}
We now estimate the term containing $f$. Since $\widehat w_j=0$ outside $\widehat A_j$, by H\"older's inequality,
\begin{equation}\label{iter:f0}
    \begin{split}
        \iint_{Q_j}|f|\widehat w_j\,dxdt&\leq
        \|f\|_{L^d(Q_1)}
        |\widehat A_j|^{\frac12-\frac1d}
        \left(\iint_{\widehat A_j}\widehat w_j^2\,dxdt\right)^{1/2}
        \\
        &\leq
        \|f\|_{L^d(Q_1)}
        |\widehat A_j|^{\frac12-\frac1d}Y_j^{1/2}.
    \end{split}
\end{equation}
Combining \eqref{iter:measure} and \eqref{iter:f0}, we obtain
\begin{equation}\label{iter:f0b}
    \begin{split}
        \iint_{Q_j}|f|\widehat w_j\,dxdt
        &\leq
        C\|f\|_{L^d(Q_1)}
        2^{j(1-\frac2d)}
        K^{-1+\frac2d}
        Y_j^{1-\frac1d}\leq
        C\,2^{Cj}Y_j^{1-\frac1d},
    \end{split}
\end{equation}
where in the last inequality we used $K\geq1$, and the constant $C$ is
allowed to depend on $\|f\|_{L^d(Q_1)}$.

Therefore, from \eqref{iter:pre}, \eqref{iter:measure}, and \eqref{iter:f0b}, we infer that
\begin{equation}\label{iter:sob2}
    \iint_{Q_{j+1}}
    \widehat w_j^{\frac{2i_B}{N}}B(\widehat w_j)\,dxdt
    \leq
    C2^{Cj}
    \left(
    Y_j+Y_j^{1-\frac1d}
    \right)^{1+\frac{i_B}{N}}.
\end{equation}
Now set
\[
\delta_0\coloneqq i_B+\frac{2i_B}{N}-2,
\]
and observe that $0<\delta_0\leq 2i_B/N$ as $2N/(N+2)<i_B=i_a+2\leq 2$ by \eqref{ia:basso}.

We claim that there exists $C\geq1$ such that
\begin{equation}\label{iter:pointwise}
    s^2+B(s)
    \leq
    C\,2^{Cj}\,K^{-\delta_0}
    s^{\frac{2i_B}{N}}B(s)\quad\text{for all $s\geq K2^{-j-3}$.}
\end{equation}
To see this, we use \eqref{B2t} so that, if $s\geq 1$, then $B(s)\geq c\,s^{i_B}$, and hence
\[
\frac{s^{2-\frac{2i_B}{N}}}{B(s)}+s^{-\frac{2i_B}{N}}
\leq C\,s^{-\delta_0}\leq C\,K^{-\delta_0}\,2^{(j+3)\delta_0},
\]
where we also used that $\delta_0\leq 2i_B/N$. If instead $K\,2^{-j-3}\leq s<1$, then $B(s)\geq c\,s^{s_B}$, and so
\[
\begin{split}
\frac{s^{2-\frac{2i_B}{N}}}{B(s)}&+s^{-\frac{2i_B}{N}}
\leq C\,s^{-\left(s_B+\frac{2i_B}{N}-2\right)}+s^{-\frac{2i_B}{N}}
\\
&\leq C\,K^{-\left(s_B+\frac{2i_B}{N}-2\right)}\,2^{(j+3)\left(s_B+\frac{2i_B}{N}-2\right)}
+K^{-\frac{2i_B}{N}}\,2^{(j+3)\frac{2i_B}{N}}\leq C\,2^{Cj}\,K^{-\delta_0},
\end{split}
\]
where in the last inequality we used $K\geq1$, $s_B\geq i_B$, and
\[
s_B+\frac{2i_B}{N}-2\geq i_B+\frac{2i_B}{N}-2=\delta_0,
\qquad
\frac{2i_B}{N}\geq\delta_0.
\]
This proves \eqref{iter:pointwise}.

Next, on $A_{j+1}$ we have, by \eqref{iter:midgap}, $\widehat w_j\geq
\kappa_{j+1}-\widehat\kappa_j
=
K2^{-j-3}$ and $w_{j+1}\leq\widehat w_j$. Hence, by \eqref{iter:pointwise},
\[
\begin{split}
Y_{j+1}
&=
\iint_{Q_{j+1}}
\big(w_{j+1}^2+B(w_{j+1})\big)\,dxdt\leq
C2^{Cj}K^{-\delta_0}
\iint_{Q_{j+1}}
\widehat w_j^{\frac{2i_B}{N}}B(\widehat w_j)\,dxdt.
\end{split}
\]
Combining this estimate with \eqref{iter:sob2}, we obtain
\begin{equation}\label{iter:recursion}
    \begin{split}
        Y_{j+1}
        \leq
        C2^{Cj}K^{-\delta_0}
        \bigg[
        Y_j^{\,1+\frac{i_B}{N}}
        +
        Y_j^{\left(1-\frac1d\right)
        \left(1+\frac{i_B}{N}\right)}
        \bigg].
    \end{split}
\end{equation}
Since by assumption $d>(N+i_B)/i_B$, we have
\begin{equation}\label{suuf:f}
    \left(1-\frac1d\right)
    \left(1+\frac{i_B}{N}\right)>1.
\end{equation}
Therefore, \eqref{iter:recursion} is a superlinear De Giorgi recursion. Finally, let
\[
Y\coloneqq
1+\iint_{Q_1}\big(u_+^2+B(u_+)\big)\,dxdt.
\]
Since $Y_j\leq Y$, by choosing
\begin{equation}\label{iter:choiceK}
    K=
    C\,Y^{\frac{i_B}{(N+2)i_B-2N}},
\end{equation}
with $C$ sufficiently large depending only on the data and on
$\|f\|_{L^d(Q_1)}$, the fast geometric convergence lemma \ref{lem:hyp} yields
\[
Y_j\longrightarrow0.
\]
Since $\kappa_j\uparrow K$ and $Q_j\downarrow Q_{1/2}$, we conclude that
\[
u\leq K\qquad\text{a.e. in }Q_{1/2}.
\]
Repeating the same argument for $-u$, whose equation satisfies the same structural assumptions as \eqref{eq:1}, gives the corresponding lower bound. This proves the theorem.     

\end{proof}

\begin{theorem}[Global boundedness]\label{thm:global:bddness}
    Let $\Omega\subset\R^N$ be a bounded
    domain, and let $u$
    be a weak solution to the Cauchy-Dirichlet problem \begin{equation}\label{test:bdr}
        \begin{cases}
            \partial_t u-\mathrm{div}\big(\A(x,t,Du) \big)=f\quad &\text{in $\Omega_T$}
            \\
            u=u_0\quad&\text{on $\partial_{\rm par}\Omega_T$.}
        \end{cases}
    \end{equation}
    under the assumptions
    \eqref{iasa}, \eqref{ass:A}, \eqref{A:hold},  \eqref{fdsmall},  
    and 
    $u_0\in L^B(0,T;W^{1,B}(\Omega))\cap L^\infty(\Omega_T)$.
Then $u\in L^\infty(\Omega_T)$, with estimate
    \begin{equation}\label{global:u:boundedness}
        \|u\|_{L^\infty(\Omega_T)}
        \leq
        \|u_0\|_{L^\infty(\Omega_T)}
        +C\left(
        1+\|f\|_{L^d(\Omega_T)}^{
        \frac{N+i_B}{i_B(N+1)-N}}
        \right),
    \end{equation}
    where $C>0$ depends only on
    $\datah$, $\operatorname{diam}\Omega$ and $T$.
\end{theorem}

\begin{proof}

Throughout the proof, $c,C>0$ denote
constants depending only on the $\data_{\rm h}$,
$\operatorname{diam}\Omega$ and $T$, unless otherwise specified.

\vspace{0.1cm}

\noindent\textit{Step 1: Global Caccioppoli
inequality on level sets.}
For every $k\geq \|u_0\|_{L^\infty(\Omega_T)}$,
since $u=u_0$ on $\partial_{\rm par}\Omega_T$,
we have
\[
    (u-k)_+\in L^B(0,T;W^{1,B}_0(\Omega)),
    \qquad (u-k)_+(\cdot,0)=0.
\]
Thus, testing \eqref{test:bdr} with $(u-k)_+$
and integrating over $\Omega\times(0,\tau)$,
for $\tau\in(0,T)$, we obtain
\[
\begin{split}
    \frac12\int_\Omega
    (u-k)_+^2(x,\tau)\,dx
    +\iint_{\Omega\times(0,\tau)}
    \A(x,t,Du)\cdot D(u-k)_+\,dxdt
    =
    \iint_{\Omega\times(0,\tau)}
    f\,(u-k)_+\,dxdt.
\end{split}
\]
As usual, these computations can be justified
by means of Steklov averages.

Since $D(u-k)_+=Du$ a.e. on $\{u>k\}$
and $D(u-k)_+=0$ a.e. outside $\{u>k\}$,
using \eqref{ovu:Linf} and
taking the supremum over $\tau\in(0,T)$
gives
\begin{equation}\label{global:cac}
\begin{split}
    &\operatorname*{ess\,sup}_{0<t<T}
    \int_\Omega (u-k)_+^2(x,t)\,dx
    +\iint_{\Omega_T}B(|D(u-k)_+|)\,dxdt\\
    &\qquad\leq
    C\iint_{\{u>k\}}|f|(u-k)_+\,dxdt
    +C|\{u>k\}|.
\end{split}
\end{equation}

By \eqref{usso}, and since $D(u-k)_+=0$
outside $\{u>k\}$, \eqref{global:cac}
yields
\begin{equation}\label{global:energy}
\begin{split}
    &\operatorname*{ess\,sup}_{0<t<T}
    \int_\Omega (u-k)_+^2(x,t)\,dx
    +\iint_{\Omega_T}|D(u-k)_+|^{i_B}\,dxdt\\
    &\qquad\leq
    C\iint_{\{u>k\}}|f|(u-k)_+\,dxdt
    +C|\{u>k\}|.
\end{split}
\end{equation}

    \vspace{0.1cm}

\noindent\textit{Step 2: Parabolic Sobolev
inequality and the source term.}
Set
\[
\begin{split}
    q&\coloneqq\frac{i_B(N+2)}{N},\\
    \mathcal E_k&\coloneqq
    \operatorname*{ess\,sup}_{0<t<T}
    \int_\Omega (u-k)_+^2(x,t)\,dx
    +\iint_{\Omega_T}|D(u-k)_+|^{i_B}\,dxdt.
\end{split}
\]
Since $(u-k)_+$ has zero spatial trace,
the parabolic Sobolev inequality
\cite[Chapter I, Proposition 3.1]{DiBenedetto1993} gives
\begin{equation}\label{cka}
\begin{split}
    \iint_{\Omega_T}(u-k)_+^q\,dxdt
    &\leq
    C\left(
    \operatorname*{ess\,sup}_{0<t<T}
    \int_\Omega (u-k)_+^2\,dx
    \right)^{i_B/N}
    \iint_{\Omega_T}|D(u-k)_+|^{i_B}\,dxdt\\
    &\leq C\mathcal E_k^{1+i_B/N}.
\end{split}
\end{equation}

By H\"older's inequality,
\[
\begin{split}
    \iint_{\{u>k\}}|f|(u-k)_+\,dxdt
    &\leq
    \|f\|_{L^d(\Omega_T)}
    \left(
    \iint_{\Omega_T}(u-k)_+^q\,dxdt
    \right)^{1/q}
    |\{u>k\}|^{1-\frac1d-\frac1q}\\
    &\leq
    C\|f\|_{L^d(\Omega_T)}\,
    \mathcal E_k^\mu\,|\{u>k\}|^\nu,
\end{split}
\]
where
\[
    \mu\coloneqq
    \frac{N+i_B}{i_B(N+2)},
    \qquad
    \nu\coloneqq
    1-\frac1d-\frac{N}{i_B(N+2)}.
\]
Since $i_B>1$ and $d>(N+i_B)/i_B$,
we have $0<\mu<1$ and $\nu>0$.

Inserting the preceding estimate into
\eqref{global:energy}, we obtain
\[
    \mathcal E_k
    \leq
    C\|f\|_{L^d(\Omega_T)}\,
    \mathcal E_k^\mu\,|\{u>k\}|^\nu
    +C|\{u>k\}|.
\]
Young's inequality with conjugate exponents
$1/\mu$ and $1/(1-\mu)$ then yields
\begin{equation}\label{global:energy2}
    \mathcal E_k
    \leq
    C\|f\|_{L^d(\Omega_T)}^{\frac1{1-\mu}}
    |\{u>k\}|^{\frac{\nu}{1-\mu}}
    +C|\{u>k\}|.
\end{equation}

    \vspace{0.1cm}

    \noindent\textit{Step 3: De Giorgi iteration.}
    For $j=0,1,\dots$, let
    \[
        k_j\coloneqq \|u_0\|_{L^\infty(\Omega_T)}+K(1-2^{-j}),
        \qquad
        A_j\coloneqq
        \Omega_T\cap\{u>k_j\},
        \qquad
        Y_j\coloneqq|A_j|,
    \]
    where $K\geq1$ will be chosen later.

    Since $k_{j+1}-k_j=K2^{-j-1}$,
    we have
    \[
        (u-k_j)_+
        \geq K2^{-j-1}
        \qquad\text{on }A_{j+1}.
    \]
    Therefore, by \eqref{cka} and \eqref{global:energy2},
    \[
    \begin{split}
        Y_{j+1}
        &\leq
        \frac{2^{q(j+1)}}{K^q}
        \iint_{\Omega_T}(u-k_j)_+^q\,dxdt\leq
        \frac{C2^{qj}}{K^q}
        \mathcal E_{k_j}^{1+i_B/N}\\
        &\leq
        \frac{C2^{qj}}{K^q}
        \bigg[
        Y_j^{1+\frac{i_B}{N}}
        +
        \|f\|_{L^d(\Omega_T)}^{
        \frac{1+i_B/N}{1-\mu}}
        Y_j^{
        \frac{(1+i_B/N)\nu}{1-\mu}}
        \bigg].
    \end{split}
    \]
   Observe that both exponents of $Y_j$
    are strictly greater than one, since
    \[
    \begin{split}
        \frac{(1+i_B/N)\nu}{1-\mu}-1\quad=
        \frac{i_B(N+2)}
        {N[i_B(N+1)-N]}
        \left(
        i_B-\frac{N+i_B}{d}
        \right)>0,
    \end{split}
    \]
    where the strict inequality follows
 from the assumption $d> (N+i_B)/i_B$.

    Consequently, setting
    \[
        \sigma\coloneqq
        \min\left\{
        \frac{i_B}{N},\,
        \frac{(1+i_B/N)\nu}{1-\mu}-1
        \right\}>0,
    \]
    and recalling that
    $Y_j\leq|\Omega_T|$, we obtain
    \begin{equation}\label{global:iteration}
        Y_{j+1}
        \leq
        \frac{C2^{qj}}{K^q}
        \left(
        1+
        \|f\|_{L^d(\Omega_T)}^{
        \frac{1+i_B/N}{1-\mu}}
        \right)
        Y_j^{1+\sigma}.
    \end{equation}

    Notice that $|Y_0|\leq |\Omega_T|$.
    Thus, by choosing
    \[
        K=C\left(
        1+\|f\|_{L^d(\Omega_T)}^{
        \frac{1+i_B/N}{q(1-\mu)}}
        \right)
    \]
    with $C$ sufficiently large
    depending only on the $\datah$,
    $\operatorname{diam}\Omega$ and $T$, the fast geometric
    convergence lemma \ref{lem:hyp}
    applied to \eqref{global:iteration}
    gives
    \[
        Y_j\longrightarrow0.
    \]
    Since
    \[
        \frac{1+i_B/N}{q(1-\mu)}
        =
        \frac{N+i_B}{i_B(N+1)-N},
    \]
    and $k_j\uparrow \|u_0\|_{L^\infty(\Omega_T)}+K$, we obtain
    \[
        u\leq \|u_0\|_{L^\infty(\Omega_T)}+
        C\left(
        1+\|f\|_{L^d(\Omega_T)}^{
        \frac{N+i_B}{i_B(N+1)-N}}
        \right)
        \qquad\text{a.e. in }\Omega_T.
    \]

    Finally, applying the same argument
    to $-u$, whose equation satisfies
    the same structural assumptions,
    gives the corresponding lower bound.
    This proves \eqref{global:u:boundedness}.
\end{proof}

\section{Gradient regularity of homogeneous problems}\label{sec:homogeneous}
In this section we study solutions $v $ to the homogeneous parabolic equation
\begin{equation}\label{eq:hom}
    \partial_t v-\mathrm{div}(\A_0(t,Dv))=0\quad \text{in $\Omega_T$.}
\end{equation}
Here we assume that $\A_0:(0,T)\times \R^N\to \R^N$ is a Caratheodory function such that 
\[
\text{$\xi\mapsto\A_0(t,\xi)$ is of class $C^1(\R^N\sm \{0\})$ for almost every $t\in (0,T)$,}
\]
 and it satisfies
\begin{equation}\label{ass:Ahom}
        \begin{split}
&\sum_{i,j=1}^N\left|\frac{\partial \A_0^i }{\partial \xi_j}(t,\xi)\right|\leq \L_0\,a\big(|\xi|\big),
\qquad\sum_{i,j=1}^N \frac{\partial \A^i_0 }{\partial \xi_j}(t,\xi)\,\eta_i\,\eta_j\geq \l_0\,a\big(|\xi|\big)\,|\eta|^2
    \end{split}
\end{equation}
for almost every $t\in (0,T)$, for all $\xi\in \R^N\sm\{0\}$ and all $\eta\in \R^N$.

\subsection{Spatial gradient boundedness}\label{subse:grbd}
The main theorem of this section is the following.
\begin{theorem}\label{thm:grbddhom}
    Let $v$ be a local weak solution to \eqref{eq:hom} with \eqref{iasa} and \eqref{ass:Ahom} in force. 
Then there exists a constant $C=C(\data)>0$ such that, for all $Q_{R,S}=Q_{R,S}(x_0,t_0)\Subset \Omega_T $, the following hold:
\begin{itemize}
    \item[(i)]  if  $i_a>-\frac{4}{N+2}$, then by setting 
    \begin{equation}\label{dufe}
    F(s)\coloneqq s^2a(s)^{(N+2)/2}\quad\text{and}\quad E\coloneqq\miint_{Q_{R,S}}B(|Dv|)\,dxdt,
    \end{equation}
    we have the $L^\infty$ gradient estimate
\begin{equation}\label{gr:bd1}
 \|Dv\|_{L^\infty(Q_{R/2,S/4})}
 \le C\left\{\left(\frac S{R^2}E\right)^{1/2}
 +F^{-1}\left(\left(\frac{R^2}{S}\right)^{N/2}E\right)\right\},
\end{equation}
where $F^{-1}$ is the inverse function of $F$;

\item[(ii)]  if $v\in L^\infty(Q)$, then
\begin{equation}\label{gr:bd2}
        \|Dv\|_{L^\infty(Q_{R/2,S/4})}\leq C\,\left\{\left(\frac{\operatorname{osc}_{Q_{R,S}}v}{R}\right)+B^{-1}\left(\frac{(\operatorname{osc}_{Q_{R,S}}v)^2}{S}\right) \right\}.
    \end{equation}
\end{itemize}
\end{theorem}

Throughout this section, we shall work under the a priori regularity assumption:
\begin{equation}\tag{S}\label{smooth}
    \text{the function $\A_0(t,\xi)$, and the solution $v(x,t)$ to \eqref{eq:hom} are smooth.}
\end{equation}
This assumption will allow us to perform all the necessary computations; it will be removed in Section  \ref{sec:approximation} by means of a suitable approximation procedure.

We start by proving a Caccioppoli inequality for $v$.

\begin{proposition}\label{prop:cac:osc}
    Let $v$ be a bounded local weak solution to \eqref{eq:hom} with \eqref{iasa}, \eqref{ass:Ahom} and \eqref{smooth} in force. Then for all $ Q_{R,S}=Q_{R,S}(x_0,t_0)\Subset \Omega_T$, we have
    \begin{equation}\label{cac:osc}
        \miint_{Q_{R/2,S/4}} B(|Dv|)\,dxdt\leq C\left[B\left(\frac{\omega}{R}\right)+\frac{\omega^2}{S} \right],\quad \omega\coloneqq \operatorname*{osc}_{Q_{R,S}}v,
    \end{equation}
    with $C=C(\data)>0$.
\end{proposition}

\begin{proof}
By translation, we may assume $(x_0,t_0)=(0,0)$. Let $c\in \R$ and Let $\zeta\in C^\infty(Q_{R,S};[0,1])$ be a cutoff function
with compact spatial support in $B_R$, vanishing at $t=-S$, and such that
\[
    \zeta\equiv1\quad\text{on }Q_{R/2,S/4},
    \qquad |D\zeta|\leq\frac{C(N)}R,
    \qquad |\partial_t\zeta|\leq\frac{C(N)}S.
\]
We multiply \eqref{eq:hom} with $(v-c)\,\zeta^{s_B}$ and integrate over $Q_{R,t}=B_R\times (-S,t)$, and find
\begin{equation*}
    \frac{1}{2}\iint_{Q_{R,t}} \partial_t (v-c)^2\,\zeta^{s_B}\,dxdt+\iint_{Q_{R,t}} \A_0(t,Dv)\cdot \left[ Dv\,\zeta^{s_B}+s_B\,(v-c)\,\zeta^{s_B-1}\,D\zeta \right]\,dxdt=0.
\end{equation*}
  Integrating by parts, and since $\zeta(\cdot,-S)=0$,
  \[
   \frac{1}{2}\iint_{Q_{R,t}} \partial_t (v-c)^2\,\zeta^{s_B}\,dxdt=\frac{1}{2}\int_{B_R\times \{t\}} (v-c)^2\,\zeta^{s_B}\,dx-\frac{s_B}{2}\iint_{Q_{R,t}}(v-c)^2\,\zeta^{s_B-1}\,\partial_t\zeta\,dxdt.
  \]
By \eqref{deb:ass:A} (with $\A(x,t,\xi)=\A_0(t,\xi)$ and $\Lambda_{\mathrm{h}}=0$),
\[
\A_0(t,Dv)\cdot Dv\,\zeta^{s_B}\geq c_1\,B(|Dv|)\zeta^{s_B},
\]
  and using also \eqref{young1}, \eqref{tB:usef} and \eqref{wBbt},
  \[
  \begin{split}
      |\A_0(t,Dv)\,(v-c)\,\zeta^{s_B-1}\cdot D\zeta|&\leq C\,b(|Dv|)\zeta^{s_B-1}\,|(v-c)|\,|D\zeta|
      \\
      &\leq \delta\,\widetilde{B}\left(b(|Dv|)\zeta^{s_B-1} \right)+C_\delta\,B\big(|(v-c)|\,|D\zeta| \big)
      \\
      &\leq C\,\delta\,B(|Dv|)\,\zeta^{s_B}+C_\delta\,B\big(|(v-c)|\,|D\zeta| \big).
  \end{split}
  \]
Combining all of the above expressions, and choosing $\delta$ small enough depending on $\data$ to reabsorb terms, we thus get
\begin{equation*}
    \begin{split}
        \frac{1}{S}\sup_{t\in (-S/4,0
        )}&\mint_{B_{R/2}\times \{t\}} (v-c)^2\,dx+\miint_{Q_{R/2,S/4}} B(|Dv|)\,dxdt
        \\
        &\leq C\,\miint_{Q_{R,S}} B\left(\frac{|v-c|}{R}\right)\,dxdt+C\,\miint_{Q_{R,S}} \frac{(v-c)^2}{S}\,dxdt.
    \end{split}
\end{equation*}
Choosing $c=\inf_{Q_{R,S}}v$, and using that $|v-c|\leq \omega$ and the monotonicity of $B$ yield \eqref{cac:osc}
\end{proof}

Next we establish a conditional pointwise gradient estimate
on intrinsic cylinders.

\begin{proposition}\label{prop:temp:grbd}
    Let $v$ be a local solution to \eqref{eq:hom} with \eqref{iasa}, \eqref{ass:Ahom} and \eqref{smooth} in force. Suppose that 
    \begin{equation}\label{ass:lambda}
    Q_{R}^{(\l)}(x_0,t_0)\Subset \Omega_T\quad\text{and}\quad |Dv|\leq 4\,\lambda \quad\text{on $Q_{R}^{(\l)}(x_0,t_0)$,}
    \end{equation}
    where $Q_{R}^{(\l)}(x_0,t_0)$ is defined by \eqref{def:cilintr}. Then
    \begin{equation}\label{first:bd}
        |Dv|(x_0,t_0)\leq \lambda+C\,\left(\frac{1}{a(\l)}\, \miint_{Q^{(\l)}_R(x_0,t_0)} B(|Dv|)\,dxdt\right)^{1/2}.
    \end{equation}
    with $C>0$ depending on $\data$.
\end{proposition}

\begin{proof}
For the sake of exposition, we divide the proof into a few steps.
\vspace{0.1cm}

\noindent\textit{Step 1: $|Dv|^2$ is a subsolution.} We differentiate Equation \eqref{eq:hom} with respect to $x_k$, we apply the chain rule $D_k\A_0(t,Dv)=\nabla_\xi \A_0(t,Dv)\,D(D_k v)$ and we multiply the resulting equation by $D_k v$. Summing over $k=1,\dots,N$ yields
\[
\begin{split}
0&=\sum_{k=1}^N \partial_t (D_kv)\,D_k v-\sum_{k=1}^N\mathrm{div}\big(\nabla_\xi \A_0(t,Dv)\,D(D_kv)\big)\,D_kv
\\
&=\frac{1}{2}\partial_t |Dv|^2-\frac{1}{2}\mathrm{div}\big(\nabla_\xi \A_0(t,Dv)\,D(|Dv|^2) \big)+\sum_{k=1}^N \nabla_\xi \A_0(t,Dv)\,D(D_kv)\cdot D(D_kv)
\end{split}
\]
Owing to \eqref{ass:Ahom},
\[
\sum_{k=1}^N \nabla_\xi \A_0(t,Dv)\,D(D_kv)\cdot D(D_kv)\geq c\,a(|Dv|)\,|D^2v|^2\geq 0,
\]
so coupling the two expression above, we deduce that $h\coloneqq |Dv|^2$ is a nonnegative subsolution to
\begin{equation}\label{subsol}
    \partial_t h-\mathrm{div}(\nabla_\xi \A_0(t,Dv)\,Dh)\leq 0\quad\text{in $\Omega_T$.}
\end{equation}

\noindent\textit{Step 2: Scaling and conclusion.}
By \eqref{subsol} and the stability of weak subsolutions under nondecreasing convex truncations, the function
$H\coloneqq (h-\lambda^2)_+$ is still a nonnegative weak subsolution to
\begin{equation}\label{subsol:2}
    \partial_t H-\mathrm{div}\big(\nabla_\xi\A_0(t,Dv)\,DH\big)\leq 0
    \quad\text{in $Q^{(\l)}_R(x_0,t_0)$.}
\end{equation}
See, for instance, \cite[Chapter 1, Lemma 1.1]{DiBenedettoGianazzaVespri2012}.
Now, by \eqref{ass:lambda}, we have
\begin{equation*}
    \lambda\leq |Dv|\leq 4\,\lambda
    \quad\text{a.e. on $Q^{(\l)}_R(x_0,t_0)\cap\{H>0\}$.}
\end{equation*}
Hence, by \eqref{ultra:utile},
\begin{equation}\label{soa}
    c\,a(\lambda)\leq a(|Dv|)\leq C\,a(\lambda)
    \quad\text{a.e. on $Q^{(\l)}_R(x_0,t_0)\cap\{H>0\}$,}
\end{equation}
for some positive constants $0<c\leq 1 \leq  C$ depending only on $i_a,s_a$.

Since $DH=0$ a.e. on $\{H=0\}$, the value of the coefficient matrix on this set is immaterial for the equation satisfied by $H$. We may therefore redefine it there in any convenient way. In particular, we set
\begin{equation*}
    M(x,t)\coloneqq
    \begin{cases}
        \nabla_\xi\A_0\big(t,Dv(x,t)\big)
        &\text{on $Q^{(\l)}_R(x_0,t_0)\cap\{H>0\}$,}
        \\
        a(\lambda)\,\mathrm{Id}
        &\text{on $Q^{(\l)}_R(x_0,t_0)\cap\{H=0\}$.}
    \end{cases}
\end{equation*}
With this choice,
\[
    M\,DH=\nabla_\xi\A_0(t,Dv)\,DH
    \quad\text{a.e. in $Q^{(\l)}_R(x_0,t_0)$,}
\]
and therefore $H$ is a nonnegative weak subsolution to
\begin{equation}\label{subsol:t}
    \partial_t H-\mathrm{div}\big(M(x,t)\,DH\big)\leq 0
    \quad\text{in $Q^{(\l)}_R(x_0,t_0)$.}
\end{equation}
Moreover, by \eqref{ass:Ahom}, for every $\xi\in\R^N$ we have
\[
    c\,a(\lambda)|\xi|^2
    \leq M(x,t)\xi\cdot\xi,
    \qquad
    |M(x,t)|
    \leq C\,a(\lambda)
\]
for a.e. $(x,t)\in Q^{(\l)}_R(x_0,t_0)$. We now apply a scaling argument. Set
\[
    y\coloneqq \frac{x-x_0}{R},
    \qquad
    \tau\coloneqq \frac{a(\lambda)}{R^2}(t-t_0),
\]
and define
\[
    \widetilde H(y,\tau)\coloneqq
    H\left(x_0+Ry,t_0+\frac{R^2}{a(\lambda)}\tau\right),
\]
together with
\[
    \widetilde M(y,\tau)\coloneqq
    \frac{1}{a(\lambda)}
    M\left(x_0+Ry,t_0+\frac{R^2}{a(\lambda)}\tau\right).
\]
Then by \eqref{subsol:t}, we have that $\widetilde H$ is a nonnegative weak subsolution to
\begin{equation}\label{subsol:scaled}
    \partial_\tau \widetilde H
    -\mathrm{div}_y\big(\widetilde M(y,\tau)\,D_y\widetilde H\big)
    \leq 0
    \quad\text{in $Q_1$.}
\end{equation}
Moreover, by the ellipticity estimates obtained above, for every $\xi\in\R^N$ we have
\[
    c|\xi|^2
    \leq
    \widetilde M(y,\tau)\xi\cdot\xi,
    \qquad
    |\widetilde M(y,\tau)|
    \leq
    C
\]
for a.e. $(y,\tau)\in Q_1$. Hence \eqref{subsol:scaled} is a uniformly parabolic equation, with ellipticity constants depending only on the $\data$.

We can therefore apply the standard $L^\infty$ estimate for nonnegative weak subsolutions to uniformly parabolic equations, obtaining
\begin{equation}\label{linfty:scaled}
    \operatorname*{sup}_{Q_{1/2}}\widetilde H
    \leq
    C\left(
    \miint_{Q_1}\widetilde H\,dyd\tau
    \right),
\end{equation}
where $C=C(N,i_a,s_a)>0$. See, e.g.,  \cite[Chapter VI, Theorem 6.17]{Lieberman1996}, \cite[Chapter 12, Theorem 2.1]{DiBenedettoGianazza2023}.
Scaling back to the original variables, \eqref{linfty:scaled} yields
\begin{equation}\label{linfty:H}
    \operatorname*{sup}_{Q^{(\l)}_{R/2}(x_0,t_0)}H
    \leq
    C\left(
    \miint_{Q^{(\l)}_R(x_0,t_0)}H\,dxdt
    \right).
\end{equation}
Recalling that $H=(|Dv|^2-\lambda^2)_+$, and using \eqref{soa} and that $a(t)\,t^2\leq C\,B(t)$ by \eqref{Bcomeb}, we conclude that
\begin{equation*}
\begin{split}
    \operatorname*{ess\,sup}_{Q^{(\l)}_{R/2}(x_0,t_0)}
    \big(|Dv|^2-\lambda^2\big)_+
   & \leq
    C\left(
    \miint_{Q^{(\l)}_R(x_0,t_0)}
    \big(|Dv|^2-\lambda^2\big)_+\,dxdt
    \right)
    \\
    &\leq C'\,\left(\frac{1}{a(\l)}
    \miint_{Q^{(\l)}_R(x_0,t_0)}
    a(|Dv|)\,\big(|Dv|^2-\lambda^2\big)_+\,dxdt
    \right)
    \\
    &\leq C''\,\left(\frac{1}{a(\l)}
    \miint_{Q^{(\l)}_R(x_0,t_0)}
    B(|Dv|)\,dxdt
    \right).
    \end{split}
\end{equation*}
This implies \eqref{first:bd}, that is our thesis.
\end{proof}

% \begin{remark}\label{rem:altri:cil}
%     \rm{Estimate \eqref{first:bd} remains valid if the intrinsic cylinders
%     $Q^{(\lambda)}_R(x_0,t_0)$ are replaced by
%     \[
%         \mathcal{Q}^{(\lambda)}_r(x_0,t_0)
%         \coloneqq
%         B_{r\sqrt{a(\lambda)}}(x_0)\times(t_0-r^2,t_0).
%     \]
%     Indeed, the two families of cylinders differ only by a reparametrization of
%     the radius, since
%     \[
%         \mathcal{Q}^{(\lambda)}_r(x_0,t_0)
%         =
%         Q^{(\lambda)}_{r\sqrt{a(\lambda)}}(x_0,t_0).
%     \]
%     Thus, the claim follows simply by replacing
%     $R$ with $r\sqrt{a(\lambda)}$ in \eqref{first:bd}.}
% \end{remark}

We are now in the position to prove Theorem \ref{thm:grbddhom} under the regularity assumption \eqref{smooth}.

\begin{proof}[Proof of Theorem \ref{thm:grbddhom} (i) under \eqref{smooth}] \textit{Step 1: fixing the parameters.} Let $Q_{R,S}=Q_{R,S}(x_0,t_0)\Subset \Omega_T$; by translation, we may assume $(x_0,t_0)=(0,0)$. Fix
\begin{equation}\label{let:parameter}
    \frac{1}{2}\leq \sigma<\eta\leq \frac{3}{4},\quad \mathcal{Q}_\sigma\coloneqq Q_{\sigma R, \sigma^2 S},\quad M(\eta)\coloneqq \|Dv\|_{L^\infty(\mathcal{Q}_\eta)}.
\end{equation}
   We may assume $M(\eta)>0$, otherwise there is nothing to prove, and we set
   \begin{equation}\label{l:set}
   \lambda\coloneqq \frac{1}{4}M(\eta).
   \end{equation}
   Hence $|Dv|\leq 4\lambda$ throughout $\mathcal{Q}_\eta$.  We also put
   \begin{equation}\label{put}
   r\coloneqq \frac{\eta-\sigma}{4}R,\qquad \theta_\lambda\coloneqq \frac{S}{R^2}\,a(\lambda).
   \end{equation}
   For $(x_1,t_1)\in \mathcal{Q}_\sigma$, we observe that, since
\[
\frac{S}{R^2}\min\{1,\theta_\lambda^{-1}\}
=
\frac{\min\{1,\theta_\lambda\}}{a(\lambda)},
\]
we have
   
\[
Q^{(\lambda)}_{r\min\{1,\sqrt{\theta_\lambda}\}}(x_1,t_1)=B_{r\min\{1,\sqrt{\theta_\lambda}\}}(x_1)
\times
\left(
t_1-\frac{S}{R^2}\,r^2
\min\{1,\theta_\lambda^{-1}\},
\,t_1
\right),
\]
and by the definition of $r$, we also have
\[
Q^{(\lambda)}_{r\min\{1,\sqrt{\theta_\lambda}\}}(x_1,t_1)\subset Q_{r,\frac{S}{R^2}r^2}(x_1,t_1)\subset \mathcal{Q}_{\eta},
\]
\[
\frac{|Q_{r,\frac{S}{R^2}r^2}(x_1,t_1)|}{|Q^{(\lambda)}_{r\min\{1,\sqrt{\theta_\lambda}\}}(x_1,t_1)|}=\frac{\max\{1,\theta_\lambda \}}{\min\{1,\theta_\lambda\}^{N/2}},
\]
   and 
\[
\frac{|Q_{R,S}|}{|Q_{r,\frac{S}{R^2}r^2}(x_1,t_1)|}=\left(\frac{R}{r}\right)^{N+2}=\frac{4^{N+2} }{(\eta-\sigma)^{N+2}}.
\]
   Therefore we may apply the result of Proposition \ref{prop:temp:grbd}, which together with the expressions above yields
\begin{equation*}
\begin{split}
        |Dv|(x_1,t_1) &\leq \lambda+C\,\left(\frac{1}{a(\l)}\, \miint_{Q^{(\lambda)}_{r\min\{1,\sqrt{\theta_\lambda}\}}(x_1,t_1)} B(|Dv|)\,dxdt\right)^{1/2}
        \\
        &\leq \lambda+C\,\left(\frac{1}{a(\l)}\, \frac{\max\{1,\theta_\lambda\}}{\min\{1,\theta_\lambda\}^{N/2}}\,\miint_{Q_{r,\frac{S}{R^2}r^2}(x_1,t_1)} B(|Dv|)\,dxdt\right)^{1/2}
        \\
        &\leq \lambda+\frac{C'}{(\eta-\sigma)^{\frac{(N+2)}{2}}}\,\left(\frac{1}{a(\l)}\, \frac{\max\{1,\theta_\lambda\}}{\min\{1,\theta_\l^{N/2}\}}\miint_{Q_{R,S}} B(|Dv|)\,dxdt\right)^{1/2}
        \end{split}
\end{equation*}
Recalling the definition of $E$ in \eqref{dufe}, we have found
\begin{equation}\label{ste:1}
    |Dv|(x_1,t_1)\leq \lambda+\frac{C}{(\eta-\sigma)^{\frac{(N+2)}{2}}}\,\left(\frac{1}{a(\l)}\, \frac{\max\{1,\theta_\lambda\}}{\min\{1,\theta_\lambda\}^{N/2}}\,E\right)^{1/2},
\end{equation}
for all $(x_1, t_1)\in \mathcal{Q}_\sigma$, with $C=C(\data)>0$.
\vspace{0.1cm}

\noindent\textit{Step 2: the alternatives and iteration.}
In the case
\begin{equation}\label{if:1}
    M(\sigma)\leq 2\lambda=\frac{1}{2}M(\eta),
\end{equation}
then we already have the contraction needed for iteration, so there is nothing to prove. Thus, assume that \eqref{if:1} does not hold. Then there exists $(x_1,t_1)\in \mathcal{Q}_\sigma$ such that $|Dv|(x_1, t_1)>2\lambda$. From this information and \eqref{ste:1}, we deduce
\begin{equation}\label{deduce}
    \lambda\leq \frac{C}{(\eta-\sigma)^{\frac{(N+2)}{2}}}\,\left(\frac{1}{a(\l)}\, \frac{\max\{1,\theta_\lambda\}}{\min\{1,\theta_\lambda\}^{N/2}}\,E\right)^{1/2}.
\end{equation}
In particular, in the case $\theta_\lambda \geq 1$, from \eqref{deduce} and \eqref{put} we deduce
\begin{equation}\label{deduce:1}
    \lambda\leq  \frac{C}{(\eta-\sigma)^{\frac{(N+2)}{2}}}\,\frac{\theta_\lambda^{1/2}}{a(\lambda)^{1/2}}\,E^{1/2}=\frac{C}{(\eta-\sigma)^{\frac{(N+2)}{2}}}\,\left( \frac{S}{R^2}E\right)^{1/2},
\end{equation}
whereas in the case $\theta_\lambda<1$ we get
\begin{equation}\label{deduce:2}
     \lambda\leq  \frac{C}{(\eta-\sigma)^{\frac{(N+2)}{2}}}\left( \frac{E}{a(\lambda)\,\theta_\lambda^{N/2}}\right)^{1/2}=\frac{C}{(\eta-\sigma)^{\frac{(N+2)}{2}}}\left(\frac{(R^2/S)^{N/2}}{a(\lambda)^{1+N/2}}\,E \right)^{1/2}.
\end{equation}
Now let $F(s)$ be given by \eqref{dufe}, and observe that
\[
\frac{F'(s)s}{F(s)}=2+\frac{(N+2)}{2}\,\frac{a'(s)s}{a(s)}.
\]
By the assumption $i_a>-4/(N+2)$, the function $F$ satisfies
\begin{equation}\label{conseq}
    0<2+\frac{N+2}{2}i_a
    \leq i_F\coloneqq\inf_{s>0}\frac{F'(s)s}{F(s)}
    \leq \sup_{s>0}\frac{F'(s)s}{F(s)}
    \eqqcolon s_F
    \leq 2+\frac{N+2}{2}s_a<\infty.
\end{equation}
In particular, $F$ is strictly increasing on $(0,\infty)$ and hence admits an inverse. Therefore, \eqref{deduce:2} yields
\begin{equation}\label{deduce:3}
    \lambda
    \leq
    F^{-1}\left(
        \frac{C\,(R^2/S)^{N/2}}{(\eta-\sigma)^{N+2}}\,E
    \right)
    \leq
    \frac{C'}{(\eta-\sigma)^{(N+2)/i_F}}\,
    F^{-1}\left(
        \left(\frac{R^2}{S}\right)^{N/2}E
    \right),
\end{equation}
Observe that in the last inequality we used   \eqref{b2t}, with $F$ in place of $b$, and taking into account \eqref{conseq}.

Since trivially $M(\sigma)\leq M(\eta)=4\lambda$, combining the alternatives \eqref{if:1}, \eqref{deduce:1} and \eqref{deduce:3}, we have found
\begin{equation}
    M(\sigma)\leq \frac{1}{2}M(\eta)+\frac{C}{(\eta-\sigma)^{\frac{(N+2)}{2}}}\left( \frac{S}{R^2}E\right)^{1/2}+\frac{C}{(\eta-\sigma)^{(N+2)/i_F}}\, F^{-1}\left(
        \left(\frac{R^2}{S}\right)^{N/2}E
    \right),
\end{equation}
for all $1/2\leq \sigma<\eta\leq 3/4$. The iterative Lemma \ref{lem:iteration} yields 
\[
M(1/2)\leq C\left( \frac{S}{R^2}E\right)^{1/2}+C\, F^{-1}\left(
        \left(\frac{R^2}{S}\right)^{N/2}E
    \right),
\]
that is \eqref{gr:bd1}, our thesis.
\end{proof}

\begin{proof}[Proof of Theorem \ref{thm:grbddhom} (ii) under \eqref{smooth}] Let $\sigma,\eta, \mathcal{Q}_\sigma$, $M(\eta)$  and $\lambda$ be as in \eqref{let:parameter} and \eqref{l:set}. Again, we may assume $4\lambda=M(\eta)>0$ for otherwise there is nothing to prove. Now set
\begin{equation}\label{new:r}
    r\coloneqq \frac{(\eta-\sigma)}{4}\,\min\Big\{R,\sqrt{S\,a(\lambda)} \Big\}.
\end{equation}
With this choice, for every $(x_1,t_1)\in \mathcal{Q}_\sigma$, we have that 
\[
Q^{(\lambda)}_{2r}(x_1,t_1)\subset  \mathcal{Q}_\eta.
\]
    Since $|Dv|\leq 4\lambda=M(\eta)$ in $\mathcal{Q}_\eta$, we may apply Propositions \ref{prop:temp:grbd} and couple it with  Proposition \ref{prop:cac:osc} (with $Q^{(\lambda)}_r(x,t)$ in place of $Q_{R,S}$), and get
\begin{equation}
\begin{split}
    |Dv|(x_1,t_1)&\leq
    \lambda+C\,\left(\frac{1}{a(\lambda)}\,\miint_{Q^{(\lambda)}_r(x_1,t_1)}B(|Dv|)\,dxdt \right)^{1/2}
    \\
    &\leq\lambda+C\, \left[ \frac{B(\omega/r)}{a(\lambda)}+\frac{\omega^2}{r^2}\right]^{1/2},
    \end{split}
\end{equation}
   for all $(x_1,t_1)\in \mathcal{Q}_\sigma$, where $\omega=\operatorname*{osc}_{Q_{R,S}}v$. Thus, we have found
   \begin{equation}\label{st:2}
       M(\sigma)\leq \lambda+C\, \left[ \frac{B(\omega/r)}{a(\lambda)}+\frac{\omega^2}{r^2}\right]^{1/2}.
   \end{equation}
We now fix a parameter $K=K(\data)\geq 1$ to be chosen later, and we distinguish two cases.
\\
\noindent\textit{Case 1.} Assume that $\lambda\,r\geq K\omega$; then
\[
\frac{\omega}{r}\leq \frac{\l}{K},
\]
   which together with \eqref{st:2} and the monotonicity of $B$, \eqref{B2t}, \eqref{Bcomeb} gives
   \begin{equation*}
   \begin{split}
       M(\sigma)&\leq \lambda+C\left[\frac{B(\l/K)}{a(\l)}+\left(\frac{\l}{K}\right)^2 \right]^{1/2}\leq \lambda+C\left[K^{-i_B}\frac{B(\lambda)}{a(\lambda)}+\left(\frac{\l}{K}\right)^2\right]^{1/2}
       \\
       &\leq \lambda+ C'\left[\frac{\lambda}{K^{i_B/2}}+\frac{\lambda}{K}\right]\leq \lambda+\frac{2C'}{K^{i_B/2}}\,\lambda.
       \end{split}
   \end{equation*}
   where in the last inequality we used that $K\geq 1$ and $i_B\leq 2$. We now choose $K=K(\data)\geq 1$ so large that $2C'/K^{i_B/2}<1$, so the above inequality implies
   \begin{equation}\label{alt:1}
       M(\sigma)\leq 2\lambda=\frac{1}{2}M(\eta).
   \end{equation}
\noindent\textit{Case 2.} Suppose that  $\lambda\,r< K\omega$. Then by \eqref{new:r} we have
   \[
   \min\big\{\l R,\l\sqrt{Sa(\l)}\big\}<\frac{4K}{(\eta-\sigma)}\omega;
   \]
thus either
\[
\l <\frac{4K}{(\eta-\sigma)}\frac{\omega}{R}\quad\text{or}\quad\l\sqrt{a(\l)} <\frac{4K}{(\eta-\sigma)}\frac{\omega}{\sqrt{S}}.
\]
In the second case, recalling  \eqref{Bcomeb}, we have
\[
B(\l)\leq \frac{(4K)^2}{(\eta-\sigma)^2}\,\frac{\omega^2}{S};
\]
Combining the two alternatives above, and using \eqref{B2t}, we obtain
\begin{equation}\label{alt:2}
    \l\leq \frac{C}{(\eta-\sigma)}\frac{\omega}{R}+\frac{C}{(\eta-\sigma)^{2/i_B}}\,B^{-1}\left(\frac{\omega^2}{S} \right).
\end{equation}
\noindent\textit{Conclusion.} Combining the two cases \eqref{alt:1}, \eqref{alt:2} and that $M(\sigma)\leq M(\eta)=4\lambda$ gives
\[
M(\sigma)\leq \frac{1}{2}M(\eta)+ \frac{C}{(\eta-\sigma)}\frac{\omega}{R}+\frac{C}{(\eta-\sigma)^{2/i_B}}\,B^{-1}\left(\frac{\omega^2}{S} \right),
\]
for all $1/2\leq \sigma<\eta\leq 3/4$. The iterative Lemma \ref{lem:iteration}  gives
\[
M(1/2)\leq C\,\frac{\omega}{R}+C\,B^{-1}\left(\frac{\omega^2}{S} \right),
\]
that is \eqref{gr:bd2}, our thesis.
\end{proof}

\subsection{Oscillation estimates}\label{subsec:oscillation} In this section we prove H\"older continuity of the solution $v$ to \eqref{eq:hom}.

\begin{lemma}[Oscillation estimate]\label{lem:osc:v}
    Let $v$ be a local weak solution to \eqref{eq:hom}
    with \eqref{iasa} and \eqref{ass:A}
    in force, and let
    $Q_{R,S}=Q_{R,S}(x_0,t_0)\Subset\Omega_T$.
    Assume that
\[
\|Dv\|_{L^\infty(Q_{R,S})}\leq \lambda.
\]
Then
    \begin{equation}\label{osc:v}
        \operatorname*{osc}_{Q_{R,S}}v
        \leq C\left(
        R \,\l+\frac{S}{R}b(\l)
        \right),
    \end{equation}
    where $C>0$ depends only on $\data$.
    In particular, if $Q^{(\lambda)}_R=Q_R^{(\lambda)}(x_0,t_0)$ satisfies \eqref{ass:lambda}, 
    then
    \begin{equation}\label{osc:v:intrinsic}
        \operatorname*{osc}_{Q_R^{(\lambda)}}v
        \leq C\,\lambda\, R.
    \end{equation}
\end{lemma}

\begin{proof}[Proof under \eqref{smooth}]
    By translation, we may assume that $(x_0,t_0)=(0,0)$.  Let $\eta\in C^\infty_c(B_R)$
    be a nonnegative function such that
    \[
        \int_{B_R}\eta_R\,dx=1,
        \qquad
        \|\eta_R\|_{L^\infty(B_R)}
        \leq C\,R^{-N},
        \qquad
        \|D\eta_R\|_{L^1(B_R)}
        \leq\frac{C}{R},
    \]
    where $C=C(N)>0$, and set
    \[
        m(t)\coloneqq
        \int_{B_R}v(y,t)\eta_R(y)\,dy.
    \]
    Since $\|Dv\|_{L^\infty(Q_{R,S})}
    \leq\lambda$, for all $(x,t)\in Q_{R,S}$
     we have
    \[
    \begin{split}
        |v(x,t)-m(t)|
        &\leq
        \int_{B_R}
        |v(x,t)-v(y,t)|\,\eta_R(y)\,dy\\
        &\leq
        \lambda\int_{B_R}
        |x-y|\,\eta_R(y)\,dy
        \leq 2\lambda R.
    \end{split}
    \]

    On the other hand, by  \eqref{eq:hom}, we have
    \[
    \begin{split}
        m'(t)
        =-\int_{B_R}
        \A_0(t,Dv)\cdot D\eta_R\,dx=-\int_{B_R}
        \big(\A_0(t,Dv)-\A_0(t,0)\big)
        \cdot D\eta_R\,dx
    \end{split}
    \]
    for all $t\in(-S,0)$.
    Hence, by Lemma \ref{lemma:Agrco}
    and the monotonicity of $b$, we obtain
    \[
    \begin{split}
        |m'(t)|
        &\leq
        C\int_{B_R}
        b(|Dv|)|D\eta|\,dx\leq
        Cb(\lambda)\|D\eta\|_{L^1(B_R(x_0))}
        \leq C\frac{b(\lambda)}{R}.
    \end{split}
    \]
    Consequently,
    \[
        |m(t)-m(s)|
        \leq C\frac{S}{R}b(\lambda)
    \]
    for all $t,s\in(t_0-S,t_0)$.

    Combining the previous estimates,
    for all $(x,t),(y,s)\in Q_{R,S}$
    we infer that
    \[
    \begin{split}
        |v(x,t)-v(y,s)|
        &\leq
        |v(x,t)-m(t)|
        +|m(t)-m(s)|+
        |m(s)-v(y,s)|\\
        &\leq
        C\left(
        R\lambda+\frac{S}{R}b(\lambda)
        \right),
    \end{split}
    \]
   which proves \eqref{osc:v}. Finally, if $S=R^2/a(\lambda)$,
    recalling that $b(\lambda)=\lambda a(\lambda)$,
    we have
    \[
        \frac{S}{R}b(\lambda)
        =\frac{R}{a(\lambda)}b(\lambda)
        =\lambda R,
    \]
    and \eqref{osc:v:intrinsic} follows.
\end{proof}

\subsection{Spatial gradient H\"older continuity}
In this section we establish H\"older continuity of the spatial gradient of solutions to \eqref{eq:hom}. The main result is the following.

\begin{theorem}\label{thm:gradhold}
    Let $v$ be a local weak solution to \eqref{eq:hom} under the assumptions \eqref{iasa} and \eqref{ass:Ahom}.
    Assume that either
\[
i_B>\frac{2N}{N+2}\quad\text{or}\quad v\in L^\infty_{\rm loc}(\Omega_T).
\]
    Then there exists $\beta=\beta(\data)\in (0,1)$ such that
    \[
    Dv\in C_{\rm{loc}}^{\beta,\beta/2}(\Omega_T).
    \] 
    Moreover there exists a constant $C=C(\data)>0$ such that the following hold
    \begin{itemize}
        \item[(i)] for all
$Q^{(\lambda)}_R=Q^{(\lambda)}_R(x_0,t_0)$  such that 
\[
\sup_{Q^{(\lambda)}_R} |Dv|\leq \lambda,
\]
we have
    \begin{equation}\label{osc:intr}
        \operatorname*{osc}_{Q^{(\lambda)}_r} Dv\leq C\,\lambda\,\left( \frac{r}{R}\right)^{\beta}\quad \text{for all }0<r\leq R;
    \end{equation}
\item[(ii)] for all $Q_{R}=Q_{R}(x_0,t_0)\Subset \Omega_T$ and all $0<r<R$, we have
\begin{equation}\label{osc:extr}
    \operatorname*{osc}_{Q_{r}} Dv\leq C\,M_0\,\left(\frac{r}{R} \,a^\star(M_0)\right)^\beta,\qquad M_0\coloneqq\|Dv\|_{L^\infty(Q_R)},
\end{equation}
    where $a^\star(r)\coloneqq \max\{a(r)^{1/2}, a(r)^{-1/2}\}$.
    \end{itemize}
 \end{theorem}

As in the previous section, throughout this part we shall assume that \eqref{smooth} is in force.

Theorem \ref{thm:gradhold} will be a consequence of the following
\begin{proposition}[The fundamental alternative]\label{prop:fund:alt}
    Let $v$ be a local solution to \eqref{eq:hom} under the assumptions \eqref{iasa}, \eqref{ass:Ahom} and \eqref{smooth}.  Suppose that
    \begin{equation}\label{intr:hup0}
       Q_{R}^{(\lambda)}= Q_{R}^{(\lambda)}(x_0,t_0)\Subset \Omega_T,\quad \max_{k=1,\dots,N}\sup_{Q_{_{R}}^{(\lambda)}} |D_kv|\leq \lambda,
    \end{equation}  
%     and set
% \[
% \mathrm{M}(R)\coloneqq \sup_{k=1,\dots,N}\sup_{Q_R^{(\lambda)}} |D_k v|,\qquad 0<R\leq R_0.
% \]
There exist universal constants
    \begin{equation*}
        \mu_0,\eta_0,\delta_0\in (0,1)\quad\text{depending only on $\data$}
    \end{equation*}
    such that the following alternatives hold:
    \begin{itemize}
        \item[(I)] if for some $k=1,\dots,N$ we have either 
          \begin{equation}\label{alt:0}
    \begin{split}
       \big |\{D_k v<\lambda/2\}&\cap Q^{(\l)}_R\big|\leq \mu_0\,|Q^{(\l)}_R|
       \\
       &\text{or}
       \\
       \big |\{D_k v>-\lambda/2&\}\cap Q^{(\l)}_R\big|\leq \mu_0\,|Q^{(\l)}_R|
       \end{split}
    \end{equation}  
        then
\begin{equation}\label{thalt1:modDu}
    |D v|\geq \lambda/4\quad\text{in $Q^{(\l)}_{R/2}$.}
\end{equation}
    \item[(II)] In the complementary case, that is, if \begin{equation}\label{alt:3}
        \begin{cases}
            \big |\{D_k v<\lambda/2\}\cap Q^{(\l)}_R\big|>\mu_0\,|Q^{(\l)}_R|\qquad\text{and}
            \\
            \\
            \big |\{D_k v>-\lambda/2\}\cap Q^{(\l)}_R\big|> \mu_0\,|Q^{(\l)}_R|,\quad\text{for all $k=1,\dots,N$,}
        \end{cases}
    \end{equation}
    then
    \begin{equation}\label{thesis:alt3}
        \max_{k=1,\dots,N}\sup_{Q_{\delta_0 R}^{(\lambda)}} |D_k v|\leq \eta_0\,\lambda\,.
    \end{equation}
    \end{itemize}
\end{proposition}

We start by proving the parabolic counterpart of the De Giorgi's type inequalities of \cite[Lemmas 4.6-4.7]{Antonini2027}.

\begin{proposition}\label{pr:1DG}
    Let $v$ be a local solution to \eqref{eq:hom} under the assumptions
    \eqref{iasa}, \eqref{ass:Ahom} and \eqref{smooth}.
    
    Let $g:\R\to\R$ be a Lipschitz function such that $g'(t)\geq0$
    for a.e. $t\in\R$, and let $G$ be a nonnegative primitive of $g$.
    Then there exist constants $0<c<1<C$ depending on $\data$
    such that the following holds.
    For every cylinder
    \[
        Q_\tau\coloneqq\mathcal{B}\times(t_*,\tau)
        \Subset\Omega_T,
    \]
    where $\mathcal{B}\subset\R^N$ is open, every $k=1,\dots,N$,
    and every smooth function $\zeta$ on
    $\mathcal{B}\times[t_*,\tau]$ such that
    \[
        \operatorname{spt}\zeta(\cdot,t)\subset K
        \qquad\text{for all }t\in[t_*,\tau],
    \]
 for some compact set $K\Subset \B$, we have
 \begin{equation}\label{DGG1}
    \begin{split}
        &\int_{\mathcal{B}\times\{\tau\}}
        G(D_kv)\,\zeta^2\,dx
        +
        c\iint_{Q_\tau}
        a(|Dv|)\,|D(D_kv)|^2\,g'(D_kv)\,\zeta^2\,dxdt
        \\
        &\quad\leq
        \int_{\mathcal{B}\times\{t_*\}}
        G(D_kv)\,\zeta^2\,dx
        +
        C\iint_{Q_\tau}
        a(|Dv|)\,|Dv|^2\,g'(D_kv)\,|D\zeta|^2\,dxdt
        \\
        &\qquad+
        C\iint_{Q_\tau}
        b(|Dv|)\,|g(D_kv)|\,|D^2\zeta^2|\,dxdt
        +
        \iint_{Q_\tau}
        G(D_kv)\,|\partial_t\zeta^2|\,dxdt.
    \end{split}
    \end{equation}
    If, in addition,
    \[
        |\{g\neq0\}\cap\{g'=0\}|=0,
    \]
    then
    \begin{equation}\label{DGG2}
    \begin{split}
        &\int_{\mathcal{B}\times\{\tau\}}
        G(D_kv)\,\zeta^2\,dx
        +
        c\iint_{Q_\tau}
        a(|Dv|)\,|D(D_kv)|^2\,g'(D_kv)\,\zeta^2\,dxdt
        \\
        &\leq
        \int_{\mathcal{B}\times\{t_*\}}
        G(D_kv)\,\zeta^2\,dx
        +
        C\iint_{Z_k\cap Q_\tau}
        a(|Dv|)\,
        \frac{g^2(D_kv)}{g'(D_kv)}
        |D\zeta|^2\,dxdt
        \\
        &\quad+
        \iint_{Q_\tau}
        G(D_kv)\,|\partial_t\zeta^2|\,dxdt,
    \end{split}
    \end{equation}
    where $Z_k\coloneqq
        \{D(D_kv)\neq0\}\cap\{g(D_kv)\neq0\}$.
\end{proposition}

\begin{proof}
    We differentiate Equation \eqref{eq:hom} with respect to $x_k$,
    and setting $w_k\coloneqq D_kv$ we get
    \begin{equation}\label{eq:linear}
        \partial_t w_k-
        \mathrm{div}\big(\nabla_\xi\A_0(t,Dv)\,Dw_k\big)=0
        \quad\text{in $\Omega_T$.}
    \end{equation}
    Multiplying \eqref{eq:linear} by $g(w_k)\,\zeta^2$
    and integrating over $Q_\tau$, we get
    \begin{equation}\label{tp:1}
    \begin{split}
        &\int_{\mathcal{B}\times\{\tau\}}
        G(w_k)\,\zeta^2\,dx
        +
        \iint_{Q_\tau}
        \nabla_\xi\A_0(t,Dv)\,Dw_k
        \cdot g'(w_k)\,Dw_k\,\zeta^2\,dxdt
        \\
        &\quad+
        \iint_{Q_\tau}
        \nabla_\xi\A_0(t,Dv)\,Dw_k
        \cdot D\zeta^2\,g(w_k)\,dxdt
        \\
        &=
        \int_{\mathcal{B}\times\{t_*\}}
        G(w_k)\,\zeta^2\,dx
        +
        \iint_{Q_\tau}
        G(w_k)\,\partial_t\zeta^2\,dxdt,
    \end{split}
    \end{equation}
    where we used that, via integration by parts,
    \[
    \begin{split}
        \iint_{Q_\tau}
       & (\partial_t w_k)\,g(w_k)\,\zeta^2\,dxdt
        =
        \iint_{Q_\tau}
        \partial_t G(w_k)\,\zeta^2\,dxdt
        \\
        &=
        \int_{\mathcal{B}\times\{\tau\}}
        G(w_k)\,\zeta^2\,dx
        -
        \int_{\mathcal{B}\times\{t_*\}}
        G(w_k)\,\zeta^2\,dx-
        \iint_{Q_\tau}
        G(w_k)\,\partial_t\zeta^2\,dxdt.
    \end{split}
    \]
    The second and third integrals on the left-hand side of
    \eqref{tp:1} can be estimated exactly as in
    \cite[Proof of Lemma 4.6]{Antonini2027}.
    For the sake of completeness, we report the full details.

    Via integration by parts, we find
    \begin{equation}\label{parts}
    \begin{split}
        \int_{\mathcal{B}}
        &\nabla_\xi\A_0(t,Dv)\,Dw_k
        \cdot D\zeta^2\,g(w_k)\,dx=
        \int_{\mathcal{B}}
        D_k\big(\A_0(t,Dv)\big)
        \cdot D\zeta^2\,g(w_k)\,dx
        \\
        &=
        -\int_{\mathcal{B}}
        \A_0(t,Dv)\cdot D\zeta^2\,
        g'(w_k)\,D_kw_k\,dx-
        \int_{\mathcal{B}}
        \A_0(t,Dv)\cdot D(D_k\zeta^2)\,g(w_k)\,dx.
    \end{split}
    \end{equation}
    Then, by \eqref{ass:Ahom},
    \begin{equation}\label{cro:1}
    \begin{split}
        \iint_{Q_\tau}
        \nabla_\xi\A_0(t,Dv)\,Dw_k
        \cdot g'(w_k)\,Dw_k\,\zeta^2\,dxdt\geq
        c_0\iint_{Q_\tau}
        a(|Dv|)\,|Dw_k|^2\,g'(w_k)\,\zeta^2\,dxdt
    \end{split}
    \end{equation}
    for some $c_0=c_0(\data)>0$.
    Moreover, using \eqref{deb:ass:A}
    (with $\Lambda_{\rm h}=0$), that $b(t)=a(t)t$,
    and Young's inequality,
    \begin{equation*}
    \begin{split}
        \bigg|\iint_{Q_\tau}
        &\A_0(t,Dv)\cdot D\zeta^2\,
        g'(w_k)\,D_kw_k\,dxdt\bigg|\leq
        C\iint_{Q_\tau}
        b(|Dv|)\,|\zeta|\,|D\zeta|\,
        g'(w_k)\,|D_kw_k|\,dxdt
        \\
        &\leq
        \frac{c_0}{4}\iint_{Q_\tau}
        a(|Dv|)\,|Dw_k|^2\,g'(w_k)\,\zeta^2\,dxdt+
        C\iint_{Q_\tau}
        a(|Dv|)\,|Dv|^2\,g'(w_k)\,|D\zeta|^2\,dxdt,
    \end{split}
    \end{equation*}
    and
    \begin{equation*}
    \begin{split}
        \bigg|\iint_{Q_\tau}
        \A_0(t,Dv)\cdot D(D_k\zeta^2)\,g(w_k)\,dxdt\bigg|\leq
        C\iint_{Q_\tau}
        b(|Dv|)\,|g(w_k)|\,|D^2\zeta^2|\,dxdt.
    \end{split}
    \end{equation*}
    Combining the two estimates above with
    \eqref{cro:1} and \eqref{parts}, substituting the resulting
    estimates into \eqref{tp:1}, and using that $G\geq0$,
    we obtain \eqref{DGG1} after absorbing the corresponding
    term into the left-hand side.

    Next, by \eqref{ass:Ahom} and Young's inequality\footnote{ Here the computations are justified since
       $ \big|Z_k\cap Q_\tau\cap\{g'(w_k)=0\}\big|=0 $
    by the coarea formula and the assumption
    $|\{g'=0\}\cap\{g\neq0\}|=0$.},
    \begin{equation*}
    \begin{split}
        2\bigg|\iint_{Q_\tau}
        &\nabla_\xi\A_0(t,Dv)\,Dw_k
        \cdot D\zeta\,\zeta\,g(w_k)\,dxdt\bigg|\leq
        C\iint_{Z_k\cap Q_\tau}
        a(|Dv|)\,|Dw_k|\,|D\zeta|\,|\zeta|\,|g(w_k)|\,dxdt
        \\
        &\leq
        \frac{c_0}{4}\iint_{Q_\tau}
        a(|Dv|)\,|Dw_k|^2\,g'(w_k)\,\zeta^2\,dxdt+
        C\iint_{Z_k\cap Q_\tau}
        a(|Dv|)\,
        \frac{g(w_k)^2}{g'(w_k)}\,|D\zeta|^2\,dxdt.
    \end{split}
    \end{equation*}
    Combining this estimate with \eqref{cro:1} and \eqref{tp:1},
    and using again that $G\geq0$, we obtain \eqref{DGG2}
    after absorption.
    This completes the proof.
\end{proof}

\begin{proposition}[Parabolic De Giorgi's type inequalities on intrinsic cylinders]\label{prop:levels}
Let $v$ be a local solution to \eqref{eq:hom} under the assumptions \eqref{iasa}, \eqref{ass:Ahom}, \eqref{smooth}, and suppose that 
\begin{equation}\label{intr:hup}
       Q_{R}^{(\lambda)}= Q_{R}^{(\lambda)}(x_0,t_0)\Subset \Omega_T,\quad\text{and}\quad \max_{k=1,\dots,N}\sup_{Q_{_{R}}^{(\lambda)}} |D_kv|\leq \lambda.
    \end{equation}  
Then there exists $C=C(\data)>0$ such that, for all $0<r_1< r_2\leq R$, for all $k=1,\dots,N$, and for all 
    \begin{equation}\label{kell}
        \frac{\lambda}{8}\leq \kappa<\ell\leq \lambda,
    \end{equation}
    we have the following De Giorgi's type inequalities
    \begin{equation}\label{DGl1}
    \begin{split}
      \sup_{t\in (t_0-r_1^2/a(\l),t_0) }\int_{B_{r_1}(x_0)\times \{t\}} (D_kv-\ell)_-^2dx &+ a(\l) \iint_{Q^{(\lambda)}_{r_1}\cap \{\kappa<D_k v<\ell\} } |D(D_k v)|^2\,dxdt
      \\
      &\leq C\,\frac{a(\l)\,\l^2}{(r_2-r_1)^2}\,|Q^{(\l)}_{r_2}\cap \{D_k v<\ell\}|,
      \end{split}
    \end{equation}
and
\begin{equation}\label{DGl2}
    \begin{split}
        \sup_{t\in (t_0-r_1^2/a(\l),t_0) }\int_{B_{r_1}(x_0)\times \{t\}} (D_kv-\kappa)_+^2dx &+ a(\l) \iint_{Q^{(\lambda)}_{r_1}\cap \{D_k v>\kappa\} } |D(D_k v)|^2\,dxdt
      \\
      &\leq C\,\frac{a(\l)\,(\l-\kappa)^2}{(r_2-r_1)^2}\,|Q^{(\l)}_{r_2}\cap \{D_k v>\kappa\}|.
    \end{split}
\end{equation}
\end{proposition}

\begin{proof}
  By translation, we may assume $(x_0,t_0)=(0,0)$.  Let $K$ be a compact set such that $B_{r_1}\Subset K\Subset B_{r_2}$, and let $\zeta$ be a smooth cut-off function such that, 
    \begin{equation}\label{zeta}
    \begin{split}
        &0\leq \zeta\leq 1,\qquad \operatorname{spt}\zeta(\cdot,t)\subset K\quad\text{for all $t\in (-r_2^2/a(\l),0]$},
        \\
      & \zeta\equiv 1\quad\text{on $Q^{(\lambda)}_{r_1}$},\qquad \zeta(\cdot,-r_2^2/a(\l))=0,
        \\
      &   |D\zeta|^2
    +|D^2\zeta|\leq \frac{C(N)}{(r_2-r_1)^2},\qquad |\partial_t \zeta|\leq \frac{C(N)}{(r_2-r_1)^2}\,a(\l).
    \end{split}
    \end{equation}
Owing to \eqref{intr:hup} and \eqref{kell}, we have that $\l/8\leq |Dv|\leq \sqrt{N}\,\lambda$ on $\{D_kv>\kappa\}\cap \mathrm{spt}\,\zeta$: it follows by \eqref{ultra:utile} that
\begin{equation}\label{uutil:1}
    c\,a(\lambda)\leq a(|Dv|)\leq C\,a(\lambda)\quad\text{on $\{D_kv>\kappa\}\cap \mathrm{spt}\,\zeta$,}
\end{equation}
and by using the monotonicity of $b$ and \eqref{b2t},
\begin{equation}\label{uutil:2}
    a(|Dv|)\,|Dv|=b(|Dv|)\leq b(\sqrt{N}\lambda)\leq C\,b(\lambda)=C\,a(\l)\,\l\quad\text{on $Q^{(\l)}_R$,}
\end{equation}
 for some positive constants $c,C$ depending only on $N,i_a,s_a$.

We apply \eqref{DGG1} with 
\[
Q_\tau=B_{r_2}\times (-r_2^2/a(\l),\tau),\qquad g(t)=-(t-\ell)_-=(t-\ell)\mathbbm{1}_{\{t<\ell\}},\qquad G(t)=\frac12(t-\ell)_-^2.
\]
Using that $\zeta(\cdot,-r_2^2/a(\lambda))=0$, and taking the supremum over $\tau\in (-r^2_2/a(\l),0) $, we obtain
\begin{equation*}
\begin{split}
    \sup_{t\in(-r_2^2/a(\lambda),0)}
    &\int_{B_{r_2}\times \{t\}}
    (D_kv-\ell)_-^2\,\zeta^2\,dx
    +
    \iint_{Q_{r_2}^{(\lambda)}\cap\{D_kv<\ell\}}
    a(|Dv|)\,|D(D_kv)|^2\,\zeta^2\,dxdt
    \\
    \leq\;&
    C\iint_{Q_{r_2}^{(\lambda)}\cap\{D_kv<\ell\}}
    a(|Dv|)\,|Dv|^2\,|D\zeta|^2\,dxdt+
    C\iint_{Q_{r_2}^{(\lambda)}}
    b(|Dv|)\,(D_kv-\ell)_-\,|D^2\zeta^2|\,dxdt
    \\
    &+
    \frac{C}{2}\iint_{Q_{r_2}^{(\lambda)}}
    (D_kv-\ell)_-^2\,|\partial_t\zeta^2|\,dxdt .
\end{split}
\end{equation*}

Then  we estimate via \eqref{uutil:1} and \eqref{zeta},
\[
\begin{split}
    \int_{\{D_k v<\ell\}} &a(|Dv|)\,|D(D_kv)|^2\,\zeta^2\,dxdt \geq \int_{\{\kappa<D_k v<\ell\}}a(|Dv|)\,|D(D_kv)|^2\,\zeta^2\,dxdt
    \\
    &\geq c\,a(\l)\,\int_{\{\kappa<D_k v<\ell\}}|D(D_kv)|^2\,\zeta^2\,dxdt
\geq c\,a(\l)\,\int_{Q_{r_1}^{(\l)}\cap\{\kappa<D_k v<\ell\}}|D(D_kv)|^2\,dxdt
\end{split}
\]
We estimate the remaining terms using \eqref{uutil:2}, the bound $|Dv|\leq\sqrt{N}\,\lambda$, and, by  \eqref{intr:hup} and \eqref{kell}, the estimate $|(D_kv-\ell)_-|\leq 2\lambda$, together with the properties in \eqref{zeta}. This proves \eqref{DGl1}.

Next, we apply \eqref{DGG2} with the same $Q_\tau$, and
\[
g(t)=(t-\kappa)_+,\qquad G(t)=\frac12(t-\kappa)_+^2,
\] In the resulting inequality, we use \eqref{uutil:1}-\eqref{uutil:2}, estimate $(D_kv-\kappa)_+\leq(\lambda-\kappa)$, and exploit the properties in \eqref{zeta}. This yields \eqref{DGl2}.

\end{proof}

We shall also need the logarithmic type estimate as in
\cite[II. Proposition 3.2]{DiBenedetto1993}. To this end, for
$m\geq2$, we introduce the function
\begin{equation}\label{def:Psim}
    \Psi_m(s)\coloneqq
    \left[
    \log\left(
    \frac{\lambda/2}
    {\lambda/2-(s-\l/2)_++2^{-m}\lambda}
    \right)
    \right]_+,
    \qquad s\leq\lambda.
\end{equation}
Notice that
\begin{equation}\label{Psim:prop}
    \{\Psi_m>0\}
    =
    \left\{
    s>\frac{\lambda}{2}+2^{-m}\lambda
    \right\},
    \qquad
    0\leq\Psi_m\leq(m-1)\log2.
\end{equation}

\begin{proposition}[Logarithmic estimate]\label{prop:log}
    Let $v$ be a local solution to \eqref{eq:hom} under the assumptions
    \eqref{iasa}, \eqref{ass:Ahom}, \eqref{smooth}, and suppose that
    \eqref{intr:hup} holds.

    There exists $C=C(\data)>0$ such that, for every
    $m\geq2$, every $k=1,\dots,N$, every
    \[
        t_0-\frac{R^2}{a(\lambda)}<t_*<t_0,
    \]
    and every time independent function
    $\zeta\in C^\infty_c(B_R(x_0))$, we have
    \begin{equation}\label{log:est:time}
    \begin{split}
        \sup_{t\in(t_*,t_0)}
        \int_{B_R(x_0)\times\{t\}}
        \Psi_m^2(D_kv)\,\zeta^2\,dx
        &\leq
        \int_{B_R(x_0)\times\{t_*\}}
        \Psi_m^2(D_kv)\,\zeta^2\,dx
        \\
        &\quad+
        C\,a(\lambda)
        \int_{t_*}^{t_0}\int_{B_R(x_0)}
        \Psi_m(D_kv)\,|D\zeta|^2\,dxdt.
    \end{split}
    \end{equation}
\end{proposition}

\begin{proof}
    For brevity, set $\Psi\coloneqq\Psi_m$.
    On the set $\{\Psi>0\}$, we have
    \begin{equation}\label{Psi:der}
        \Psi'(s)
        =
        \frac{1}{\lambda-s+2^{-m}\lambda},
        \qquad
        \Psi''(s)
        =
        \big(\Psi'(s)\big)^2.
    \end{equation}
    We apply \eqref{DGG2} with
    \[
        g(s)\coloneqq2\Psi(s)\Psi'(s)
        \quad\text{on $\{\Psi>0\}$},
        \qquad
        G(s)\coloneqq\Psi^2(s),
    \]
    setting $g(s)=0$ on $\{\Psi(s)=0\}$.
    By \eqref{Psi:der}, on $\{\Psi>0\}$ we have
    \[
        g'(s)
        =
        2\big(\Psi'(s)\big)^2
        +2\Psi(s)\Psi''(s)
        =
        2\big(1+\Psi(s)\big)
        \big(\Psi'(s)\big)^2>0.
    \]
 Since only the range $s\leq\lambda$ is relevant, in view of $D_k v\leq \lambda$ by
    \eqref{intr:hup}, we may also  harmlessly extend $g,G$ outside this range
    so that the assumptions of Proposition \ref{pr:1DG} are satisfied. 
    We also have \begin{equation}\label{log:quotient}
        \frac{g(s)^2}{g'(s)}
        =
        2\,\frac{\Psi(s)^2}{1+\Psi(s)}
        \leq2\Psi(s)
        \qquad\text{on $\{\Psi>0\}$.}
    \end{equation}

    By \eqref{Psim:prop} and \eqref{intr:hup}, we have
    \[
        \frac{\lambda}{2}<D_kv\leq|Dv|\leq\sqrt{N}\lambda
        \qquad\text{on }
        \{\Psi(D_kv)>0\}\cap Q_R^{(\lambda)}.
    \]
    Thus, by \eqref{ultra:utile},
    \begin{equation}\label{log:acomp}
        c\,a(\lambda)
        \leq a(|Dv|)
        \leq C\,a(\lambda)
        \qquad\text{on }
        \{\Psi(D_kv)>0\}\cap Q_R^{(\lambda)},
    \end{equation}
    for some positive constants $c,C$ depending only on $\data$.

       Applying \eqref{DGG2} on
    \[
        Q_\tau=B_R(x_0)\times(t_*,\tau),
        \qquad t_*<\tau<t_0,
    \]
    dropping the nonnegative second term on the left-hand side,
    and using \eqref{log:quotient}, \eqref{log:acomp},
    and the fact that $\partial_t\zeta=0$, we obtain
    \[
    \begin{split}
        \int_{B_R(x_0)\times\{\tau\}}
        \Psi^2(D_kv)\,\zeta^2\,dx
        &\leq
        \int_{B_R(x_0)\times\{t_*\}}
        \Psi^2(D_kv)\,\zeta^2\,dx+
        C\,a(\lambda)
        \iint_{Q_\tau}
        \Psi(D_kv)\,|D\zeta|^2\,dxdt.
    \end{split}
    \]
    Enlarging the last integral to
    $B_R(x_0)\times(t_*,t_0)$ and taking the supremum over
    $\tau\in(t_*,t_0)$ proves \eqref{log:est:time}.
\end{proof}

\subsubsection{Proof of the first alternative}\label{ssub:1alt} We are now in the position to prove
\begin{proof}[Proof of Proposition \ref{prop:fund:alt} (I): the first alternative.]
By translation, we may assume $(x_0,t_0)=(0,0)$. Assume that, for some $k\in\{1,\dots,N\}$,
\begin{equation}\label{at:00}
    \big|\{D_kv<\lambda/2\}\cap Q_R^{(\lambda)}\big|
    \leq \mu_0\,|Q_R^{(\lambda)}|,
\end{equation}
where $\mu_0=\mu_0(\data)\in(0,1)$ will be fixed below.
We prove that
\begin{equation}\label{first:goal}
    D_kv\geq \frac{\lambda}{4}
    \qquad\text{in }Q_{R/2}^{(\lambda)},
\end{equation}
which immediately yields \eqref{thalt1:modDu}.

The other possibility in \eqref{alt:0} follows by symmetry. Indeed, if
\[
    \big|\{D_kv>-\lambda/2\}\cap Q_R^{(\lambda)}\big|
    \leq \mu_0\,|Q_R^{(\lambda)}|,
\]
we set $\tilde v\coloneqq-v$ and $\A_1(t,\xi)\coloneqq-\A_0(t,-\xi)$;  then $\tilde v$ solves
\[
    \partial_t \tilde v-\operatorname{div}\big(\A_1(t,D\tilde v)\big)=0
    \qquad\text{in }\Omega_T,
\]
and $\A_1$ satisfies the same structural assumptions as $\A_0$. Moreover $\{D_k\tilde v<\lambda/2\}=\{D_kv>-\lambda/2\}$. Thus it suffices to prove \eqref{first:goal} under \eqref{at:00}.

% Thus set
% \[
%     w_k\coloneqq D_kv,
%     \qquad 
%     \chi\coloneqq 1+\frac{2}{N}.
% \]
For $j=0,1,\dots$, we define
\begin{equation}\label{first:seq}
    r_j\coloneqq \frac{R}{2}+\frac{R}{2^{j+1}},
    \qquad
    \ell_j\coloneqq \frac{\lambda}{4}
                    +\frac{\lambda}{2^{j+2}},
\end{equation}
and
\[
    A_j\coloneqq
    Q_{r_j}^{(\lambda)}\cap\{D_k v<\ell_j\}.
\]
Notice that
\[
    r_0=R,\qquad r_j\searrow \frac{R}{2},
    \qquad
    \ell_0=\frac{\lambda}{2},
    \qquad
    \ell_j\searrow\frac{\lambda}{4},
\]
and
\begin{equation}\label{first:gaps}
    r_j-r_{j+1}=\frac{R}{2^{j+2}},
    \qquad
    \ell_j-\ell_{j+1}=\frac{\lambda}{2^{j+3}}.
\end{equation}
In particular $ \frac{\lambda}{8}
    \leq \ell_{j+1}<\ell_j\leq\lambda$,
so that Proposition \ref{prop:levels} applies at every step. Let also
\[
    \widehat r_j\coloneqq\frac{r_j+r_{j+1}}{2},
\]
and define the truncated function
\[
    z_j
    \coloneqq
    \min\big\{(\ell_j-D_k v)_+\,,\,(\ell_j-\ell_{j+1})\big\}.
\]
We have
\begin{equation}\label{first:zgrad}
    Dz_j
    =
    -D(D_k v)\,\mathbbm{1}_{\{\ell_{j+1}<D_k v<\ell_j\}}
    \qquad\text{a.e.},
\end{equation}
whereas
\begin{equation}\label{first:zlevel}
    z_j=\ell_j-\ell_{j+1}
    \qquad\text{on }\{D_k v<\ell_{j+1}\}.
\end{equation}

Applying \eqref{DGl1} with $ r_1=\widehat r_j,\,
    r_2=r_j,\,
    \kappa=\ell_{j+1},\,
    \ell=\ell_j$, and using \eqref{first:zgrad}, we obtain
\begin{equation}\label{first:energy1}
\begin{split}
    &\sup_{t\in
    (-\widehat r_j^2/a(\lambda),\,0)}
    \int_{B_{\widehat r_j}\times \{t\}}z_j^2\,dx
    +
    a(\lambda)
    \iint_{Q_{\widehat r_j}^{(\lambda)}}
    |Dz_j|^2\,dxdt
    \\
    &\qquad\leq
    C\,\frac{a(\lambda)\,\lambda^2}
    {(r_j-\widehat r_j)^2}\,|A_j|
    \leq
    C\,4^j\frac{a(\lambda)\lambda^2}{R^2}\,|A_j|.
\end{split}
\end{equation}

Choose now $\psi_j\in C_c^\infty(B_{\widehat r_j})$ such that
\[
    0\leq\psi_j\leq1,
    \qquad
    \psi_j\equiv1\quad\text{in }B_{r_{j+1}},
    \qquad
    |D\psi_j|
    \leq\frac{C(N)}{(r_j-r_{j+1})}.
\]
Since
\[
    0\leq z_j
    \leq(\ell_j-\ell_{j+1})\,\mathbbm{1}_{A_j},
\]
from \eqref{first:energy1} and \eqref{first:gaps} we infer
\begin{equation}\label{first:energy2}
\begin{split}
    &\sup_{t\in
    (-\widehat r_j^2/a(\lambda),\,0)}
    \int_{B_{\widehat r_j}}
    (z_j\psi_j)^2\,dx
    +
    a(\lambda)
    \iint_{Q_{\widehat r_j}^{(\lambda)}}
    |D(z_j\psi_j)|^2\,dxdt\leq
    C\,4^j\frac{a(\lambda)\,\lambda^2}{R^2}\,|A_j|.
\end{split}
\end{equation}

We then set 
\[
\chi=1+\frac{2}{N},
\]
 and we  apply the parabolic Sobolev inequality \cite[I. Proposition 3.1]{DiBenedetto1993} (with $p=m=2$) to $z_j\psi_j$, thus getting
\begin{equation}\label{first:sob}
\begin{split}
    \iint_{Q_{\widehat r_j}^{(\lambda)}}
    |z_j\psi_j|^{2\chi}\,dxdt
    &\leq
    C
    \left(
        \sup_{t\in (-\widehat{r}_j^2/a(\l),0)}\int_{B_{\widehat r_j}}
        (z_j\psi_j)^2\,dx
    \right)^{2/N}
    \iint_{Q_{\widehat r_j}^{(\lambda)}}
    |D(z_j\psi_j)|^2\,dxdt
    \\
    &\leq
    \frac{C}{a(\lambda)}
    \left(
        4^j\,
        \frac{a(\lambda)\,\lambda^2}{R^2}
        |A_j|
    \right)^{\chi}.
\end{split}
\end{equation}
On the other hand, by \eqref{first:zlevel} and the choice of $\psi_j$,
\[
    z_j\psi_j
    =
    \ell_j-\ell_{j+1}
    \qquad\text{on }A_{j+1}.
\]
Therefore \eqref{first:sob} gives
\[
    (\ell_j-\ell_{j+1})^{2\chi}|A_{j+1}|
    \leq
    \frac{C}{a(\lambda)}
    \left(
        4^j
        \frac{a(\lambda)\lambda^2}{R^2}
        |A_j|
    \right)^\chi.
\]
Using \eqref{first:gaps}, we conclude that
\begin{equation}\label{first:measure-rec}
    |A_{j+1}|
    \leq
    C_0\,\mathfrak b^j\,
    a(\lambda)^{2/N}R^{-2\chi}\,
    |A_j|^\chi,
\end{equation}
where $C_0>0$ and $\mathfrak b>1$ depend only on $\data$.

We now normalize the measures by setting
\[
    Y_j\coloneqq\frac{|A_j|}{|Q_R^{(\lambda)}|}.
\]
Since
\[
\begin{split}
    a(\lambda)^{2/N}R^{-2\chi}
    |Q_R^{(\lambda)}|^{\chi-1}
    &=
    |B_1|^{2/N}
    a(\l)^{2/N}R^{-2\chi}
    \left(\frac{R^{N+2}}{a(\lambda)}\right)^{2/N}=C(N),
\end{split}
\]
Equation \eqref{first:measure-rec} reduces to the dimensionless superlinear recursion
\begin{equation}\label{first:Yrec}
    Y_{j+1}
    \leq
    C_1\,\mathfrak b^j\,Y_j^{1+2/N},
\end{equation}
where $C_1=C_1(\data)>0$. Moreover, by \eqref{at:00},
\[
    Y_0
    =
    \frac{
    |\{D_k v<\lambda/2\}\cap Q_R^{(\lambda)}|
    }{|Q_R^{(\lambda)}|}
    \leq\mu_0.
\]
Thereby choosing $\mu_0=\mu_0(\data)\in(0,1)$
small enough for the iteration Lemma \ref{lem:hyp}  applied to
\eqref{first:Yrec}, we obtain  $Y_j\longrightarrow0$. Since $r_j\searrow R/2$ and $\ell_j\searrow\lambda/4$, it follows that
\[
    \big|
    Q_{R/2}^{(\lambda)}
    \cap\{D_kv<\lambda/4\}
    \big|=0,
\]
which implies \eqref{first:goal}, that is our thesis.
\end{proof}

\subsubsection{Proof of the second alternative} In this part we establish the second alternative of Proposition \ref{prop:fund:alt}. To this end, let us assume that \eqref{at:00} is violated, that is
\begin{equation}\label{not:alt1}
    |\{D_k v<\l/2\}\cap Q^{(\l)}_R|>\mu_0\,|Q^{(\l)}_R|.
\end{equation}

We start by showing the following.

\begin{lemma}[Good time level]\label{lem:good:time}
    Let $v$ be a local solution to \eqref{eq:hom} under the
    assumptions \eqref{iasa}, \eqref{ass:Ahom} and \eqref{smooth} and  \eqref{intr:hup};
    suppose that \eqref{not:alt1} holds for
    some $k=1,\dots,N$. Then there exists
    \begin{equation}\label{good:time:interval}
        t_*\in\left(
        t_0-\frac{R^2}{a(\lambda)},\,
        t_0-\left( \frac{\mu_0}{2}\right)\frac{R^2}{a(\lambda)}
        \right)
    \end{equation}
    such that
    \begin{equation}\label{good:time:slice}
        \big|\{x\in B_R(x_0):D_kv(x,t_*)>\lambda/2\}\big|
        \leq
        \left(\frac{1-\mu_0}{1-\mu_0/2}\right)\,|B_R(x_0)|.
    \end{equation}
\end{lemma}

\begin{proof}
    Set
    \[
        I\coloneqq\left(
        t_0-\frac{R^2}{a(\lambda)},\,
        t_0-\left( \frac{\mu_0}{2}\right)\frac{R^2}{a(\lambda)}
        \right),
        \qquad
        A(t)\coloneqq
        \{x\in B_R(x_0):D_kv(x,t)>\lambda/2\}.
    \]
    By Fubini's theorem and \eqref{not:alt1}, we have
    \begin{equation*}
    \begin{split}
        \int_I |A(t)|\,dt
        &\leq
        \big|\{D_kv>\lambda/2\}\cap Q_R^{(\lambda)}\big|\leq
        |Q_R^{(\lambda)}|
        -\big|\{D_kv<\lambda/2\}\cap Q_R^{(\lambda)}\big|<
        (1-\mu_0)|Q_R^{(\lambda)}|.
    \end{split}
    \end{equation*}
Dividing by  $|I|=\left(1-\frac{\mu_0}{2}\right)
        \frac{R^2}{a(\lambda)}$ yields
    \[
        \frac{1}{|I|}\int_I |A(t)|\,dt
        <
       \left( \frac{1-\mu_0}{1-\mu_0/2}\right)\,|B_R(x_0)|.
    \]
  Consequently, there must exist $t_*\in I$ such that
    \eqref{good:time:slice} holds. This ends the proof.
\end{proof}

\begin{lemma}[Propagation of lower level sets]\label{lem:propagation}
 Under the notations and assumptions of Lemma \ref{lem:good:time},  there exists an integer $m_0\geq3$, depending only
    on $\data$, such that
    \begin{equation}\label{propagation:lower}
        \big|
        \{x\in B_R(x_0):
        D_kv(x,t)\leq(1-2^{-m_0})\lambda\}
        \big|
        \geq
        \frac{\mu_0}{8}\,|B_R(x_0)|
    \end{equation}
    for all $t\in (t_*,t_0)$.
\end{lemma}

\begin{proof}
    By translation, we may assume $(x_0,t_0)=(0,0)$. We set
\[
    \beta_0\coloneqq\frac{1-\mu_0}{1-\mu_0/2},
    \quad\text{and}\quad
    \rho\coloneqq\left(1-\frac{\mu_0}{8}\right)^{1/N}.
\]
Observe that $\rho\in (1/2,1)$ depends only on $\data$, and
    \begin{equation}\label{propagation:rho}
        |B_R\setminus B_{\rho R}|
        =
        \frac{\mu_0}{8}\,|B_R|.
    \end{equation}
    Let $\zeta\in C^\infty_c(B_R)$ be a time independent
    cut-off function such that
    \begin{equation}\label{propagation:cutoff}
        0\leq\zeta\leq1,\qquad
        \zeta=1\quad\text{in }B_{\rho R},
        \qquad
        |D\zeta|\leq\frac{C}{(1-\rho)R}.
    \end{equation}

    By Lemma \ref{lem:good:time}, there exists
    \[
        t_*\in
        \left(
        -\frac{R^2}{a(\lambda)},\,
        -\frac{\mu_0}{2}\frac{R^2}{a(\lambda)}
        \right)
    \]
    such that
    \begin{equation}\label{propagation:good}
        \big|
        \{x\in B_R:D_kv(x,t_*)>\lambda/2\}
        \big|
        \leq
        \beta_0\,|B_R|.
    \end{equation}
    We fix an integer $m\geq3$, and apply \eqref{log:est:time},
    thus getting 
    \begin{equation}\label{propagation:logestimate}
    \begin{split}
        \sup_{t\in(t_*,0)}
        \int_{B_R\times\{t\}}
        \Psi_m^2(D_kv)\,\zeta^2\,dx
        &\leq
        \int_{B_R\times\{t_*\}}
        \Psi_m^2(D_kv)\,\zeta^2\,dx
        \\
        &\quad+
        C\,a(\lambda)
        \int_{t_*}^0\int_{B_R}
        \Psi_m(D_kv)\,|D\zeta|^2\,dxdt.
    \end{split}
    \end{equation}
    Owing to \eqref{Psim:prop} and \eqref{propagation:good},
    we have
    \begin{equation}\label{propagation:initial}
    \begin{split}
        \int_{B_R\times\{t_*\}}
        \Psi_m^2(D_kv)\,\zeta^2\,dx
        &\leq
        (m-1)^2(\log2)^2
        \big|
        \{x\in B_R:D_kv(x,t_*)>\lambda/2\}
        \big|
        \\
        &\leq
        \beta_0(m-1)^2(\log2)^2\,|B_R|.
    \end{split}
    \end{equation}
    Moreover, using again \eqref{Psim:prop},
    \eqref{propagation:cutoff}, and
    $-t_*\leq R^2/a(\lambda)$, we find
    \begin{equation}\label{propagation:error}
    \begin{split}
        a(\lambda)
        \int_{t_*}^0\int_{B_R}
        \Psi_m(D_kv)\,|D\zeta|^2\,dxdt
        &\leq
        \frac{C\,a(\lambda)(m-1)(-t_*)}
        {(1-\rho)^2R^2}\,|B_R|\leq
        \frac{C(m-1)}{(1-\rho)^2}\,|B_R|.
    \end{split}
    \end{equation}
    Combining \eqref{propagation:logestimate}--\eqref{propagation:error},
    we get
    \begin{equation}\label{propagation:Psi}
        \sup_{t\in(t_*,0)}
        \int_{B_R\times\{t\}}
        \Psi_m^2(D_kv)\,\zeta^2\,dx
        \leq
        \left[
        \beta_0(m-1)^2(\log2)^2
        +
        \frac{C(m-1)}{(1-\rho)^2}
        \right]|B_R|.
    \end{equation}

    Let us now set
    \[
        E_m(t)\coloneqq
        \left\{
        x\in B_R:
        D_kv(x,t)>(1-2^{-m})\lambda
        \right\}.
    \]
    On $E_m(t)$, by \eqref{def:Psim},
    \[
        \Psi_m(D_kv)
        \geq
        \log\left(
        \frac{\lambda/2}
        {\lambda/2-[(1-2^{-m})\lambda-\l/2]+2^{-m}\lambda}
        \right)
        =
        (m-2)\log2.
    \]
    Since $\zeta=1$ in $B_{\rho R}$, from
    \eqref{propagation:Psi} it follows that
    \begin{equation}\label{propagation:Em}
        |E_m(t)\cap B_{\rho R}|
        \leq
        \left[
        \beta_0
        \left(\frac{m-1}{m-2}\right)^2
        +
        \frac{C(m-1)}
        {(1-\rho)^2(m-2)^2}
        \right]|B_R|
    \end{equation}
    for all $t\in(t_*,0)$.
    The quantity in square brackets converges to $\beta_0$
    as $m\to\infty$. Therefore, we may choose an integer
    $m_0\geq3$, depending only on $\data$, so large that
    \[
        |E_{m_0}(t)\cap B_{\rho R}|
        \leq
        \frac{1+\beta_0}{2}\,|B_R|
        \leq
        \left(1-\frac{\mu_0}{4}\right)|B_R|,
    \]
    where in the last inequality we used that
    \[
        \frac{1+\beta_0}{2}
        =
        1-\frac{\mu_0}{2(2-\mu_0)}
        \leq
        1-\frac{\mu_0}{4}.
    \]
    Consequently, by \eqref{propagation:rho},
    \[
    \begin{split}
        |E_{m_0}(t)|
        &\leq
        |E_{m_0}(t)\cap B_{\rho R}|
        +
        |B_R\setminus B_{\rho R}|
        \\
        &\leq
        \left(1-\frac{\mu_0}{4}\right)|B_R|
        +
        \frac{\mu_0}{8}|B_R|
        =
        \left(1-\frac{\mu_0}{8}\right)|B_R|.
    \end{split}
    \]
    Taking complements in $B_R$, we obtain
    \[
        \big|
        \{x\in B_R:
        D_kv(x,t)\leq(1-2^{-m_0})\lambda\}
        \big|
        \geq
        \frac{\mu_0}{8}\,|B_R|
    \]
    for all $t\in(t_*,0)$. This proves \eqref{propagation:lower} and completes the proof.
\end{proof}

\begin{lemma}[Arbitrary smallness of upper level sets]\label{lem:small:measure}
    Under the notations and assumptions of Lemma \ref{lem:good:time}, set
    \[
        Q_*\coloneqq B_R(x_0)\times(t_*,t_0).
    \]
    Then, for every $\theta_0\in(0,1)$, there exists an integer
    $m_1=m_1(\data,\theta_0)$ such that,
    \begin{equation}\label{small:measure:goal}
        \big|
        Q_*\cap\{D_kv>(1-2^{-m_1})\,\lambda\}
        \big|
        \leq\theta_0\,|Q_*|,
    \end{equation}
\end{lemma}

\begin{proof}
    By translation, we may assume $(x_0,t_0)=(0,0)$.
    Set $I\coloneqq(t_*,0)$, so that $Q_*=B_R\times I$ and
    \[
        \frac{\mu_0}{2}\frac{R^2}{a(\lambda)}
        <|I|<
        \frac{R^2}{a(\lambda)}.
    \]
    We fix $\theta_0\in(0,1)$ and set
    \[
        \gamma\coloneqq\frac{\mu_0\theta_0}{16},
        \qquad
        \sigma\coloneqq(1-\gamma)^{1/N},
        \qquad
        J\coloneqq
        I\cap\left(-\frac{\sigma^2R^2}{a(\lambda)},0\right).
    \]
    Observe that $\sigma\in(1/2,1)$,
    $|B_R\setminus B_{\sigma R}|=\gamma|B_R|$, and
    \[
        1-\sigma^2\leq2\gamma,
        \qquad
        |I\setminus J|
        \leq\frac{(1-\sigma^2)R^2}{a(\lambda)}.
    \]
    Therefore,
    \begin{equation}\label{small:measure:boundary}
    \begin{split}
        |Q_*\setminus(B_{\sigma R}\times J)|
        &\leq
        |B_R\setminus B_{\sigma R}|\,|I|
        +|B_R|\,|I\setminus J|
        \\
        &\leq
        \left(\gamma+\frac{2(1-\sigma^2)}{\mu_0}\right)|Q_*|
        \\
        &\leq
        \frac{\mu_0+4}{16}\theta_0|Q_*|
        \leq\frac{\theta_0}{2}|Q_*|.
    \end{split}
    \end{equation}

    Let $m_0=m_0(\data)\in \N$ be given by
    Lemma \ref{lem:propagation}. For every integer $j\geq m_0$,
    we set
    \[
        \ell_j\coloneqq(1-2^{-j})\lambda,
        \qquad
        A_j\coloneqq
        (B_{\sigma R}\times J)\cap\{D_kv>\ell_j\}.
    \]
    By Lemma \ref{lem:propagation}, for every
    $j\geq m_0+1$ and every $t\in J$ we have
    \begin{equation}\label{small:measure:lower}
    \begin{split}
        \big|\{x\in B_{\sigma R}:D_kv(x,t)<\ell_j\}\big|
        &\geq
        \big|\{x\in B_R:D_kv(x,t)\leq\ell_{m_0}\}\big|
        -\gamma|B_R|
        \\
        &\geq
        \left(\frac{\mu_0}{8}-\gamma\right)|B_R|
        \geq\frac{\mu_0}{16}|B_R|,
    \end{split}
    \end{equation}
    where we used that $\ell_{m_0}\leq \ell_j$, $|B_R\setminus B_{\sigma R}|=\gamma |B_R|$ and
    $\gamma\leq\mu_0/16$.

    Applying \eqref{DGl2} with $r_1=\sigma R$, $r_2=R$
    and $\kappa=\ell_j$, and using
    $B_{\sigma R}\times J\subset Q_{\sigma R}^{(\lambda)}$,
    we obtain
    \begin{equation}\label{small:measure:energy}
    \begin{split}
        \iint_{A_j}|D(D_kv)|^2\,dxdt
        &\leq
        \frac{C(2^{-j}\lambda)^2}{(1-\sigma)^2R^2}
        \big|Q_R^{(\lambda)}\cap\{D_kv>\ell_j\}\big|
        \\
        &\leq
        \frac{C(2^{-j}\lambda)^2}{(1-\sigma)^2R^2}|Q_*|,
    \end{split}
    \end{equation}
    where $C=C(\data)>0$, and
    in the last inequality we used that $|Q_R^{(\lambda)}|\leq(2/\mu_0)|Q_*|$.

  For every $t\in J$ and every $j\geq m_0+1$, we apply
Lemma \ref{lem:isoper} to $D_kv(\cdot,t)$ in
$B_{\sigma R}$, with $\kappa=\ell_j$ and
$\ell=\ell_{j+1}$. Thus,
\[
\begin{split}
&(\ell_{j+1}-\ell_j)
\big|\{x\in B_{\sigma R}:
D_kv(x,t)>\ell_{j+1}\}\big|
\\
&\qquad\leq
C\,\frac{(\sigma R)^{N+1}}
{\big|\{x\in B_{\sigma R}:
D_kv(x,t)<\ell_j\}\big|}
\int_{B_{\sigma R}\cap
\{\ell_j<D_kv<\ell_{j+1}\}}
|D(D_kv)|\,dx.
\end{split}
\]
By \eqref{small:measure:lower} and since
$\sigma<1$, we have
\[
\begin{split}
\frac{(\sigma R)^{N+1}}
{\big|\{x\in B_{\sigma R}:
D_kv(x,t)<\ell_j\}\big|}
\leq
\frac{16(\sigma R)^{N+1}}
{\mu_0|B_R|}\leq C(\data)R.
\end{split}
\]
Recalling that
$\ell_{j+1}-\ell_j=2^{-j-1}\lambda$,
integrating over $J$, and applying
Fubini's theorem, we obtain
\[
\begin{split}
2^{-j-1}\lambda\,|A_{j+1}|
&\leq
CR\int_J\int_{B_{\sigma R}\cap
\{\ell_j<D_kv<\ell_{j+1}\}}
|D(D_kv)|\,dxdt
\\
&\leq
CR\iint_{A_j\setminus A_{j+1}}
|D(D_kv)|\,dxdt.
\end{split}
\]
Hence, by H\"older's inequality and
\eqref{small:measure:energy}, we find
\[
\begin{split}
2^{-j-1}\lambda\,|A_{j+1}|
&\leq
CR
\left(
\iint_{A_j}|D(D_kv)|^2\,dxdt
\right)^{1/2}
|A_j\setminus A_{j+1}|^{1/2}
\\
&\leq
\frac{C\,2^{-j}\lambda}{1-\sigma}
|Q_*|^{1/2}
|A_j\setminus A_{j+1}|^{1/2}.
\end{split}
\]
Consequently,
\begin{equation}\label{small:measure:iteration}
|A_{j+1}|^2
\leq
\frac{C}{(1-\sigma)^2}
|Q_*|\,|A_j\setminus A_{j+1}|,
\end{equation}
where $C=C(\data)>0$. Since the sets $A_j$ are decreasing, for every
integer $M\geq m_0+2$ and every
$j=m_0+1,\dots,M-1$, we have
$A_M\subset A_{j+1}$. Hence, by
\eqref{small:measure:iteration},
\[
\begin{split}
(M-m_0-1)|A_M|^2
&\leq
\sum_{j=m_0+1}^{M-1}|A_{j+1}|^2\leq
\frac{C}{(1-\sigma)^2}|Q_*|
\sum_{j=m_0+1}^{M-1}
|A_j\setminus A_{j+1}|
\\
&=
\frac{C}{(1-\sigma)^2}|Q_*|
|A_{m_0+1}\setminus A_M|\leq
\frac{C}{(1-\sigma)^2}|Q_*|^2.
\end{split}
\]
Consequently,
\begin{equation}\label{small:measure:decay}
    |A_M|
    \leq
    \frac{C}{(1-\sigma)\sqrt{M-m_0-1}}
    |Q_*|.
\end{equation}
We recall that
$\sigma^N=1-\gamma$, where
$\gamma=\mu_0\theta_0/16$. Since $\sigma<1$,
\[
    \gamma
    =1-\sigma^N
    =(1-\sigma)\sum_{h=0}^{N-1}\sigma^h
    \leq N(1-\sigma),
\]
and therefore
\[
    \frac{1}{1-\sigma}
    \leq\frac{N}{\gamma}
    =\frac{16N}{\mu_0\theta_0}.
\]
Inserting this inequality into
\eqref{small:measure:decay}, we obtain
\[
    |A_M|
    \leq
    \frac{C_1}{\theta_0\sqrt{M-m_0-1}}
    |Q_*|,
\]
where $C_1=C_1(\data)>0$.

We may thus choose an integer
$m_1\geq m_0+2$, depending only on
$\data$ and $\theta_0$, such that
\[
    m_1-m_0-1
    \geq 4\,(C_1)^2\,\theta_0^{-4},
\]
and with this choice, we infer that
\begin{equation}\label{small:measure:final}
    |A_{m_1}|
    \leq\frac{\theta_0}{2}|Q_*|.
\end{equation}
Finally, recalling the definition of
$A_{m_1}$ and using
\eqref{small:measure:boundary} together
with \eqref{small:measure:final},
we conclude that
\[
\begin{split}
\big|
Q_*\cap\{D_kv>&(1-2^{-m_1})\lambda\}
\big|\leq
|A_{m_1}|
+
|Q_*\setminus(B_{\sigma R}\times J)|
\\
&\leq
\frac{\theta_0}{2}|Q_*|
+
\frac{\theta_0}{2}|Q_*|
=\theta_0|Q_*|.
\end{split}
\]
This completes the proof.
\end{proof}

\begin{lemma}\label{lem:quasi:2alt}
    Under the notation and assumptions of
    Lemmas \ref{lem:good:time}, there exist
    $\delta_0,\eta_0\in(0,1)$, depending only
    on $\data$, such that
    \begin{equation}\label{part:2alt}
        D_kv(x,t)\leq\eta_0\lambda
        \quad\text{for all }
        (x,t)\in Q_{\delta_0R}^{(\lambda)}(x_0,t_0).
    \end{equation}
\end{lemma}

\begin{proof}
    By translation, we may assume $(x_0,t_0)=(0,0)$.
    Set
    \[
        \delta_0=\delta_0(\data)
        \coloneqq\frac{\sqrt{\mu_0}}{4}.
    \]
    By \eqref{good:time:interval}, we have
    \[
        Q_{2\delta_0R}^{(\lambda)}
        \subset Q_*\coloneqq B_R\times(t_*,0).
    \]
    Fix $\theta_0=\theta_0(\data)\in(0,1)$,
    to be chosen below, and let
    $m_1=m_1(\data,\theta_0)$ be given by
    Lemma \ref{lem:small:measure}. Then
    \begin{equation}\label{quasi:initial}
        \big|Q_*\cap
        \{D_kv>(1-2^{-m_1})\lambda\}\big|
        \leq\theta_0|Q_*|.
    \end{equation}

    For $j=0,1,\dots$, we set
    \[
        r_j\coloneqq\delta_0R(1+2^{-j}),
        \qquad
        \ell_j\coloneqq
        \left(1-2^{-m_1-1}(1+2^{-j})\right)\lambda,
    \]
    and
    \[
        A_j\coloneqq
        Q_{r_j}^{(\lambda)}\cap\{D_kv>\ell_j\}.
    \]
    Observe that
    \[
        r_0=2\delta_0R,\qquad
        r_j\searrow\delta_0R,
        \qquad
        \ell_j\nearrow(1-2^{-m_1-1})\lambda,
    \]
    and
    \[
        r_j-r_{j+1}
        =\delta_0\,2^{-j-1}R,
        \qquad
        \ell_{j+1}-\ell_j
        =2^{-m_1-j-2}\lambda.
    \]
    In particular, $\lambda/8\leq\ell_j<\lambda$. Let
    \[
        \widehat r_j\coloneqq
        \frac{r_j+r_{j+1}}{2},
        \qquad
        w_j\coloneqq(D_kv-\ell_j)_+.
    \]
    Applying \eqref{DGl2} with
    $r_1=\widehat r_j$, $r_2=r_j$,
    and $\kappa=\ell_j$, we obtain
    \begin{equation}\label{quasi:energy1}
        \sup_{t\in
        (-\widehat r_j^2/a(\lambda),0)}
        \int_{B_{\widehat r_j}}w_j^2\,dx
        +
        a(\lambda)
        \iint_{Q_{\widehat r_j}^{(\lambda)}}
        |Dw_j|^2\,dxdt\leq
        C\,4^j\,
        \frac{a(\lambda)(\lambda-\ell_j)^2}{R^2}
        |A_j|,
    \end{equation}
    where $C=C(\data)>0$. Next choose $\psi_j\in C_c^\infty(B_{\widehat r_j})$
    such that
    \[
        0\leq\psi_j\leq1,\qquad
        \psi_j\equiv1
        \quad\text{in }B_{r_{j+1}},
        \qquad
        |D\psi_j|
        \leq\frac{C(N)}{r_j-r_{j+1}}.
    \]
    Since, by \eqref{intr:hup},
    \[
        0\leq w_j\leq
        (\lambda-\ell_j)\mathbbm{1}_{A_j}
        \quad\text{in }Q_{r_j}^{(\lambda)},
    \]
    using \eqref{quasi:energy1} and the
    properties of $\psi_j$, we infer
    \begin{equation}\label{quasi:energy}
    \begin{split}
        \sup_{t\in
        (-\widehat r_j^2/a(\lambda),0)}
        \int_{B_{\widehat r_j}}
        (w_j\psi_j)^2\,dx+
        &a(\lambda)
        \iint_{Q_{\widehat r_j}^{(\lambda)}}
        |D(w_j\psi_j)|^2\,dxdt
        \\
       & \leq
        C\,4^j
        \frac{a(\lambda)(\lambda-\ell_j)^2}{R^2}
        |A_j|.
    \end{split}
    \end{equation}

    We now set $\chi\coloneqq1+2/N$.
    Applying the parabolic Sobolev inequality
    \cite[I. Proposition 3.1]{DiBenedetto1993}
    (with $p=m=2$) to $w_j\psi_j$, and using
    \eqref{quasi:energy}, we find
    \[
    \begin{split}
        &\iint_{Q_{\widehat r_j}^{(\lambda)}}
        |w_j\psi_j|^{2\chi}\,dxdt
        \\
        &\quad\leq
        C\left(
        \sup_{t\in
        (-\widehat r_j^2/a(\lambda),0)}
        \int_{B_{\widehat r_j}}
        (w_j\psi_j)^2\,dx
        \right)^{2/N}
        \iint_{Q_{\widehat r_j}^{(\lambda)}}
        |D(w_j\psi_j)|^2\,dxdt
        \\
        &\quad\leq
        \frac{C}{a(\lambda)}
        \left(
        4^j\frac{a(\lambda)
        (\lambda-\ell_j)^2}{R^2}
        |A_j|
        \right)^\chi.
    \end{split}
    \]
    On the other hand, by the definition of
    $w_j$ and the choice of $\psi_j$, we have
    \[
        w_j\psi_j\geq
        \ell_{j+1}-\ell_j
        \qquad\text{on }A_{j+1}.
    \]
    Therefore,
    \[
        (\ell_{j+1}-\ell_j)^{2\chi}
        |A_{j+1}|
        \leq
        \frac{C}{a(\lambda)}
        \left(
        4^j\frac{a(\lambda)
        (\lambda-\ell_j)^2}{R^2}
        |A_j|
        \right)^\chi,
    \]
    and since $ \frac{\lambda-\ell_j}
        {\ell_{j+1}-\ell_j}
        \leq2^{j+2}$, we conclude that
\begin{equation}\label{quasi:measure}
        |A_{j+1}|
        \leq
        C_0\,\mathfrak b^j\,
        a(\lambda)^{2/N}R^{-2\chi}
        |A_j|^\chi,
    \end{equation}
    where $C_0>0$ and $\mathfrak b>1$
    depend only on $\data$.

    We now normalize the measures by setting
    \[
        Y_j\coloneqq
        \frac{|A_j|}
        {|Q_{2\delta_0R}^{(\lambda)}|}.
    \]
    Since $\delta_0=\delta_0(\data)$,
    \eqref{quasi:measure} yields
    \begin{equation}\label{quasi:Yrec}
        Y_{j+1}\leq
        C_1\,\mathfrak b^j\,Y_j^{1+2/N},
    \end{equation}
    where $C_1=C_1(\data)>0$.
    Moreover, by \eqref{quasi:initial}
    and \eqref{good:time:interval}, we have
    \[
        Y_0
        \leq\theta_0
        \frac{|Q_*|}
        {|Q_{2\delta_0R}^{(\lambda)}|}
        \leq C(\data)\theta_0.
    \]
    Choosing $\theta_0=\theta_0(\data)>0$
    sufficiently small, Lemma \ref{lem:hyp}
    applied to \eqref{quasi:Yrec} yields
    $Y_j\longrightarrow0$.

    Since $r_j\searrow\delta_0R$ and
    $\ell_j\nearrow(1-2^{-m_1-1})\lambda$,
    we infer
    \[
        \big|
        Q_{\delta_0R}^{(\lambda)}
        \cap
        \{D_kv>(1-2^{-m_1-1})\lambda\}
        \big|=0.
    \]
which implies
    \eqref{part:2alt} with
    \[
        \eta_0=\eta_0(\data)\coloneqq
        1-2^{-m_1-1}\in(0,1).
    \]
This completes the proof.
\end{proof}

We are now in the position to prove the second alternative.

\begin{proof}[Proof of Proposition \ref{prop:fund:alt} (II):
the second alternative]
    Suppose that \eqref{alt:3} holds.
    Applying Lemma \ref{lem:quasi:2alt} to $v$
    and $\widetilde v\coloneqq-v$, which solves an
    equation satisfying the same structural assumptions,
    we obtain
    \[
        -\eta_0\lambda\leq D_kv\leq\eta_0\lambda
        \qquad\text{in }Q_{\delta_0R}^{(\lambda)},
    \]
    for every $k=1,\dots,N$. This proves \eqref{thesis:alt3}, that is our thesis.
\end{proof}

\subsubsection{Proof of Theorem \ref{thm:gradhold}
with \eqref{smooth} in force}\label{f:subsec}

Let us first recall the oscillation decay estimate for
uniformly parabolic equations. We state it in the linear
form needed below.

\begin{lemma}[Uniformly parabolic oscillation decay]
\label{lem:par:osc}
    Let $w$ be a local weak solution to
    \[
        \partial_t w-
        a(\lambda)\operatorname{div}
        \big(\mathbb A(x,t)Dw\big)=0
        \quad\text{in }Q_R^{(\lambda)},
    \]
    where $\lambda>0$ and $\mathbb A$ is a
    measurable matrix satisfying
    \[
        c_*|\xi|^2
        \leq\mathbb A(x,t)\xi\cdot\xi,
        \qquad
        |\mathbb A(x,t)|\leq C_*
    \]
    for a.e. $(x,t)\in Q_R^{(\lambda)}$
    and every $\xi\in\R^N$, with
    $0<c_*\leq C_*<\infty$.

    Then there exist $\beta_g\in(0,1)$
    and $C_g\geq1$, depending only on
    $N,c_*,C_*$, such that
    \begin{equation}\label{par:osc:linear}
        \operatorname*{osc}_{Q_r^{(\lambda)}}w
        \leq C_g
        \left(\frac rR\right)^{\beta_g}
        \operatorname*{osc}_{Q_R^{(\lambda)}}w
    \end{equation}
    for every $0<r\leq R$.
\end{lemma}

These estimates are standard consequences of the
De Giorgi--Nash--Moser theory for uniformly
parabolic equations; see, e.g.,
\cite[Chapter 12, Theorems 2.1 and 3.1]
{DiBenedettoGianazza2023}.
Notice that the constants are independent of
$a(\lambda)$, as follows by rescaling the
time variable.

\begin{proof}[Proof of Theorem
\ref{thm:gradhold} under \eqref{smooth}]

\noindent\textit{(i) Intrinsic oscillation estimate.}
Let $Q_R^{(\lambda)}\Subset\Omega_T$ be such that \eqref{intr:hup0} holds.
We prove \eqref{osc:intr} by combining
Proposition \ref{prop:fund:alt} with
Lemma \ref{lem:par:osc}.

So let $\eta_0,\delta_0\in(0,1)$ be the constants
given by Proposition \ref{prop:fund:alt}.
We set
\[
    K=K(\data)\coloneqq\max\{1,\eta_0^{-s_a/2}\},
\]
and we fix $\theta=\theta(\data)\in(0,1)$
sufficiently small so that
\begin{equation}\label{hold:theta}
    \theta K\leq
    \min\left\{\delta_0,\frac14\right\},
    \qquad
    C_g(2\theta K)^{\beta_g}\leq\eta_0,
    \qquad
    \theta<\eta_0,
\end{equation}
where $C_g,\beta_g$ are the constants from
Lemma \ref{lem:par:osc}, with ellipticity
constants $c_*,C_*$ depending only on $\data$,
as specified in \eqref{determined} below.

For $j=0,1,\dots$, define
\[
    R_j\coloneqq\theta^jR,
    \qquad
    \lambda_j\coloneqq\eta_0^j\lambda.
\]
Observe that, by \eqref{ciu:a}, $\frac{a(\lambda_j)}
    {a(\lambda_{j+1})}
    \leq\eta_0^{-s_a}\leq K^2$. Consequently, our choice of $\theta$ implies
\begin{equation}\label{hold:nesting}
    Q_{R_{j+1}}^{(\lambda_{j+1})}
    \subset
    Q_{\delta_0R_j}^{(\lambda_j)}
    \subset Q_{R_j}^{(\lambda_j)}
\end{equation}
for every $j\geq0$. We claim that
\begin{equation}\label{hold:discrete}
    \operatorname*{osc}_{
    Q_{R_j}^{(\lambda_j)}}Dv
    \leq C\lambda_j,
    \qquad j=0,1,\dots,
\end{equation}
where $C=C(\data)>0$.

Suppose first that the second alternative of
Proposition \ref{prop:fund:alt} holds at every
step. By \eqref{thesis:alt3} and
\eqref{hold:nesting}, we obtain inductively
\[
    \max_{k=1,\dots,N}
    \sup_{Q_{R_j}^{(\lambda_j)}}|D_kv|
    \leq\lambda_j.
\]
Therefore,
\[
    \operatorname*{osc}_{
    Q_{R_j}^{(\lambda_j)}}Dv
    \leq2\sqrt N\,\lambda_j,
\]
which proves \eqref{hold:discrete} in this case.

Suppose instead that the first alternative
holds at some step, and let $j_0\geq0$
be the first such integer.
Until this step, the second alternative
and \eqref{hold:nesting} give
\[
    \max_{k=1,\dots,N}
    \sup_{Q_{R_j}^{(\lambda_j)}}|D_kv|
    \leq\lambda_j,
    \qquad j=0,\dots,j_0.
\]
By the first alternative at step $j_0$
and \eqref{thalt1:modDu}, we obtain
\[
    \frac{\lambda_{j_0}}4
    \leq |Dv|
    \leq\sqrt N\,\lambda_{j_0}
    \qquad\text{in }
    Q_{R_{j_0}/2}^{(\lambda_{j_0})}.
\]
Hence, by \eqref{ultra:utile} and \eqref{ass:Ahom},
\begin{equation}\label{determined}
    c_*a(\lambda_{j_0})|\xi|^2
    \leq
    \nabla_\xi\A_0(t,Dv)\xi\cdot\xi,
    \qquad
    |\nabla_\xi\A_0(t,Dv)|
    \leq C_* a(\lambda_{j_0})
\end{equation}
in $Q_{R_{j_0}/2}^{(\lambda_{j_0})}$,
where $c_*,C_*>0$ depend only on $\data$.

Differentiating \eqref{eq:hom} with
respect to $x_k$, we see that
$w_k\coloneqq D_kv$ solves
\[
    \partial_t w_k-
    \operatorname{div}\big(
    \nabla_\xi\A_0(t,Dv)Dw_k\big)=0.
\]
Thus $w_k$ satisfies a uniformly parabolic
linear equation in
$Q_{R_{j_0}/2}^{(\lambda_{j_0})}$,
and Lemma \ref{lem:par:osc} applies.

For every $j>j_0$, by
\eqref{ciu:a} we have $ \frac{a(\lambda_{j_0})}{a(\lambda_j)}
    \leq \eta_0^{-s_a(j-j_0)}\leq  K^{2(j-j_0)}$,
and since $\theta K\leq1/4$, we have 
\[
    Q_{R_j}^{(\lambda_j)}
    \subset
    Q_{(\theta K)^{j-j_0}R_{j_0}}
    ^{(\lambda_{j_0})}\subset
    Q_{R_{j_0}/2}^{(\lambda_{j_0})}.
\]
Applying
\eqref{par:osc:linear} to every component $w_k=D_k v$ 
of $Dv$ yields
\[
\begin{split}
  &\operatorname*{osc}_{Q_{R_j}^{(\lambda_j)}}Dv\leq
    \operatorname*{osc}_{
    Q_{(\theta K)^{j-j_0}R_{j_0}}^{(\lambda_{j_0})}}
    Dv\leq
    C_g\left(
    \frac{(\theta K)^{j-j_0}R_{j_0}}
    {R_{j_0}/2}
    \right)^{\beta_g}
    \operatorname*{osc}_{
    Q_{R_{j_0}/2}^{(\lambda_{j_0})}}Dv
    \\
    &\quad=
    C_g\big(2(\theta K)^{j-j_0}\big)^{\beta_g}
    \operatorname*{osc}_{
    Q_{R_{j_0}/2}^{(\lambda_{j_0})}}Dv\leq
    C(\data)
    (\theta K)^{(j-j_0)\beta_g}
    \lambda_{j_0}
    \\
    &\quad\leq C(\data)\eta_0^{j-j_0}\lambda_{j_0}
    =C(\data)\lambda_j,
\end{split}
\]
where in the last inequality we used
\eqref{hold:theta}. This proves
\eqref{hold:discrete} also in the second case.

We now pass from the discrete estimate
\eqref{hold:discrete} to arbitrary radii.
Set
\[
    \vartheta=\vartheta(\data)\coloneqq
    \theta\,\eta_0^{(-i_a)/2}\in(0,1).
\]
By \eqref{ciu:a}, we have $\frac{a(\lambda_j)}{a(\lambda)}
    \leq \eta_0^{ji_a}
    =\eta_0^{-j(-i_a)}$,
and consequently
\begin{equation}\label{hold:fixed}
    Q_{\vartheta^jR}^{(\lambda)}
    \subset Q_{R_j}^{(\lambda_j)}
    \qquad\text{for every }j\geq0.
\end{equation}
Given $0<r\leq R$, choose $j\geq0$
such that
\[
    \vartheta^{j+1}R<r\leq\vartheta^jR.
\]
Then, using \eqref{hold:fixed}
and \eqref{hold:discrete}, we obtain
\[
\begin{split}
    \operatorname*{osc}_{Q_r^{(\lambda)}}Dv
    &\leq
    \operatorname*{osc}_{
    Q_{R_j}^{(\lambda_j)}}Dv
    \leq C\eta_0^j\lambda\leq
    C\lambda\left(\frac rR\right)^\beta,
\end{split}
\]
where, recalling $\vartheta<\eta_0<1$, we set
\[
    \beta=\beta(\data)\coloneqq
    \frac{\log\eta_0}{\log\vartheta}\in(0,1).
\]
This proves \eqref{osc:intr}.

\medskip
\noindent\textit{(ii) Oscillation estimate on
standard parabolic cylinders.}
Let $Q_R=Q_R(x_0,t_0)\Subset\Omega_T$,
and set
\[
    M_0\coloneqq
    \|Dv\|_{L^\infty(Q_R)}.
\]
If $M_0=0$, the conclusion is immediate.
We may therefore assume $M_0>0$.

We distinguish two cases, according to
the size of $a(M_0)$.

If $a(M_0)\geq1$, then
\[
    Q_R^{(M_0)}\subset Q_R.
\]
For every $0<r\leq R/a(M_0)^{1/2}$,
we also have
\[
    Q_r\subset
    Q_{r\,a(M_0)^{1/2}}^{(M_0)}.
\]
Applying \eqref{osc:intr} with
$\lambda=M_0$ gives
\[
    \operatorname*{osc}_{Q_r}Dv
\leq
    \operatorname*{osc}_{
    Q_{r\,a(M_0)^{1/2}}^{(M_0)}}Dv\leq
    CM_0\left(
    \frac{r\,a(M_0)^{1/2}}{R}
    \right)^\beta.
\]
For $R/a(M_0)^{1/2}<r<R$, the same
estimate follows, after increasing $C$,
from the trivial bound
$\operatorname*{osc}_{Q_r}Dv\leq2M_0$.

Suppose now that $a(M_0)<1$, and set
\[
    \rho\coloneqq R\,a(M_0)^{1/2}.
\]
Then $Q_\rho^{(M_0)}\subset Q_R$,
and, for every $0<r\leq\rho$,
\[
    Q_r\subset Q_r^{(M_0)}.
\]
Therefore, again by \eqref{osc:intr},
\[
\begin{split}
    \operatorname*{osc}_{Q_r}Dv
    \leq
    \operatorname*{osc}_{Q_r^{(M_0)}}Dv\leq
    CM_0\left(\frac r\rho\right)^\beta
    =
    CM_0\left(
    \frac{r}{R\,a(M_0)^{1/2}}
    \right)^\beta.
\end{split}
\]
For $\rho<r<R$, the same bound follows
from $\operatorname*{osc}_{Q_r}Dv\leq2M_0$. Recalling that
\[
    a^\star(M_0)
    =\max\big\{
    a(M_0)^{1/2},a(M_0)^{-1/2}
    \big\},
\]
we conclude in both cases that
\[
    \operatorname*{osc}_{Q_r}Dv
    \leq
    CM_0\left(
    \frac rR\,a^\star(M_0)
    \right)^\beta,
    \qquad 0<r<R,
\]
which is precisely \eqref{osc:extr}. The proof is complete.
\end{proof}

\section{Proof of Theorem
\ref{main:thm}}\label{sec:perturb}

In this section, we complete the proof of Theorem \ref{main:thm}. In order to justify the computations, we will work under
the additional assumption
\begin{equation}\tag{S1}\label{pert:smooth}
\begin{split}
    &\text{the functions }u,\ \A,\ f\text{ are smooth, and}
    \\
    &c\,|\eta|^2\leq \nabla_\xi\A(x,t,\xi)\eta\cdot\eta,\qquad |\nabla_\xi\A(x,t,\xi)|\leq C,
    \end{split}
\end{equation}
for some constants $0<c\leq C$, and all $(x,t,\xi)$ and all $\eta\in \R^N$.
Again, assumption \eqref{pert:smooth} will be removed by means of a suitable approximation procedure in Section \ref{sec:approximation}.

\begin{lemma}[Oscillation lemma]\label{lem:pert:gluing}
    Let $u$ be a local weak solution to \eqref{eq:1} under \eqref{iasa}-\eqref{ass:f} and \eqref{pert:smooth}.  Suppose that
    $Q_{R,S}=Q_{R,S}(x_0,t_0)\Subset\Omega_T$ and
    $\|Du\|_{L^\infty(Q_{R,S})}\leq\lambda$, with $\lambda>0$.
    Then
    \begin{equation}\label{pert:gluing-general}
    \begin{split}
        \operatorname*{osc}_{Q_{R,S}}u
        \leq C\bigg[R\lambda+
        \frac{S}{R}\Big\{b(\lambda)
        +\Lambda_{\rm h}R^\alpha(1+b(\lambda))\Big\}+S^{1-1/d}R^{-N/d}\|f\|_{L^d(Q_{R,S})}\bigg],
    \end{split}
    \end{equation}
    where $C=C(\data,d)>0$.
    In particular, for $S=R^2/a(\lambda)$,
    \begin{equation}\label{pert:gluing-intrinsic}
        \operatorname*{osc}_{Q_R^{(\lambda)}}u
        \leq C\lambda R\big(1+\mathcal P_\lambda(R)\big),
    \end{equation}
    where 
    \begin{equation}\label{pert:size}
        \mathcal P_\lambda(R)\coloneqq
        \Lambda_{\rm h}R^\alpha\frac{1+b(\lambda)}{b(\lambda)}
        +R^{1-\frac{N+2}{d}}
        \frac{a(\lambda)^{1/d}}{b(\lambda)}
        \|f\|_{L^d(Q_R^{(\lambda)})}.
    \end{equation}
\end{lemma}

\begin{proof}
   We argue as in the proof of Lemma \ref{lem:osc:v}.  Choose a nonnegative $\eta\in C_c^\infty(B_R(x_0))$ such that
    \[
        \int_{B_R(x_0)}\eta\,dx=1,\qquad
        \|\eta\|_\infty\leq CR^{-N},\qquad
        \|D\eta\|_{L^1}\leq C/R,
    \]
    and set $m(t)\coloneqq\int_{B_R(x_0)}u(x,t)\eta(x)\,dx$.
    Since $|Du|\leq\lambda$, we have
    \begin{equation}\label{pert:spatial-gluing}
        |u(x,t)-m(t)|
        \leq\int_{B_R(x_0)}|u(x,t)-u(y,t)|\eta(y)\,dy
        \leq2\lambda R.
    \end{equation}
    On the other hand, the weak formulation of \eqref{eq:1} gives
    \[
        m'(t)=-\int_{B_R(x_0)}
        \big(\A(x,t,Du)-\A(x_0,t,0)\big)\cdot D\eta\,dx
        +\int_{B_R(x_0)}f\eta\,dx.
    \]
    where we also used that $\int_{B_R(x_0)}\A(x_0,t,0)\cdot D\eta\,dx=0$.
    By \eqref{A:hold}, \eqref{co:gr}, and the monotonicity of $b(t)$,
    \[
        |\A(x,t,Du)-\A(x_0,t,0)|
        \leq Cb(\lambda)+\Lambda_{\rm h}R^\alpha(1+b(\lambda)).
    \]
    Integrating over $t_0-S\leq s<t\leq t_0$, and applying
    H\"older's inequality to the source term, we obtain
    \[
        |m(t)-m(s)|\leq
        C\frac{S}{R}\Big\{b(\lambda)
        +\Lambda_{\rm h}R^\alpha(1+b(\lambda))\Big\}+CS^{1-1/d}R^{-N/d}\|f\|_{L^d(Q_{R,S})}.
    \]
    Combining this estimate with \eqref{pert:spatial-gluing}
    proves \eqref{pert:gluing-general}.
    Finally, if $S=R^2/a(\lambda)$, then
    \[
        \frac{S}{R}b(\lambda)=\lambda R,
        \qquad
        \frac{S^{1-1/d}R^{-N/d}}{\lambda R}
        =R^{1-\frac{N+2}{d}}\frac{a(\lambda)^{1/d}}{b(\lambda)},
    \]
    which gives \eqref{pert:gluing-intrinsic}.
\end{proof}

\subsection{Freezing the coefficients and comparison}
\label{subsec:pert:comparison}
In the following section, we consider a cylinder $Q^{(\l)}_R=Q_R^{(\lambda)}(x_0,t_0)$ as in \eqref{def:cilintr} such that
    \[Q_R^{(\lambda)}\Subset\Omega_T,
        \qquad\lambda\geq1,
        \qquad\|Du\|_{L^\infty(Q_R^{(\lambda)})}\leq\lambda.
    \]

    Let $v$ be the weak solution to
    \begin{equation}\label{pert:frozen}
        \begin{cases}
            \partial_t v-\mathrm{div}\big(\A(x_0,t,Dv)\big)=0
            &\text{in }Q_R^{(\lambda)},\\
            v=u&\text{on }\partial_{\rm par}Q_R^{(\lambda)}.
        \end{cases}
    \end{equation}

We establish the following

\begin{lemma}[Comparison lemma]\label{lem:pert:comparison}
 Assume that \eqref{iasa}-\eqref{ass:f} and \eqref{pert:smooth} hold.
    There exists a unique weak solution $v$ to \eqref{pert:frozen} such that $v\in C^\infty(Q^{(\l)}_R)$.
    
    Moreover, there exists $c_0=c_0(\data)>0$ such that, if

    \begin{equation}\label{lh:f}
    \Lambda_{\rm h}R^\alpha\leq c_0\quad\text{and}\quad    R^{1-\frac{N+2}{d}}
        \frac{a(\lambda)^{1/d}}{b(\lambda)}
        \|f\|_{L^d(Q_R^{(\lambda)})}\leq1,
    \end{equation}
    then
    \begin{equation}\label{pert:frozen-gradient}
        \|Dv\|_{L^\infty(Q_{R/2}^{(\lambda)})}\leq C\lambda,
    \end{equation}
        \begin{equation}\label{pert:frozen-oscillation}
        \operatorname*{osc}_{Q_r^{(\lambda)}}Dv
        \leq C\lambda\left(\frac rR\right)^\beta,
        \qquad0<r\leq R/2,
    \end{equation}
    and
    \begin{equation}\label{pert:L2-comparison}
         \miint_{Q_{R/2}^{(\lambda)}}
    |Du-Dv|^2\,dxdt
    \leq C\lambda^2\mathcal P_\lambda(R)^{2/s_B}.
    \end{equation}
    Here $C$ depends only on $\data$ and $d$.
\end{lemma}

\begin{proof}
    Existence, uniqueness and smoothness of $v$ follow from the monotone evolution
    theory and schauder theory for parabolic equations--see Section \ref{subsec:pert:approximation}.
    By the parabolic maximum principle, $v$ takes values between the
    infimum and the supremum of $u$ on the parabolic boundary.
    Thus, by Lemma \ref{lem:pert:gluing}, \eqref{pert:uniform-ratios}
    and \eqref{lh:f},
    \begin{equation}\label{pert:frozen-values}
        \operatorname*{osc}_{Q_R^{(\lambda)}}v
        \leq\operatorname*{osc}_{Q_R^{(\lambda)}}u
        \leq C\lambda R,
        \qquad |u-v|\leq C\lambda R.
    \end{equation}
    Applying \eqref{gr:bd2} with $S=R^2/a(\lambda)$ gives
    \[
    \begin{split}
        \|Dv\|_{L^\infty(Q_{R/2}^{(\lambda)})}
        &\leq C\left[
        \frac{\operatorname*{osc}_{Q_R^{(\lambda)}}v}{R}
        +B^{-1}\left(
        \frac{a(\lambda)}{R^2}
        (\operatorname*{osc}_{Q_R^{(\lambda)}}v)^2\right)\right]\\
        &\leq C\left[\lambda+B^{-1}\big(Ca(\lambda)\lambda^2\big)\right]
        \leq C\lambda,
    \end{split}
    \]
    where the last inequality follows from \eqref{Bcomeb} and
    \eqref{B2t}. This proves \eqref{pert:frozen-gradient}.

    We next subtract Equations \eqref{eq:1} and \eqref{pert:frozen}, and test with $u-v$.
    Since $u-v$ vanishes on the parabolic boundary, dropping the
    nonnegative term  $\frac12\int_{B_R(x_0)}
|u(x,t_0)-v(x,t_0)|^2\,dx$ gives
    \begin{equation}\label{pert:comparison-identity}
    \begin{split}
        &\iint_{Q_R^{(\lambda)}}
        \big(\A(x_0,t,Du)-\A(x_0,t,Dv)\big)\cdot(Du-Dv)\,dxdt\\
        &\quad\leq
        \iint_{Q_R^{(\lambda)}}
        \big(\A(x_0,t,Du)-\A(x,t,Du)\big)\cdot(Du-Dv)\,dxdt
        +\iint_{Q_R^{(\lambda)}} f(u-v)\,dxdt.
    \end{split}
    \end{equation}
By \eqref{strong:coer}, we have
\begin{equation}\label{luc}
\begin{split}
    \iint_{Q_R^{(\lambda)}}&
        \big(\A(x_0,t,Du)-\A(x_0,t,Dv)\big)\cdot(Du-Dv)\,dxdt
        \\
        &\geq c\,\iint_{Q_R^{(\lambda)}}a(|Du|+|Dv|)|Du-Dv|^2\,dxdt. 
        \end{split}
\end{equation}
Since $\lambda\geq 1$, by \eqref{B=1} and the monotonicity of $B$ we have $B(\lambda)\geq c$. Thus, we may estimate
\[
|Du-Dv|\leq C(B(\lambda)+B(|Du-Dv)),
\]
whereas by choosing the parameter $\delta$ in \eqref{el:comparison} sufficiently
    small to absorb $B(|Dv|)$, and since $B(|Du|)\leq B(\lambda)$ we have
    \begin{equation}\label{pert:algebraic-comparison}
        B(|Du-Dv|)+B(|Dv|)
        \leq CB(\lambda)
        +Ca(|Du|+|Dv|)|Du-Dv|^2.
    \end{equation}
    Then Young's inequality \eqref{young1} and \eqref{wBbt} imply
    \[
    \begin{split}
        b(\lambda)|Du-Dv|
        &\leq B(|Du-Dv|)+\widetilde B(b(\lambda))\\
        &\leq CB(\lambda)
        +Ca(|Du|+|Dv|)|Du-Dv|^2.
    \end{split}
    \]
    Therefore, by \eqref{A:hold} and the estimates above, we infer
    \begin{equation}\label{pert:coefficient-error}
    \begin{split}
        &\big|\A(x_0,t,Du)-\A(x,t,Du)\big|\,|Du-Dv|\\
        &\quad\leq
        C\Lambda_{\rm h}R^\alpha
        \frac{1+b(\lambda)}{b(\lambda)}
        \big[B(\lambda)+a(|Du|+|Dv|)|Du-Dv|^2\big].
    \end{split}
    \end{equation}
    Moreover, by \eqref{pert:frozen-values} and H\"older's inequality,
    \begin{equation}\label{pert:source-error}
    \begin{split}
        \miint_{Q_R^{(\lambda)}}|f(u-v)|\,dxdt
        &\leq C\lambda R\,|Q_R^{(\lambda)}|^{-1/d}
        \|f\|_{L^d(Q_R^{(\lambda)})}\\
        &\leq CB(\lambda)R^{1-\frac{N+2}{d}}
        \frac{a(\lambda)^{1/d}}{b(\lambda)}
        \|f\|_{L^d(Q_R^{(\lambda)})}.
    \end{split}
    \end{equation}
    Inserting \eqref{luc},  \eqref{pert:coefficient-error} and
    \eqref{pert:source-error} into \eqref{pert:comparison-identity},
    and choosing $c_0$ sufficiently small to re-absorb the term
    $a(|Du|+|Dv|)|Du-Dv|^2$, we find
    \begin{equation}\label{pert:weighted-comparison}
        \miint_{Q_R^{(\lambda)}}
        a(|Du|+|Dv|)|Du-Dv|^2\,dxdt
        \leq CB(\lambda)\mathcal P_\lambda(R).
    \end{equation}
    On $Q_{R/2}^{(\lambda)}$, we have $|Du|+|Dv|\leq C\lambda$.
  The inequalities \eqref{ciu:a} and \eqref{Bcomeb} thus give
    \[
    \begin{split}
        a(|Du|+|Dv|)|Du-Dv|^2
        &\geq c\,a(\lambda)
        \left(\frac{|Du-Dv|}{\lambda}\right)^{s_a}|Du-Dv|^2\\
        &\geq cB(\lambda)
        \left(\frac{|Du-Dv|}{\lambda}\right)^{s_B},
    \end{split}
    \]
 which together with \eqref{pert:weighted-comparison}, upon restricting to $Q^{(\l)}_{R/2}$,  yields 
 \[
 \miint_{Q_{R/2}^{(\lambda)}}|Du-Dv|^{s_B}\,dxdt
        \leq C\lambda^{s_B}\mathcal P_\lambda(R)
 \]
so \eqref{pert:L2-comparison} follows by H\"older's inequality as $s_B\geq 2$.

Finally, let $C_0=C_0(\data,d)\geq1$ be the constant in
\eqref{pert:frozen-gradient}, and set $V\coloneqq v/C_0$.
Then
\[
    \partial_t V-
    \mathrm{div}\big(\widetilde{\A}(t,DV)\big)=0,
    \qquad
    \widetilde{\A}(t,\xi)\coloneqq
    C_0^{-1}\A(x_0,t,C_0\xi),\qquad  \sup_{Q_{R/2}^{(\lambda)}}|DV|\leq\lambda.
\]
Moreover,  by \eqref{ciu:a}, 
the vector field $\widetilde{\A}$ satisfies
\eqref{ass:Ahom} with the same function $a$ and
structural constants depending only on $\data$ and $d$.
Therefore, Theorem \ref{thm:gradhold} applied to $V$
on $Q_{R/2}^{(\lambda)}$ yields
\[
\begin{split}
    \operatorname*{osc}_{Q_r^{(\lambda)}}Dv
    &=C_0\operatorname*{osc}_{Q_r^{(\lambda)}}D V\leq C(\data,d)\,\lambda
    \left(\frac rR\right)^\beta,
    \qquad 0<r\leq R/2,
\end{split}
\]
where $\beta=\beta(\data,C_0)\in(0,1)$. This proves \eqref{pert:frozen-oscillation}, and completes the proof.

\end{proof}

\subsection{Proof of Theorem \texorpdfstring{\ref{main:thm}}{main theorem}}\label{subsec:pf:mainthm} In the following section, we establish the following:

\begin{proof}[Proof of Theorem \ref{main:thm} under \eqref{pert:smooth}]
\noindent\textbf{Step 1. Interpolation of the comparison radii.}
\label{subsec:pert:campanato}
Fix $Q_{2\rho}(x_0,t_0)\Subset \Omega_T$, and set
\begin{equation}\label{pert:scales}
\begin{split}
    \sigma&\coloneqq\min\left\{\alpha,1-\frac{N+2}{d}\right\},
    \\
    R_*&\coloneqq
    \left(\frac{c_*}{1+\Lambda_{\rm h}
    +\|f\|_{L^d(Q_{2\rho}(x_0,t_0))}}\right)^{1/\sigma},
\end{split}
\end{equation}
where $c_*=c_*(\data)\in(0,1)$ is chosen so small that
$2c_*\leq\min\{c_0,1\}$, with $c_0$ as in Lemma
\ref{lem:pert:comparison}.
We also let
\[
\rho\leq R_1<R_2\leq2\rho,
\]
and set
\begin{equation}\label{pert:interpolation-geometry}
    \lambda\coloneqq1+\|Du\|_{L^\infty(Q_{R_2}(x_0,t_0))},
    \qquad
    R_0\coloneqq\frac18\min\{R_2-R_1,R_*\}
    \min\{1,\sqrt{a(\lambda)}\}.
\end{equation}
We  observe that, for every $(x,t)\in Q_{R_1}(x_0,t_0)$, we  have
\begin{equation}\label{pert:geometric-inclusion}
    Q_{2R_0}^{(\lambda)}(x,t)\subset Q_{R_2}(x_0,t_0).
\end{equation}
Indeed, $2R_0\leq(R_2-R_1)/4$ and
$4R_0^2/a(\lambda)\leq(R_2-R_1)^2/16
\leq R_2^2-R_1^2$.

Notice also that, as $\lambda\geq1$, we have $b(\lambda)\geq b(1)=1$, and \eqref{ciu:a}, $-1<i_a\leq 0$ imply
\begin{equation}\label{pert:uniform-ratios}
    \frac{1+b(\lambda)}{b(\lambda)}\leq2,
    \qquad
    \frac{a(\lambda)^{1/d}}{b(\lambda)}
    =\lambda^{-1}a(\lambda)^{-(1-1/d)}
    \leq\lambda^{-1-i_a(1-1/d)}\leq1.
\end{equation}
Consequently, for all  $Q_R^{(\lambda)}(x,t)\subset Q_{2\rho}(x_0,t_0)$ and
$R\leq R_*\leq1$, the choice \eqref{pert:scales} into \eqref{pert:size} yields
\begin{equation}\label{pert:smallness}
    \mathcal P_\lambda(R)
    \leq2\big(\Lambda_{\rm h}+\|f\|_{L^d(Q_{2\rho}(x_0,t_0))}\big)R^\sigma
    \leq2c_*\left(\frac{R}{R_*}\right)^\sigma.
\end{equation}
In particular, this estimate is uniform in $\lambda\geq1$.
By \eqref{pert:smallness}, the assumptions of Lemma
\ref{lem:pert:comparison} are  thus satisfied for every $R\leq R_*$.

So fix $(x,t)\in Q_{R_1}(x_0,t_0)$ and $0<r\leq R/8$, with $R\leq R_0$ to be fixed.
Let $v$ be the solution to problem \eqref{pert:frozen} on $Q_R^{(\lambda)}(x,t)$.
By the properties of the integral averages, Lemma
\ref{lem:pert:comparison}, and \eqref{pert:smallness}, we have

\begin{equation}\label{pert:two-scales}
\begin{split}
    &\miint_{Q_r^{(\lambda)}(x,t)}
    |Du-(Du)_{Q_r^{(\lambda)}(x,t)}|^2\,dyds\\
    &\quad\leq
    2\,\miint_{Q_r^{(\lambda)}(x,t)}|Du-Dv|^2\,dyds
    +2\,\miint_{Q_r^{(\lambda)}(x,t)}
    |Dv-(Dv)_{Q_r^{(\lambda)}(x,t)}|^2\,dyds\\
    &\quad\leq C\lambda^2\left[
    \left(\frac Rr\right)^{N+2}
    \left(\frac{R}{R_*}\right)^{2\sigma/s_B}+\left(\frac rR\right)^{2\beta}
    \right].
\end{split}
\end{equation}

For $0<r<R_0/8$, we now choose
\begin{equation}\label{pert:radii}
    R\coloneqq R_0\left(\frac{8r}{R_0}\right)^\kappa,
    \qquad
    \kappa\coloneqq
    \frac{N+2+2\beta}{N+2+2\beta+2\sigma/s_B}\in(0,1),
\end{equation}
 and 
\begin{equation}\label{pert:exponent}
    \beta_*\coloneqq
    \frac{2\beta\sigma}{s_B(N+2+2\beta)+2\sigma}\in(0,1).
\end{equation}
Then $8r<R<R_0$. Since $R_0\leq R_*$, the two terms in
\eqref{pert:two-scales} satisfy
\[
\begin{split}
    \left(\frac rR\right)^{2\beta}
    &\leq C\left(\frac r{R_0}\right)^{2\beta(1-\kappa)},\\
    \left(\frac Rr\right)^{N+2}
    \left(\frac R{R_*}\right)^{2\sigma/s_B}
    &\leq C\left(\frac r{R_0}\right)^{
    \kappa(N+2+2\sigma/s_B)-(N+2)}.
\end{split}
\]
The choice of $\kappa$ and $\beta_*$ gives
\[
    2\beta(1-\kappa)
    =\kappa(N+2+2\sigma/s_B)-(N+2)
    =2\beta_*.
\]
We conclude that
\begin{equation}\label{pert:campanato}
    \miint_{Q_r^{(\lambda)}(x,t)}
    |Du-(Du)_{Q_r^{(\lambda)}(x,t)}|^2\,dyds
    \leq C\lambda^2\left(\frac r{R_0}\right)^{2\beta_*},
    \qquad0<r\leq R_0.
\end{equation}
Notice that, for $R_0/8\leq r\leq R_0$, this follows from the trivial bound
$|Du|\leq\lambda$, after  possibly increasing $C$. 

Then, by
\eqref{pert:campanato}, for  $j=0,1,\dots,$ we have
\[
\begin{split}
    &\left|(Du)_{Q_{2^{-j-1}r}^{(\lambda)}(x,t)}
    -(Du)_{Q_{2^{-j}r}^{(\lambda)}(x,t)}\right|\\
    &\quad\leq C\left(
    \miint_{Q_{2^{-j}r}^{(\lambda)}(x,t)}
    |Du-(Du)_{Q_{2^{-j}r}^{(\lambda)}(x,t)}|^2\,dyds
    \right)^{1/2}
    \leq C\lambda\,2^{-j\beta_*}(r/R_0)^{\beta_*}.
\end{split}
\]
Summing over $j\geq0$, and using the continuity of $Du$ at $(x,t)$ under
\eqref{pert:smooth}, gives
\begin{equation}\label{pert:point-to-mean}
    |Du(x,t)-(Du)_{Q_r^{(\lambda)}(x,t)}|
    \leq C\lambda\left(\frac r{R_0}\right)^{\beta_*},
    \qquad0<r\leq R_0.
\end{equation}
We remark that all constants in \eqref{pert:campanato} and
\eqref{pert:point-to-mean} depend only on $\data,\alpha,d$.

\medskip
\noindent\textbf{Step 2. The quantitative gradient bound.}
\label{subsec:pert:gradient}
    For each $(x,t)\in Q_{R_1}(x_0,t_0)$, integration by parts in
    space gives
    \begin{equation}\label{pert:mean-gradient}
        |(Du)_{Q_r^{(\lambda)}(x,t)}|
        \leq\frac{1}{|Q_r^{(\lambda)}|}
        \int_{t-r^2/a(\lambda)}^t
        \left|\int_{\partial B_r(x)}\big(u-\inf_{Q_{2\rho}(x_0,t_0)}u\big)\,\nu\,d\mathcal H^{N-1}
        \right|\,ds\leq C\frac{\omega}{r},
    \end{equation}
    where $\nu$ is the outward normal to $\partial B_r(x)$, and $\omega=\operatorname{osc}_{Q_{2\rho}(x_0,t_0)} u$.
    Thus \eqref{pert:point-to-mean} implies
    \[
        |Du(x,t)|\leq C\lambda(r/R_0)^{\beta_*}+C\omega/r.
    \]
    Choose $r=\theta R_0$, with
    $\theta=\theta(\data,\alpha,d)\in(0,1)$ sufficiently small
    that $C\theta^{\beta_*}\leq1/8$.
    Taking the supremum over $(x,t)\in Q_{R_1}(x_0,t_0)$, and recalling the definition of $R_0$ in \eqref{pert:interpolation-geometry}, we get
    \begin{equation}\label{pert:pre-absorption}
    \begin{split}
        \|Du\|_{L^\infty(Q_{R_1}(x_0,t_0))}
        \leq\frac{\lambda}{8}
        +C\frac{\omega}{\min\{R_2-R_1,R_*\}}
        \left(1+\frac{1}{\sqrt{a(\lambda)}}\right).
    \end{split}
    \end{equation}
We now claim that
    \begin{equation}\label{pert:interpolation}
    \begin{split}
        \|Du\|_{L^\infty(Q_{R_1}(x_0,t_0))}
        \leq{}&\frac14\|Du\|_{L^\infty(Q_{R_2}(x_0,t_0))}\\
        &+C\left[1+
        \frac{\omega}{\min\{R_2-R_1,R_*\}}
        +B^{-1}\left(
        \frac{\omega^2}{\min\{R_2-R_1,R_*\}^2}\right)\right].
    \end{split}
    \end{equation}
    If $\omega=0$, the assertion is immediate.
    If
    $\lambda<B^{-1}(\omega^2/\min\{R_2-R_1,R_*\}^2)$,
    it follows directly from $|Du|\leq\lambda$ by  \eqref{pert:interpolation-geometry}.
   In the complementary case, by \eqref{B2t} and
\eqref{Bcomeb}, we have
\[
\begin{split}
    \frac{\omega^2}{\min\{R_2-R_1,R_*\}^2}
    &=B\left(B^{-1}\left(
    \frac{\omega^2}{\min\{R_2-R_1,R_*\}^2}
    \right)\right)\\
    &\leq B(\lambda)\left[
    \frac{B^{-1}\left(
    \omega^2/\min\{R_2-R_1,R_*\}^2
    \right)}{\lambda}
    \right]^{i_B}\\
    &\leq C\lambda^2a(\lambda)\left[
    \frac{B^{-1}\left(
    \omega^2/\min\{R_2-R_1,R_*\}^2
    \right)}{\lambda}
    \right]^{i_B}.
\end{split}
\]
Taking square roots gives
\[
    \frac{\omega}
    {\min\{R_2-R_1,R_*\}\sqrt{a(\lambda)}}
    \leq C\lambda\left[
    \frac{B^{-1}\left(
    \omega^2/\min\{R_2-R_1,R_*\}^2
    \right)}{\lambda}
    \right]^{i_B/2},
\]
    and since $1<i_B\leq2$, Young's inequality with exponents $(\frac{2}{i_B},\frac{2}{2-i_B})$ implies
    \[C\frac{\omega}
        {\min\{R_2-R_1,R_*\}\sqrt{a(\lambda)}}\leq\frac\lambda8
        +C B^{-1}\left(
        \frac{\omega^2}{\min\{R_2-R_1,R_*\}^2}\right).
    \]
    Substituting into \eqref{pert:pre-absorption} and recalling
    $\lambda=1+\|Du\|_{L^\infty(Q_{R_2}(x_0,t_0))}$ proves
    \eqref{pert:interpolation}.

    Next, we observe that
    \[
        \min\{R_2-R_1,R_*\}
        \geq\frac{R_2-R_1}{\rho}\min\{\rho,R_*\},
        \qquad\rho\leq R_1<R_2\leq2\rho,
    \]
and thus by \eqref{B2t}, \eqref{pert:interpolation} becomes
\[
\begin{split}
    \|Du\|_{L^\infty(Q_{R_1}(x_0,t_0))}
    \leq{}&\frac14\|Du\|_{L^\infty(Q_{R_2}(x_0,t_0))}\\
    &+C\left(\frac{\rho}{R_2-R_1}\right)^{2/i_B}
    \left[1+\frac{\omega}{\min\{\rho,R_*\}}
    +B^{-1}\left(
    \frac{\omega^2}{\min\{\rho,R_*\}^2}\right)\right].
\end{split}
\]
Applying Lemma \ref{lem:iteration} with
$\phi(r)=\|Du\|_{L^\infty(Q_r(x_0,t_0))}$,
$t=R_1$, $s=R_2$, and $R=2\rho$, we finally obtain \eqref{pert:gradient}.

\medskip
\noindent\textbf{Step 3. H\"older continuity of the solution and the gradient.}
From the results of the previous step, we have
\[
    \|Du\|_{L^\infty(Q_\rho(x_0,t_0))}\leq\lambda.
\]
with $\lambda$ as in \eqref{pert:lambda-final}.
Then we fix
\[
    R_0\coloneqq\frac1{32}\min\{\rho,R_*\}
    \min\{1,\sqrt{a(\lambda)}\}.
\]
For every $(x,t)\in Q_{\rho/2}(x_0,t_0)$, the cylinders
$Q_{2R_0}^{(\lambda)}(x,t)$ are contained in
$Q_\rho(x_0,t_0)$.
Consequently, \eqref{pert:campanato} and
\eqref{pert:point-to-mean} hold for all these centers
with the same $R_0$ and $\lambda$.

Let $(x_1,t_1),(x_2,t_2)\in Q_{\rho/2}(x_0,t_0)$,
with $t_1\leq t_2$, and set
\[
    r=|x_1-x_2|+\sqrt{a(\lambda)(t_2-t_1)}.
\]
If $0<r\leq R_0/4$, then
\[
    Q_r^{(\lambda)}(x_1,t_1)\cup
    Q_r^{(\lambda)}(x_2,t_2)
    \subset Q_{2r}^{(\lambda)}(x_1,t_2).
\]
Since $(x_1,t_2)\in Q_{\rho/2}(x_0,t_0)$,
by \eqref{pert:campanato}
we obtain
\[
\begin{split}
    &\left|(Du)_{Q_r^{(\lambda)}(x_1,t_1)}
    -(Du)_{Q_{2r}^{(\lambda)}(x_1,t_2)}\right|+
    \left|(Du)_{Q_r^{(\lambda)}(x_2,t_2)}
    -(Du)_{Q_{2r}^{(\lambda)}(x_1,t_2)}\right|\\
    &\quad\leq C\left(
    \miint_{Q_{2r}^{(\lambda)}(x_1,t_2)}
    |Du-(Du)_{Q_{2r}^{(\lambda)}(x_1,t_2)}|^2\,dyds
    \right)^{1/2}\leq C\lambda\left(\frac r{R_0}\right)^{\beta_*}.
\end{split}
\]
Combining this estimate with \eqref{pert:point-to-mean},
applied at $(x_1,t_1)$ and $(x_2,t_2)$, gives
\[
    |Du(x_1,t_1)-Du(x_2,t_2)|
    \leq C\lambda\left(\frac r{R_0}\right)^{\beta_*}.
\]
If $r>R_0/4$, the same estimate follows from
$|Du|\leq\lambda$.
Recalling the definition of $R_0$, this proves \eqref{pert:holder}.

It remains to prove \eqref{pert:oscillation}.
For $(x,t)\in Q_{\rho/2}(x_0,t_0)$ and $0<r\leq\rho/2$,
we have $Q_r(x,t)\subset Q_\rho(x_0,t_0)$.
Applying \eqref{pert:gluing-general} with $R=r$ and $S=r^2$,
and using \eqref{pert:gradient}, we obtain
\[
\begin{split}
    \operatorname*{osc}_{Q_r(x,t)}u
    \leq C\bigg[
    &r\lambda+r b(\lambda)
    +\Lambda_{\rm h}r^{1+\alpha}(1+b(\lambda))\\
    &+r^{2-(N+2)/d}
    \|f\|_{L^d(Q_{2\rho}(x_0,t_0))}\bigg].
\end{split}
\]
Since $r\leq1$, $d>N+2$, and $b(\lambda)\geq b(1)=1$,
this gives \eqref{pert:oscillation}.
Together with the continuity of $Du$, this also proves
$u\in C^{1,1/2}_{\rm loc}(\Omega_T)$, and completes the proof.

\end{proof}

\section{Approximation procedure}\label{sec:approximation}
\label{subsec:pert:approximation}

In this section, we regularize Equation \eqref{eq:1}
and complete the proofs of Theorems
\ref{thm:grbddhom}, \ref{thm:gradhold},
\ref{main:thm}, and Lemma \ref{lem:osc:v},
by removing assumptions \eqref{smooth} and \eqref{pert:smooth}.

To this end, we fix two open cylinders
$Q\Subset\widetilde Q\Subset\Omega_T$, with $Q$
containing the closures of the cylinders
appearing in the statements of the
aforementioned results. We write
$Q=\B\times\mathcal T$, where $\B\subset\R^N$
is a ball and $\mathcal T=(t_1,t_2)\subset(0,T)$.

By Theorem \ref{thm:bddness}, or directly
by the boundedness assumption in the
corresponding case, we have
\[
    u\in L^\infty(\widetilde Q).
\]
Since subtracting the vector field
$t\mapsto\A(x_*,t,0)$, for a fixed $x_*$
in the spatial base of $\widetilde Q$, does not
change Equation \eqref{eq:1}, we may and
shall assume that $\A(x_*,t,0)=0$.
By \eqref{ovu:Linf}, the quantity
$|\A(x,t,0)|$ is then bounded on $\widetilde Q$.

Now let $\widetilde\rho\in C_c^\infty(\R^{N+1};[0,\infty))$ be a
nonnegative mollifier, supported in $B_1\times(-1,1)$, such that
$\int_{\R^{N+1}}\widetilde\rho\,dxdt=1$. For $\delta>0$, set
\[
    \widetilde\rho_\delta(x,t)\coloneqq
    \delta^{-N-2}\widetilde\rho\left(\frac{x}{\delta},\frac{t}{\delta^2}\right).
\]
For $\delta>0$ sufficiently small, we define the space-time convolutions
\[
    u_\delta^{\rm bd}\coloneqq\widetilde\rho_\delta*u,
    \qquad f_\delta\coloneqq\widetilde\rho_\delta*f
    \qquad\text{in a neighborhood of $\overline Q$.}
\]
In particular, $u_\delta^{\rm bd}\in C^\infty(\overline Q)$ and,
by Jensen's inequality and the properties of convolution,
\begin{equation}\label{boud}
\begin{split}
    \iint_Q B(|Du_\delta^{\rm bd}|)\,dxdt
    &\leq \iint_{\widetilde Q}B(|Du|)\,dxdt,\\
    \|u_\delta^{\rm bd}\|_{L^\infty(Q)}
    &\leq \|u\|_{L^\infty(\widetilde Q)},\qquad
    \|f_\delta\|_{L^d(Q)}\leq \|f\|_{L^d(\widetilde Q)}.
\end{split}
\end{equation}
Furthermore, writing
\[
    G_\delta(x,t)\coloneqq\widetilde\rho_\delta*\Big(\A\big(x,t,Du(x,t)\big)\Big),
\]
from \eqref{eq:1} we deduce that $u_\delta^{\rm bd}$ is solution to
\begin{equation}\label{homreg:bdtime}
    \partial_t u_\delta^{\rm bd}=\mathrm{div}\,G_\delta+f_\delta
    \qquad\text{in }Q.
\end{equation}
By the growth of $\A$ in \eqref{deb:ass:A},
Equations \eqref{tB:usef}, \eqref{wBbt} and Jensen's inequality, we infer
\begin{equation}\label{homreg:Gdelta}
    \iint_Q\widetilde B(|G_\delta|)\,dxdt
    \leq C\left[|\widetilde Q|+
    \iint_{\widetilde Q}B(|Du|)\,dxdt\right],
\end{equation}
where $C$ is independent of $\delta$.

Then let $\A_\e(x,t,\xi)$ be the regularization of $\A$
in $\widetilde Q$ provided by Lemma \ref{lem:approx:A}.
We also let $u_{\e,\delta}$ be solution to the parabolic problem
\begin{equation}\label{eq:homreg}
    \begin{cases}
        \partial_t u_{\e,\delta}-\mathrm{div}\big(
        \A_\e(x,t,Du_{\e,\delta})\big)=f_\delta
        \quad &\text{in $Q$}\\
        u_{\e,\delta}=u_\delta^{\rm bd}
        \quad &\text{on $\partial_{\rm par}Q$.}
    \end{cases}
\end{equation}

We have the following

\begin{proposition}\label{prop:liscia}
    For every $\e>0$ and all sufficiently small $\delta>0$,
    problem \eqref{eq:homreg} admits a unique weak solution $u_{\e,\delta}$.
    Moreover,
\begin{equation}\label{eq:homreg:energy}
\begin{split}
    &\limsup_{\e\to0}
    \iint_Q B_\e(|Du_{\e,\delta}|)\,dxdt
    \leq C\left[|\widetilde Q|+
    \iint_{\widetilde Q}B(|Du|)\,dxdt\right],\\
    &\|u_{\e,\delta}\|_{L^\infty(Q)}
    \leq\|u\|_{L^\infty(\widetilde Q)}
    +C\left(1+\|f\|_{L^d(\widetilde Q)}^{
    \frac{N+i_B}{i_B(N+1)-N}}\right),
\end{split}
\end{equation}
for some $C$ independent of $\e,\delta$.
Moreover, $u_{\e,\delta}\in C^\infty(Q)$,
and hence it solves \eqref{eq:homreg}
in the classical sense in $Q$.
In the homogeneous case, when $f=0$ and
$\A=\A_0(t,\xi)$, the first bound holds without
$|\widetilde Q|$, while the second reduces to
$\|u_{\e,\delta}\|_{L^\infty(Q)}
\leq\|u\|_{L^\infty(\widetilde Q)}$.
\end{proposition}

\begin{proof}

For every fixed $\e>0$,
the uniform ellipticity and linear growth
of $\A_\e$, ensured by
\eqref{approx:A:unif}, allow us to apply
the standard existence theory for
monotone parabolic operators; see, e.g.,
\cite[Chapter II, Section 1, Theorem 1.2
and Remark 1.10]{Lions1969}, \cite{Minty1962, Showalter1997} and \cite[Chapter 8]{Roubicek2013}.
Consequently, problem \eqref{eq:homreg}
admits a unique weak solution
\[
\begin{split}
    u_{\e,\delta}\in
    C(\overline{\mathcal T};L^2(\B))
    \cap L^2(\mathcal T;W^{1,2}(\B)),\quad \text{such that}\\
    u_{\e,\delta}-u_\delta^{\rm bd}
    \in L^2(\mathcal T;W^{1,2}_0(\B)),\quad\text{and}\quad
    u_{\e,\delta}(\cdot,t_1)=u_\delta^{\rm bd}(\cdot,t_1)
    \ \text{in }L^2(\B).
\end{split}
\]

Let us now prove its smoothness; we adapt the proof of
\cite[Proposition 4.2]{Antonini2027}
to the parabolic setting.
Throughout the proof, the constants
are allowed to depend on $\e$ and $\delta$.

We first show that
\begin{equation}\label{homreg:W22}
    D^2u_{\e,\delta}\in L^2_{\rm loc}(Q).
\end{equation}
For $k=1,\dots,N$ and $h\neq0$, let
\[
    \Delta_h^k u_{\e,\delta}(x,t)
    \coloneqq
    \frac{u_{\e,\delta}(x+he_k,t)
    -u_{\e,\delta}(x,t)}{h}.
\]
Taking spatial difference quotients
in \eqref{eq:homreg} gives
\begin{equation}\label{homreg:difference}
    \partial_t(\Delta_h^k u_{\e,\delta})
    -\mathrm{div}\left(
    \mathbb A_h
    D(\Delta_h^k u_{\e,\delta})
    \right)=\mathrm{div}\,H_h+\Delta_h^k f_\delta,
\end{equation}
where
\[
\begin{split}
    \mathbb A_h(x,t)
    &\coloneqq\int_0^1
    \nabla_\xi\A_\e\big(
    x+he_k,t,Du_{\e,\delta}(x,t)+s\big(
    Du_{\e,\delta}(x+he_k,t)
    -Du_{\e,\delta}(x,t)
    \big)\big)\,ds,\\
    H_h(x,t)&\coloneqq
    \frac{\A_\e(x+he_k,t,Du_{\e,\delta}(x,t))
    -\A_\e(x,t,Du_{\e,\delta}(x,t))}{h}.
\end{split}
\]
By \eqref{approx:A:unif}, we have
\begin{equation}\label{homreg:matell}
    c\e|\zeta|^2
    \leq\mathbb A_h(x,t)\zeta\cdot\zeta,
    \qquad
    |\mathbb A_h(x,t)|
    \leq C\e^{-1}\quad\text{for every $\zeta\in\R^N$.}
\end{equation}
For every fixed $\e>0$ and every $Q'\Subset Q$,
Lemma \ref{lem:approx:A}  gives
$\sup_{Q\times\R^N}|\nabla_x\A_\e|<\infty$.
Thus, by the properties of difference quotients,  $H_h$ and $\Delta_h^kf_\delta$ are
locally bounded independently of $h$.

Let $\eta\in C_c^\infty(Q;[0,1])$.
For $|h|$ sufficiently small, we test
\eqref{homreg:difference} with
$\eta^2\Delta_h^k u_{\e,\delta}$.
Integrating over the portion of $Q$
preceding a time $\tau$ and using
Steklov averages to justify the
time integration by parts, we obtain
\[
\begin{split}
    &\frac12\int_{\B\times \{\tau\}}
    \eta^2
    |\Delta_h^k u_{\e,\delta}|^2\,dx+
    \iint_{Q\cap\{t<\tau\}}
    \eta^2\mathbb A_h
    D(\Delta_h^k u_{\e,\delta})
    \cdot D(\Delta_h^k u_{\e,\delta})
    \,dxdt\\
    &=
    \iint_{Q\cap\{t<\tau\}}
    \eta\,\partial_t\eta\,
    |\Delta_h^k u_{\e,\delta}|^2
    \,dxdt-
    2\iint_{Q\cap\{t<\tau\}}
    \eta\,\Delta_h^k u_{\e,\delta}\,
    \mathbb A_hD(\Delta_h^k u_{\e,\delta})
    \cdot D\eta\,dxdt\\
    &\quad-\iint_{Q\cap\{t<\tau\}}
    H_h\cdot D(\eta^2\Delta_h^k u_{\e,\delta})\,dxdt
    +\iint_{Q\cap\{t<\tau\}}
    \eta^2\Delta_h^k f_\delta\,
    \Delta_h^k u_{\e,\delta}\,dxdt.
\end{split}
\]
By \eqref{homreg:matell} and Young's
inequality, we have
\[
\begin{split}
    &2\eta|\Delta_h^k u_{\e,\delta}|
    |\mathbb A_hD(\Delta_h^k u_{\e,\delta})|
    |D\eta|+
    |H_h|\,\big|D(\eta^2\Delta_h^k u_{\e,\delta})\big|\\
    &\qquad\leq
    \frac{c\e}{2}\eta^2
    |D(\Delta_h^k u_{\e,\delta})|^2
    +C_{\e,\eta}\left(
    |\Delta_h^k u_{\e,\delta}|^2+|H_h|^2\right).
\end{split}
\]
Consequently,
\begin{equation}\label{homreg:diffest}
\begin{split}
    &\operatorname*{ess\,sup}_{\tau}
    \int_{\B\times \{\tau\}}
    \eta^2
    |\Delta_h^k u_{\e,\delta}|^2\,dx+
    c\e\iint_Q
    \eta^2|D(\Delta_h^k u_{\e,\delta})|^2
    \,dxdt\\
    &\leq
    C_{\e,\delta,\eta}\left(1+
    \iint_{\operatorname{spt}\eta}
    |\Delta_h^k u_{\e,\delta}|^2\,dxdt\right).
\end{split}
\end{equation}
Since $Du_{\e,\delta}\in L^2(Q)$,
the right-hand side is bounded
independently of $h$ thanks to the standard
estimates for difference quotients \cite[Lemmas 7.23--7.24]{GilbargTrudinger2001}.
Letting $h\to0$, we deduce
\eqref{homreg:W22}.

Moreover, since
$\nabla_\xi\A_\e$ and $\nabla_x\A_\e$
are locally bounded,
the chain rule gives
\[
    \partial_t u_{\e,\delta}
    =\mathrm{div}\big(
    \A_\e(x,t,Du_{\e,\delta})\big)+f_\delta
    \in L^2_{\rm loc}(Q).
\]
We may therefore differentiate
\eqref{eq:homreg} with respect to $x_k$
in the sense of distributions, obtaining
\[
\begin{split}
    &\partial_t(D_ku_{\e,\delta})
    -\mathrm{div}\left(
    \nabla_\xi\A_\e(x,t,Du_{\e,\delta})
    D(D_ku_{\e,\delta})
    \right)\\
    &\qquad=\mathrm{div}\big(
    (\partial_{x_k}\A_\e)(x,t,Du_{\e,\delta})\big)
    +\partial_{x_k}f_\delta.
\end{split}
\]
In particular, $D_ku_{\e,\delta}$
is a weak solution of a linear
uniformly parabolic equation in
divergence form, with locally bounded
right-hand side and coefficient
matrix satisfying
\[
    c\e|\zeta|^2
    \leq
    \nabla_\xi\A_\e(x,t,Du_{\e,\delta})
    \zeta\cdot\zeta,
    \qquad
    |\nabla_\xi\A_\e(x,t,Du_{\e,\delta})|
    \leq C\e^{-1}.
\]
By the parabolic
De Giorgi--Nash--Moser theorem
\cite{Moser1964}, \cite[III. Theorem 1.1]{DiBenedetto1993},
\cite[Chapter 12]{DiBenedettoGianazza2023},
we conclude that
\[
    Du_{\e,\delta}
    \in C^{\beta,\beta/2}_{\rm loc}(Q)
\]
for some $\beta\in(0,1)$,
possibly depending on $\e,\delta$.
Since $\A_\e\in C^\infty$,
it follows that
\[
    \nabla_\xi\A_\e
    (x,t,Du_{\e,\delta}),\quad
    \nabla_x\A_\e(x,t,Du_{\e,\delta})
    \in C^{\beta,\beta/2}_{\rm loc}(Q).
\]
Applying the interior Schauder
estimates for linear uniformly
parabolic equations in divergence
form (see \cite[Chapter IV, Theorem 4.8]{Lieberman1996})
to the equation satisfied by
$D_ku_{\e,\delta}$, we obtain
\[
    Du_{\e,\delta}
    \in C^{1+\beta,(1+\beta)/2}_{\rm loc}(Q).
\]
Consequently, $D^2u_{\e,\delta}
\in C^{\beta,\beta/2}_{\rm loc}(Q)$,
and again by the equation and
the smoothness of $\A_\e$ and $f_\delta$,
we deduce
\[
    \partial_t u_{\e,\delta}
    \in C^{\beta,\beta/2}_{\rm loc}(Q).
\]
Hence $u_{\e,\delta}\in C^{2+\beta,1+\beta/2}_{\rm loc}(Q)$.
A standard bootstrap argument, obtained by iterating
the preceding argument via Schauder estimates, then yields
$u_{\e,\delta}\in C^\infty(Q)$.

Next, we derive the energy estimate.
Testing \eqref{eq:homreg} with
$u_{\e,\delta}-u_\delta^{\rm bd}$, integrating over
$\B\times(t_1,\tau)$, where $t_1<\tau<t_2$, and using
\eqref{homreg:bdtime}, we obtain
\begin{equation}\label{eq:homreg:energy1}
\begin{split}
    &\frac12\int_{\B}
    |u_{\e,\delta}-u_\delta^{\rm bd}|^2(x,\tau)\,dx
    +\int_{t_1}^{\tau}\int_{\B}
    \A_\e(x,t,Du_{\e,\delta})\cdot Du_{\e,\delta}\,dxdt\\
    &\qquad=\int_{t_1}^{\tau}\int_{\B}
    \A_\e(x,t,Du_{\e,\delta})\cdot Du_\delta^{\rm bd}\,dxdt
    +\int_{t_1}^{\tau}\int_{\B}
    G_\delta\cdot D(u_{\e,\delta}-u_\delta^{\rm bd})\,dxdt.
\end{split}
\end{equation}
Observe that we have used the initial condition
$u_{\e,\delta}(\cdot,t_1)=u_\delta^{\rm bd}(\cdot,t_1)$, the fact that $u_{\e,\delta}-u_\delta^{\rm bd}$
has zero spatial trace, and that the terms containing
$f_\delta$ cancel.
By \eqref{approx:A:co} and the uniform bound on
$|\A_\e(x,t,0)|$, we have
\begin{equation}\label{eq:homreg:energy2}
\begin{split}
    &\frac12\int_{\B}
    |u_{\e,\delta}-u_\delta^{\rm bd}|^2(x,\tau)\,dx
    +c\int_{t_1}^{\tau}\int_{\B}
    B_\e(|Du_{\e,\delta}|)\,dxdt\\
    &\qquad\leq C\int_{t_1}^{\tau}\int_{\B}
    b_\e(|Du_{\e,\delta}|)|Du_\delta^{\rm bd}|\,dxdt
    +\int_{t_1}^{\tau}\int_{\B}
    |G_\delta|\big(|Du_{\e,\delta}|+|Du_\delta^{\rm bd}|\big)\,dxdt
    +C|\widetilde Q|.
\end{split}
\end{equation}
By Young's inequality \eqref{young1} for $B_\e$ and its
complementary Young function $\widetilde B_\e$, and recalling \eqref{iaesae}, we have
\[
\begin{split}
    &Cb_\e(|Du_{\e,\delta}|)|Du_\delta^{\rm bd}|
    +|G_\delta||Du_{\e,\delta}|
    \leq \frac{c}{2}B_\e(|Du_{\e,\delta}|)
    +C B_\e(|Du_\delta^{\rm bd}|)
    +C\widetilde B_\e(|G_\delta|),
\end{split}
\]
while, by Young's inequality \eqref{Young},
\[
    |G_\delta||Du_\delta^{\rm bd}|
    \leq B(|Du_\delta^{\rm bd}|)+\widetilde B(|G_\delta|).
\]
Inserting these bounds into \eqref{eq:homreg:energy2},
absorbing the corresponding term, and taking the
supremum over $\tau\in\mathcal T$, we obtain
\begin{equation}\label{eq:homreg:energy3}
\begin{split}
    &\operatorname*{sup}_{t\in\mathcal T}\int_{\B\times \{t\}}
    |u_{\e,\delta}-u_\delta^{\rm bd}|^2(x,t)\,dx
    +\iint_Q B_\e(|Du_{\e,\delta}|)\,dxdt\\
    &\qquad\leq C\left\{|\widetilde Q|+
    \iint_Q\Big[
    B_\e(|Du_\delta^{\rm bd}|)
    +\widetilde B_\e(|G_\delta|)
    +B(|Du_\delta^{\rm bd}|)
    +\widetilde B(|G_\delta|)\Big]dxdt\right\}.
\end{split}
\end{equation}
For each fixed $\delta>0$, the functions
$Du_\delta^{\rm bd}$ and $G_\delta$ are bounded on $Q$.
As $B_\e\to B$ locally uniformly by \eqref{BE:unif}, the complementary
Young functions $\widetilde B_\e$ also converge to
$\widetilde B$ locally uniformly.
Therefore, by \eqref{boud}, \eqref{homreg:Gdelta} and \eqref{eq:homreg:energy3} we get
\[
\begin{split}
    &\limsup_{\e\to0}\bigg\{
    \operatorname*{sup}_{t\in\mathcal T}
    \int_{\B\times \{t\}}|u_{\e,\delta}-u_\delta^{\rm bd}|^2(x,t)\,dx
    +\iint_Q B_\e(|Du_{\e,\delta}|)\,dxdt\bigg\}\\
    &\qquad\leq C\left[|\widetilde Q|+
    \iint_{\widetilde Q}B(|Du|)\,dxdt\right],
\end{split}
\]
which proves the first estimate in \eqref{eq:homreg:energy}.
For what concerns the second one, by Theorem
\ref{thm:global:bddness} and \eqref{boud}, we deduce
\[
    \|u_{\e,\delta}\|_{L^\infty(Q)}
    \leq\|u_\delta^{\rm bd}\|_{L^\infty(Q)}
    +C\left(1+\|f_\delta\|_{L^d(Q)}^{
    \frac{N+i_B}{i_B(N+1)-N}}\right),
\]
which yields \eqref{eq:homreg:energy}.
When $f=0$ and $\A=\A_0(t,\xi)$,
we may assume $\A_0(t,0)=0$ and hence
$\A_{0,\e}(t,0)=0$.
The terms involving $|\widetilde Q|$ then disappear
from the energy estimate in view of \eqref{deb:ass:A} with $\Lambda_{\rm h}=0$, and the parabolic maximum
principle yields
$\|u_{\e,\delta}\|_{L^\infty(Q)}
\leq\|u\|_{L^\infty(\widetilde Q)}$.
This completes the proof.
\end{proof}

In what follows, when dealing with the homogeneous case \eqref{eq:hom}, we take
$u=v$, $\A(x,t,\xi)=\A_0(t,\xi)$ and $f=0$,
and denote the preceding approximations by
$v_\delta^{\rm bd}$, $\A_{0,\e}$ and $v_{\e,\delta}$.
We remark that, by \eqref{approx:A:hold} with $\Lambda_{\rm h}=0$, $\A_{0,\e}$ is independent of $x$.

\begin{proof}[Completion of the proof of Theorems
\ref{thm:grbddhom} and \ref{thm:gradhold},
and Lemma \ref{lem:osc:v}: removal of \eqref{smooth}]
We make use of  \eqref{gr:bd1}-\eqref{gr:bd2},  \eqref{osc:v}, \eqref{osc:intr}-\eqref{osc:extr} applied to $v_{\e,\delta}$, with $a_\e, b_\e,B_\e$ in place of $a,b,B$, and get
\begin{equation}\label{iuy}
    \|Dv_{\e,\delta}\|_{L^\infty(Q_{R/2,S/4})}
 \le C\left\{\left(\frac S{R^2}E_{\e,\delta}\right)^{1/2}
 +F_\e^{-1}\left(\left(\frac{R^2}{S}\right)^{N/2}E_{\e,\delta}\right)\right\}\qquad \text{if $i_B>\frac{2N}{N+2}$,}
\end{equation}
where
\[
    F_\e(s)\coloneqq s^2a_\e(s)^{(N+2)/2}\quad\text{and}\quad E_{\e,\delta}\coloneqq\miint_{Q_{R,S}}B_\e(|Dv_{\e,\delta}|)\,dxdt,
\]
    we have
    \begin{equation}\label{iuy1}
        \begin{split}\|Dv_{\e,\delta}\|_{L^\infty(Q_{R/2,S/4})}\leq C\,\left\{\left(\frac{\operatorname{osc}_{Q_{R,S}}v_{\e,\delta}}{R}\right)+B_\e^{-1}\left(\frac{(\operatorname{osc}_{Q_{R,S}}v_{\e,\delta})^2}{S}\right) \right\},
\\
     \operatorname*{osc}_{Q_{R,S}}v_{\e,\delta}
        \leq C\left(
        R \,\|Dv_{\e,\delta}\|_{L^\infty(Q_{R,S})}+\frac{S}{R}b_\e(\|Dv_{\e,\delta}\|_{L^\infty(Q_{R,S})})
        \right), 
        \\
    \operatorname*{osc}_{Q_{r}} Dv_{\e,\delta}\leq C\,M_0\,\left(\frac{r}{R} \,a_\e^\star(M_0)\right)^\beta,\qquad M_0\coloneqq\|Dv_{\e,\delta}\|_{L^\infty(Q_R)},
        \end{split}
\end{equation}
and if  $Q^{(\lambda)}_R=Q_R^{(\l)}(x_0,t_0)=B_R(x_0)\times (t_0-R^2/a_\e(\l), t_0)$   is such that $\|Dv_{\e,\delta}\|_{L^\infty(Q_R^{(\l)})}\leq \lambda$, then
\begin{equation}\label{sue0}
      \operatorname*{osc}_{Q^{(\lambda)}_r} Dv_{\e,\delta}\leq C\,\lambda\,\left( \frac{r}{R}\right)^{\beta}\quad \text{for all }0<r\leq R
\end{equation}

    We remark that the constant $C=C(\data)$ in view of \eqref{approx:A:ell} and \eqref{iaesae}. In particular, by \eqref{eq:homreg:energy}, \eqref{iuy1}, \eqref{BE:unif} and a standard covering argument, we have that
\[
\|v_{\e,\delta}\|_{C^{\theta,\theta/2}(Q')}+\|D v_{\e,\delta}\|_{C^{\theta,\theta/2}(Q')}\leq C\big(\data, \|v\|_{L^\infty(\widetilde Q)}, Q'\big)\quad \text{for all $Q'\Subset Q$,}
\]
    for some $\theta=\theta(\data)\in (0,1)$.

Therefore, by the Arzel\'a-Ascoli theorem,
up to a diagonal subsequence as $\e,\delta\to0$,
and possibly decreasing $\theta$, we have
\begin{equation}\label{c:to:w}
    v_{\e,\delta}\to w
    \quad\text{and}\quad
    Dv_{\e,\delta}\to Dw
    \qquad\text{in }C^{\theta,\theta/2}_{\rm loc}(Q),
\end{equation}
for some function $w$.

We claim that
\begin{equation}\label{w:energy}
    w\in L^\infty(Q)
    \cap L^B(\mathcal T;W^{1,B}(\B)),
\end{equation}
and that $w=v$ on $\partial_{\rm par}Q$.
To this end, first observe that, by \eqref{usso} applied to $B_\e$ and \eqref{iaesae}, we have $t^{i_B}\leq C(B_\e(t)+1)$ for $t\geq 0$, and thus   by
\eqref{eq:homreg:energy} along the chosen subsequence
we have
\begin{equation}\label{homreg:weakbounds}
\begin{split}
    &\iint_Q |Dv_{\e,\delta}|^{i_B}\,dxdt
    \leq
    C\left(
    1+\iint_{\widetilde Q}B(|Dv|)\,dxdt
    \right).
\end{split}
\end{equation}
which, up to further non-relabeled subsequences,
\[
    v_{\e,\delta}\to w
    \quad\text{weakly in }
    L^{i_B}(\mathcal T;W^{1,i_B}(\B)).
\]
Moreover, for every $Q'\Subset Q$,
by \eqref{c:to:w}, \eqref{BE:unif}
and \eqref{eq:homreg:energy}, we obtain
\[
    \iint_{Q'}B(|Dw|)\,dxdt
    \leq C\iint_{\widetilde Q}B(|Dv|)\,dxdt,
    \qquad
    \|w\|_{L^\infty(Q')}
    \leq\|v\|_{L^\infty(\widetilde Q)}.
\]
Letting $Q'\nearrow Q$ gives
\eqref{w:energy}.

We now show that $w=v$ on
$\partial_{\rm par}Q$.
By the properties of convolution,
\begin{equation}\label{homreg:bdconv}
    v_\delta^{\rm bd}\to v
    \quad\text{in }
    L^{i_B}(\mathcal T;W^{1,i_B}(\B)).
\end{equation}
On the other hand, since $ v_{\e,\delta}-v_\delta^{\rm bd}
    \in L^{i_B}(\mathcal T;W^{1,i_B}_0(\B))$, 
and these differences are uniformly
bounded in the same space by
\eqref{homreg:weakbounds} and
\eqref{boud}, the weak closedness
of $L^{i_B}(\mathcal T;W^{1,i_B}_0(\B))$
implies
\[
    w-v\in
    L^{i_B}(\mathcal T;W^{1,i_B}_0(\B)).
\]
Thus $w=v$ on the lateral boundary
of $Q$ in the sense of traces.

To identify the initial condition,
we recall that weak solutions to
\eqref{eq:hom} admit a representative
which is continuous in time with
values in $L^2_{\rm loc}(\Omega)$.
Consequently, the space-time convolution
gives
\[
    v_\delta^{\rm bd}(\cdot,t_1)
    \to v(\cdot,t_1)
    \qquad\text{in }L^2(\B).
\]
For every
$\varphi\in C_c^\infty(\B\times[t_1,t_2))$,
the weak formulation of
\eqref{eq:homreg} gives
\begin{equation}\label{homreg:weakform}
\begin{split}
    &-\iint_Q v_{\e,\delta}\,
    \partial_t\varphi\,dxdt
    +\iint_Q
    \A_{0,\e}(t,Dv_{\e,\delta})
    \cdot D\varphi\,dxdt=
    \int_{\B}
    v_\delta^{\rm bd}(x,t_1)
    \varphi(x,t_1)\,dx.
\end{split}
\end{equation}

By \eqref{c:to:w} and Lemma \ref{lem:approx:A}, 
passing to the limit we conclude that
\[
    -\iint_Q w\,\partial_t\varphi\,dxdt
    +\iint_Q\A_0(t,Dw)\cdot D\varphi\,dxdt=
    \int_{\B}v(x,t_1)\varphi(x,t_1)\,dx.
\]
In particular,
\[
    \partial_t w-\operatorname{div}\big(\A_0(t,Dw)\big)=0
    \qquad\text{in }Q.
\]
To justify the time continuity, observe that
\[
    w-v\in
    L^\infty(\mathcal T;L^2(\B))
    \cap L^B(\mathcal T;W^{1,B}_0(\B))
\]
and by Remark \ref{rem:energy}
\[
    \partial_t(w-v)
    =
    \operatorname{div}\big(
    \A_0(t,Dw)-\A_0(t,Dv)\big),
    \qquad
    \A_0(t,Dw)-\A_0(t,Dv)
    \in L^{\widetilde B}(Q;\R^N).
\]
The standard time-regularization argument in the Orlicz
energy spaces (see, e.g., \cite{ElmahiMeskine2005, AhmedOubeidBenkiraneSidiElVally2014}) therefore gives
\[
    w-v\in C([t_1,t_2];L^2(\B)).
\]
Since $v\in C([t_1,t_2];L^2(\B))$, the function $w$
also admits a representative in this space.
Moreover, the  weak formulation identifies its initial trace $w(\cdot,t_1)=v(\cdot,t_1)$ in $L^2(\B)$.

We have thus proved that $w$ and $v$
are weak solutions to the same equation
in $Q$, with the same parabolic boundary
values. By uniqueness,
we deduce that $ w=v$ in $Q$; consequently, \eqref{c:to:w} becomes 
\begin{equation*}
    v_{\e,\delta}\to v
    \quad\text{and}\quad
    Dv_{\e,\delta}\to Dv
    \qquad\text{in }
    C^{\theta,\theta/2}_{\rm loc}(Q)
\end{equation*}
along the chosen subsequence. Thereby passing to the limit in  \eqref{iuy}-\eqref{sue0}, and using   \eqref{ae:unif},
\eqref{be:unif} and \eqref{BE:unif}, we 
recover \eqref{gr:bd1}--\eqref{gr:bd2},
\eqref{osc:v}, and
\eqref{osc:intr}--\eqref{osc:extr}
for the original solution $v$. This completes the proofs.

\end{proof}

\begin{proof}[Completion of the proof of Theorem
\ref{main:thm}: removal of \eqref{pert:smooth}]
We make use of \eqref{pert:gradient}, \eqref{pert:holder}
and \eqref{pert:oscillation} applied to $u_{\e,\delta}$,
with $a_\e,b_\e,B_\e,f_\delta$ in place of
$a,b,B,f$. Fix $Q_{2\rho}(x_0,t_0)\Subset Q$,
with $0<2\rho\leq1$, and recalling the explicit expression of \eqref{pert:scales}, we set
\[
\begin{split}
    \omega_{\e,\delta}&\coloneqq
    \operatorname*{osc}_{Q_{2\rho}(x_0,t_0)}u_{\e,\delta},\\
    R_{*,\e,\delta}&\coloneqq\left(
    \frac{c_*}{1+\Lambda_{\rm h}
    +\|f_\delta\|_{L^d(Q_{2\rho}(x_0,t_0))}}
    \right)^{1/\sigma},\\
    \lambda_{\e,\delta}&\coloneqq
    1+\frac{\omega_{\e,\delta}}{\min\{\rho,R_{*,\e,\delta}\}}
    +B_\e^{-1}\left(\frac{\omega_{\e,\delta}^2}
    {\min\{\rho,R_{*,\e,\delta}\}^2}\right).
\end{split}
\]
 We then have
\begin{equation}\label{pert:approx-gradient}
    \|Du_{\e,\delta}\|_{L^\infty(Q_\rho(x_0,t_0))}
    \leq C\lambda_{\e,\delta},
\end{equation}
and
\begin{equation}\label{pert:approx-holder}
\begin{split}
    |Du_{\e,\delta}(x,t)-Du_{\e,\delta}(y,s)|\leq C\lambda_{\e,\delta}\left[
    \frac{|x-y|+\sqrt{a_\e(\lambda_{\e,\delta})|t-s|}}
    {\min\{\rho,R_{*,\e,\delta}\}\min\{1,\sqrt{a_\e(\lambda_{\e,\delta})}\}}
    \right]^{\beta_*}
\end{split}
\end{equation}
for all $(x,t),(y,s)\in Q_{\rho/2}(x_0,t_0)$.
Moreover, by \eqref{pert:oscillation},
\begin{equation}\label{pert:approx-gluing}
\begin{split}
    \operatorname*{osc}_{Q_r(x,t)}u_{\e,\delta}
    \leq C(\datah)\big[
    &\lambda_{\e,\delta}
    +b_\e(\lambda_{\e,\delta})\\
    &+\|f_\delta\|_{L^d(Q_{2\rho}(x_0,t_0))}
    \big]r
\end{split}
\end{equation}
whenever $(x,t)\in Q_{\rho/2}(x_0,t_0)$
and $0<r\leq\rho/2$.

We remark that $C=C(\data,\alpha,d)$ and
$\beta_*=\beta_*(\data,\alpha,d)$, are independent 
of $\e,\delta$, by \eqref{approx:A:ell},
\eqref{approx:A:hold} and \eqref{iaesae}.
In particular, by \eqref{boud} and
\eqref{eq:homreg:energy} and since $f_\delta\to f$ in $L^d$,
\[
    \omega_{\e,\delta}\leq
    2\|u_{\e,\delta}\|_{L^\infty(Q)}\leq C,
    \qquad
    R_{*,\e,\delta}\geq
    \left(\frac{c_*}{1+\Lambda_{\rm h}
    +2\|f\|_{L^d(\widetilde Q)}}\right)^{1/\sigma}>0.
\]
Since $B_\e(1)\to B(1)\geq c(i_a,s_a)>0$, by \eqref{iaesae}
and \eqref{BE:unif} we also have
$1\leq \lambda_{\e,\delta}\leq C$.
Furthermore, $a_\e(\lambda_{\e,\delta})$
is uniformly bounded above and away from zero
by \eqref{iaesae} and \eqref{ae:unif}.
Thus \eqref{pert:approx-gradient} and
\eqref{pert:approx-holder} give uniform local
bounds and H\"older estimates for $u_{\e,\delta},Du_{\e,\delta}$.
Therefore, by a standard covering argument,
\[
    \|u_{\e,\delta}\|_{C^{\theta,\theta/2}(Q')}
    +\|Du_{\e,\delta}\|_{C^{\theta,\theta/2}(Q')}
    \leq C\big(\data,\alpha,d,\Lambda_{\rm h},
    \|u\|_{L^\infty(\widetilde Q)},
    \|f\|_{L^d(\widetilde Q)},Q'\big)
    \quad\text{for all $Q'\Subset Q$,}
\]
for some $\theta=\theta(\data,\alpha,d)\in(0,1)$. By the Arzel\'a--Ascoli theorem,
up to a diagonal subsequence as $\e,\delta\to0$,
and possibly decreasing $\theta$, we have
\begin{equation}\label{pert:compactness}
    u_{\e,\delta}\to w
    \quad\text{and}\quad
    Du_{\e,\delta}\to Dw
    \qquad\text{in }C^{\theta,\theta/2}_{\rm loc}(Q),
\end{equation}
for some function $w$.
Now, as in the approximation argument for homogeneous equation, we show that the limit function
\[
    w\in  L^\infty(\mathcal T;L^2(\B))
    \cap L^B(\mathcal T;W^{1,B}(\B)),
\]
and that $w=u$ on $\partial_{\rm par}Q$.
To this end, since $t^{i_B}\leq C(B_\e(t)+1)$  by \eqref{usso} applied to $B_\e$
and \eqref{iaesae}, by \eqref{eq:homreg:energy} we have
\[
    \iint_Q|Du_{\e,\delta}|^{i_B}\,dxdt
    \leq C\left[1+
    \iint_{\widetilde Q}B(|Du|)\,dxdt\right],
\]
which, up to a further non-relabeled subsequence,
gives
\[
    u_{\e,\delta}\to w
    \quad\text{weakly in }
    L^{i_B}(\mathcal T;W^{1,i_B}(\B)).
\]
Moreover, for every $Q'\Subset Q$,
by \eqref{pert:compactness}, \eqref{BE:unif}
and \eqref{eq:homreg:energy}, we obtain
\[
    \iint_{Q'}B(|Dw|)\,dxdt
    \leq C\left[|\widetilde Q|+
    \iint_{\widetilde Q}B(|Du|)\,dxdt\right],
    \qquad
    \|w\|_{L^\infty(Q')}\leq C.
\]
Letting $Q'\nearrow Q$ gives the asserted
energy and boundedness properties.

We now show that $w=u$ on
$\partial_{\rm par}Q$.
By the properties of convolution,
\[
    u_\delta^{\rm bd}\to u
    \quad\text{in }
    L^{i_B}(\mathcal T;W^{1,i_B}(\B)).
\]
On the other hand, since
$u_{\e,\delta}-u_\delta^{\rm bd}
\in L^{i_B}(\mathcal T;W^{1,i_B}_0(\B))$,
the weak closedness of
$L^{i_B}(\mathcal T;W^{1,i_B}_0(\B))$
implies
\[
    w-u\in
    L^{i_B}(\mathcal T;W^{1,i_B}_0(\B)).
\]
Thus $w=u$ on the lateral boundary
of $Q$ in the sense of traces.

Again, by the time-continuity of $u$ with respect to $L^2_{\rm loc}(\Omega)$, and by the properties of space-time convolution 
give
\[
    u_\delta^{\rm bd}(\cdot,t_1)
    \to u(\cdot,t_1)
    \qquad\text{in }L^2(\B).
\]
For every
$\varphi\in C_c^\infty(\B\times[t_1,t_2))$,
the weak formulation of \eqref{eq:homreg} then gives
\[
\begin{split}
    &-\iint_Q u_{\e,\delta}\,
    \partial_t\varphi\,dxdt
    +\iint_Q
    \A_\e(x,t,Du_{\e,\delta})
    \cdot D\varphi\,dxdt\\
    &\quad=\iint_Q f_\delta\varphi\,dxdt
    +\int_{\B}
    u_\delta^{\rm bd}(x,t_1)
    \varphi(x,t_1)\,dx.
\end{split}
\]
By \eqref{pert:compactness}, Lemma \ref{lem:approx:A} and $f_\delta\to f$ in $L^d(Q)$,
passing to the limit we conclude that
\[
\begin{split}
    &-\iint_Q w\,\partial_t\varphi\,dxdt
    +\iint_Q\A(x,t,Dw)\cdot D\varphi\,dxdt\\
    &\quad=\iint_Qf\varphi\,dxdt
    +\int_{\B}u(x,t_1)\varphi(x,t_1)\,dx.
\end{split}
\]
In particular,
\[
    \partial_t w-\mathrm{div}(\A(x,t,Dw))=f
    \qquad\text{in }Q.
\]
Then, by Remark \ref{rem:energy}
\[
    \partial_t(w-u)
    =
    \operatorname{div}\big(
    \A(x,t,Dw)-\A(x,t,Du)\big),
    \qquad
    \A(x,t,Dw)-\A(x,t,Du)
    \in L^{\widetilde B}(Q;\R^N),
\]
so standard time-regularization argument in the Orlicz
energy spaces (see, e.g., \cite{ElmahiMeskine2005, AhmedOubeidBenkiraneSidiElVally2014}) and the  time continuity of $u$ w.r.t. the $L^2_{\rm loc}(\Omega)$ topology give
\[
    w\in C([t_1,t_2];L^2(\B)).
\]
Moreover, the  weak formulation identifies its initial trace $w(\cdot,t_1)=u(\cdot,t_1)$ in $L^2(\B)$.
 Therefore,  by uniqueness of solutions,
we deduce that $w=u$ in $Q$.
Consequently, \eqref{pert:compactness} becomes
\[
    u_{\e,\delta}\to u
    \quad\text{and}\quad
    Du_{\e,\delta}\to Du
    \qquad\text{in }
    C^{\theta,\theta/2}_{\rm loc}(Q)
\]
along the chosen subsequence.
Thereby passing to the limit in
\eqref{pert:approx-gradient}-\eqref{pert:approx-gluing}, and using
\eqref{ae:unif}, \eqref{be:unif}
and \eqref{BE:unif}, we recover
\eqref{pert:gradient} and \eqref{pert:holder}
for the original solution $u$.
This completes the proof of Theorem \ref{main:thm}.
\end{proof}

\section*{Statements and Declarations}

 \par\noindent {\bf Acknowledgments.} 
  The author is supported by the Italian Ministry of University and
Research (MUR) through the FIS 2 project "SiGmA - Singularities in Geometric Analysis: Minimal Surfaces
and Mean Curvature Flows", project code FIS-2023-02962, CUP G53C25000120001.

\bigskip
\par\noindent
\textbf{Declaration of AI assistance.}
The author used ChatGPT to assist with LaTeX typesetting,
language editing, and consistency checks.
All mathematical ideas, arguments, and results are the author's own.
The author takes full
responsibility for the content of the manuscript.

\bigskip{}{}

 \par\noindent {\bf Data availability statement.} Data sharing not applicable to this article as no datasets were generated or analyzed during the current study. 

\section*{Compliance with Ethical Standards}\label{conflicts}

\par\noindent
{\bf Funding}. This research was partly funded by GNAMPA   of the Italian INdAM - National Institute of High Mathematics (grant number not available).

\bigskip
\par\noindent
{\bf Conflict of Interest}. The author declares that there is no conflict of interest.

\printbibliography

\end{document}